\documentclass[twoside]{article}

\usepackage{amsmath,amssymb}
\usepackage[ngerman,english]{babel}
\usepackage{mathptmx}
\usepackage{graphics}
\usepackage{graphicx}
\usepackage{epstopdf}
\usepackage{grffile}
\usepackage{float}
\usepackage{floatflt}
\usepackage{epsfig, epsf, rotating, fancyhdr, afterpage, multicol}
\usepackage[hidelinks]{hyperref}
\usepackage{multirow}
\usepackage{listings}
\usepackage{booktabs}
\usepackage{cite}
\usepackage[utf8]{inputenc}
\usepackage[T1]{fontenc}
\usepackage{mathtools}

\usepackage{pdfpages}

\usepackage{xcolor}
\definecolor{darkblue}{rgb}{0,0,0.5}
\usepackage{transparent}
\usepackage{cite}
\usepackage{amsmath,amssymb,amsbsy,amsfonts,amsthm, latexsym,amsopn,amstext,amsxtra,euscript,amscd}
\usepackage{anysize}
\usepackage[margin=12pt,labelfont=bf, format=hang, justification =raggedright]{subcaption}
\usepackage[margin=12pt,labelfont=bf, labelsep=endash, format=hang, justification =raggedright]{caption}
\usepackage{titlesec}
\titleformat*{\section}{\large\bfseries}
\titleformat*{\subsection}{\normalsize\bfseries}
\titleformat*{\subsubsection}{\normalsize\bfseries}
\newcommand{\prn}[1]{\left(#1\right)}

\newtheorem{remark}{Remark}
\newtheorem{lemma}{Lemma}
\newtheorem{property}{Property}
\newtheorem{definition}{Definition}
\numberwithin{equation}{section}
\numberwithin{remark}{section}
\numberwithin{definition}{section}
\numberwithin{lemma}{section}
\numberwithin{property}{section}
\renewcommand{\O}{\mathcal{O}}
\newcommand{\sgn}{\operatorname{sgn}}

\begin{document}


\begin{center}
\textbf{\large Relations between WENO3 and Third-order Limiting in Finite Volume Methods}
\end{center}



\begin{center}
B. Schmidtmann$^{a, *}$, B. Seibold$^b$, M. Torrilhon$^a$\\
{\small{${}^a$ \emph{Center for Computational Engineering Science, RWTH Aachen University, Schinkelstr. 2, 52062 Aachen, Germany}}}\\
\small{${}^b$\emph{Department of Mathematics, Temple University, 1801 North Broad Street, Philadelphia, PA 19122, USA}} \\
{\ } \\
The final publication is available at link.springer.com 
\end{center}

\bigskip

	\begin{abstract}
		\noindent Weighted essentially non-oscillatory (WENO) and finite volume (FV) methods employ different philosophies in their way to perform limiting. We show that a generalized view on limiter functions, which considers a two-dimensional, rather than a one-dimensional dependence on the slopes in neighboring cells, allows to write WENO3 and $3^\text{rd}$-order FV schemes in the same fashion. Within 
this framework, it becomes apparent that the classical approach of FV limiters to only consider ratios of the slopes in neighboring cells, is overly restrictive. The hope of this new perspective is to establish new connections between WENO3 and FV limiter functions, which may give rise to improvements for the limiting behavior in both approaches.
	\end{abstract}
\section{Introduction}\label{sec:introduction}
%
In the finite volume (FV) approach for the numerical solution of hyperbolic conservation laws, flux evaluations at the cell interfaces are needed. Since the exact values of the fluxes are unknown at the interfaces, the true flux function is approximated by a numerical flux function. Following Godunov \cite{Godunov1959}, this numerical flux function is typically evaluated by solving a local Riemann problem at each interface, taking as input the left- and right-sided limit of a reconstruction function. If this reconstruction returns simply the adjacent cell mean values, this leads to a first order scheme. In order to obtain higher order schemes one can define higher-order reconstruction functions. For instance, one can define various types of reconstructions based on only three input values, namely the mean values of the cell of interest and its direct neighbors. Evaluating these polynomials at the cell interfaces yields input values for the numerical flux function which result in a higher-order accurate scheme. The best order of accuracy which can be obtained with only three input values is a quadratic polynomial, resulting in a $3^\text{rd}$-order-accurate linear scheme. 
Godunov's Theorem \cite{Godunov1959} implies that on fixed grids, the high-order approximation of hyperbolic conservation laws with linear schemes is impossible without creating spurious oscillations. One way to work around this is to limit the order of reconstruction at discontinuous parts of the solution. A large variety of these non-linear limiters has been developed to achieve high-order accuracy without creating oscillations \cite{HartenEngquistOsherChakravarthy1987, LiuOsherChan1994, JiangShu1996, ArtebrantSchroll2005}. 

The first limiter functions were van Leer's limiter \cite{VanLeer1974} and Roe's superbee limiter \cite{Roe1985}. These limiter functions yield second order accurate reconstructions in smooth parts of the solution and lie in the second order total variation diminishing (TVD) region of Sweby \cite{LeVeque1992}. However, due to the TVD property they reduce to $1^\text{st}$-order at smooth extrema \cite{LeVeque1992}. These classical approaches based on three-cell reconstructions, result in $2^{\text{nd}}$-order accuracy \cite{VanLeer1979, LeVeque2002}. For smooth solutions however, it is in fact possible to obtain update rules which yield $3^{\text{rd}}$-order accuracy. As described above, $3^{\text{rd}}$-order accuracy is the optimum that can be achieved with only using immediate neighborhood information and if only cell averages are stored. Higher orders of accuracy can be achieved only via wider stencils, or by storing additional information per cell, as employed for instance in discontinuous Galerkin methods \cite{ReedHill1973, CockburnShu1988} or jet schemes \cite{ChidyagwaiNaveRosalesSeibold2012}. \\
Woodward and Colella \cite{ColellaWoodward1984} proposed a $3^\text{rd}$-order reconstruction based on a four-point centered stencil. They compute the interface value, and then limit the reconstruction at shocks. Suresh and Huynh \cite{SureshHuynh1997} go beyond this and use a five-point stencil for higher-order reconstruction. The interface values are obtained by limiting a higher-order polynomial reconstruction, however, their limiting procedure is costly to apply. Another possibility is the use of non-polynomial reconstructions, cf. Marquina \cite{Marquina1994}, who introduced hyperbolic reconstruction schemes, and Harten, Engquist, Osher, and Chakravarthy \cite{HartenEngquistOsherChakravarthy1987}, who present essentially non-oscillatory (ENO) schemes. The idea behind ENO schemes is to divide the stencil of interest in smaller sub-stencils with piecewise polynomial reconstructions on each sub-stencil. The scheme then applies an optimal stencil selection procedure, which chooses the locally smoothest stencil for the reconstruction. The scheme avoids interpolation across discontinuities, however, the final approximation does not contain all available data. To overcome this drawback, weighted essentially non-oscillatory (WENO) schemes were developed by Liu, Osher, and Chan \cite{LiuOsherChan1994}. Weighted ENO schemes use a convex combination of all candidate stencils to obtain the reconstruction of the interface values. Later, Jiang and Shu \cite{JiangShu1996} further modified and improved the WENO scheme by proposing a new way of measuring the smoothness of a numerical solution and thereby increasing the order of accuracy. 

Artebrant and Schroll \cite{ArtebrantSchroll2005} developed a local double-logarithmic reconstruction which only needs three input values. Based on their work, {\v{C}}ada and Torrilhon \cite{CadaTorrilhon2009} introduced a limiter function, which yields $3^\text{rd}$ order accuracy in smooth parts of the solution and at extrema. The authors also introduce a criterion to distinguish between smooth extrema and discontinuities in order to avoid extrema clipping. This work has been further investigated in \cite{SchmidtmannHYP2014},in order to find a parameter-free smoothness criterion. 

This paper continues their work on determining a parameter-free smoothness criterion and incorporating it in the reconstruction procedure. We present a new $3^\text{rd}$-order limiter function in the FV setting and then compare it to different $3^\text{rd}$-order WENO schemes with the goal that this contributes to a better understanding of the nature of non-linear schemes. 

The time discretization of all schemes is implemented using the $3^\text{rd}$-order TVD Runge-Kutta method developed by Shu and Osher \cite{ShuOsher1988}.
	\newline	\newline
	The paper is structured as follows. In Section \ref{sec:basicFormulation} we review the basic formulation of the problems in question and interface reconstruction in general. In Section \ref{sec:LimO3}, we discuss FV reconstruction methods containing limiter functions and we provide a new smoothness indicator. Section \ref{sec:WENO} recalls the WENO scheme as first introduced by Liu et al. \cite{LiuOsherChan1994} and later improved by Jiang and Shu \cite{JiangShu1996}. We focus on WENO schemes which only need three points for each reconstruction, as from now called WENO3. Section \ref{sec:unifying} places the introduced methods into a unifying setting and compares the novel limiter function with different WENO schemes. Finally, Section \ref{sec:results} presents some numerical results which demonstrate the potential of the proposed schemes and conclusions are given in Section \ref{sec:conclusions}.
%
\section{Basic Formulation}\label{sec:basicFormulation}
	We are interested in the numerical approximation of a Cauchy Problem of the hyperbolic conservation law
	\begin{subequations}
		\label{eq:conservation_law}
		\begin{align}
		\label{eq:conservation_law_main}
		\partial_t u(x,t)+\partial_x f(u(x,t)) &= 0, \\
		u(x,0)&=u_0(x),\;\; x\in \mathbb{R}
		\end{align}
	\end{subequations}	
	in one space dimension, equipped with the initial condition $u_0(x)$, where $u  = (u_1, \ldots, u_s )^T$  and the Jacobian matrix $A(u) =\frac{d f}{d u}$ has $s$ real eigenvalues. For the sake of simplicity, we restrict our discussion and analysis to the scalar case $s = 1$. However, the developed ideas are also applicable to systems ($s>1$), in the same way other limiters extend from $s=1$ to $s>1$. We consider a regular grid in space, with the positions of the cell centers denoted by $x_i,\;i\in\mathbb{Z}$ and with uniform space intervals of size $ \Delta x$. The grid cells are defined by $C_i = [x_{i-1/2},\; x_{i+1/2}]$, where $x_{i\pm j} = x_i \pm j\Delta x$.
Finite volume (FV) methods aim at approximating the cell averages of the true solution of \eqref{eq:conservation_law} with high accuracy, see e.g. \cite{LeVeque2002}. The cell average of the true solution $u(\cdot, \cdot)$ is given by 
	\begin{align}
		\label{eq:exSol}
		\bar U(x, t) = \frac{1}{\Delta x} \int_{x-\Delta x/2}^{x+\Delta x/2} u(s,t) ds.
	\end{align}
	With this definition we call $\bar{U}_i^n = \bar U(x_i, t^n)$ the cell average of the true solution $u$ in cell $C_i$ at time $t^n$. The goal is to find an update rule to advance approximate cell averages from a given time $t^n$ to a new time $t^{n+1} = t^n+\Delta t$, such that the true cell averages are approximated with high order of accuracy. In addition, the approximate solution should not develop any (relevant) spurious oscillations.
Integrating Eq.~\eqref{eq:conservation_law_main} over the cell $C_i$ and dividing by $\Delta x$ yields
	\begin{align}
		\label{eq:approxConsLaw}
		\frac{d \bar U_i}{d t} = -\frac{1}{\Delta x} \left( f(u(x_{i + 1/2},t)) - f(u(x_{i - 1/2},t))\right)
	\end{align}
	which is still exact. We now want to find a solution approximation $\bar u_i^n$ satisfying $\bar u_i^n \approx \bar U_i^n$. The quality of the approximation $\bar{u}_i$ depends on the accurate approximation of the fluxes at the cell boundaries $f(u(x_{i \pm 1/2},t))$.
This is achieved by constructing a numerical flux function $\hat{f}(u,v)$ that is Lipschitz continuous and consistent with the (true) flux function, i.e. $\hat{f}(u,u)=f(u)$. The numerical fluxes at the boundaries of cell $C_i$ are then given by 
	\begin{subequations}
		\label{eq:approxFlux}
		\begin{align}
			\hat{f}_{i+1/2}=\hat{f}(\hat u^{(-)}_{i+1/2},\hat u^{(+)}_{i+1/2}),\\
			\hat{f}_{i-1/2}=\hat{f}(\hat u^{(-)}_{i-1/2},\hat u^{(+)}_{i-1/2}),
		\end{align}
	\end{subequations}	
 	where $\hat u^{(-)}_{i\pm 1/2}$ and $\hat u^{(+)}_{i\pm 1/2}$ are approximations to the solution values at the cell boundary $x_{i\pm 1/2}$, left-sided (${}^{(-)}$) and right-sided (${}^{(+)}$), respectively, see Fig.~\ref{fig:reconstruction}. The evolution of cell averages is thus given by
	\begin{align}
		\label{eq:approxConsLawApprox}
		\frac{d \bar u_i}{d t} = -\frac{1}{\Delta x} \left( \hat{f}_{i+1/2} - \hat{f}_{i-1/2}\right).
	\end{align}
 	 The accurate reconstruction of these left and right interface values at the cell boundaries $x_{i\pm 1/2}$, see Fig.~\ref{fig:reconstruction}, is the crucial point in this process. We are particularly interested in numerical schemes that find the approximate solution values $\hat u^{(+)}_{i-1/2}$ and $\hat u^{(-)}_{i+1/2}$ that correspond to cell $C_i$ by using only information of the cell $C_i$ and its immediate neighbor cells $C_{i-1}$ and $C_{i+1}$. The restriction to immediate neighbor cells provides local update rules. This locality is, among other advantages, beneficial for low-communication parallelization, and ensures that few ghost cells must be provided near boundaries. 
 	 
The key ingredient to getting $3^{\text{rd}}$-order accuracy is the way of reconstructing function values at cell boundaries $x_{i\pm\frac{1}{2}}$ based on cell averages. The reconstructed values $\hat u_{i\pm\frac{1}{2}}$ are then provided as input values for the numerical flux function $\hat{f}(\cdot, \cdot)$ in Eq.~\eqref{eq:approxConsLawApprox}, following the standard FV methodology \cite{LeVeque1992}.
The focus of this paper is solely on the actual reconstruction of the solution $\hat u_{i\pm\frac{1}{2}}$ at the cell interfaces. For the sake of simplicity, we shall drop the $\;\hat{}\;$.
	\begin{figure}
	\centering
		\begin{subfigure}{.49\textwidth}
			\includegraphics[width=0.9\textwidth]{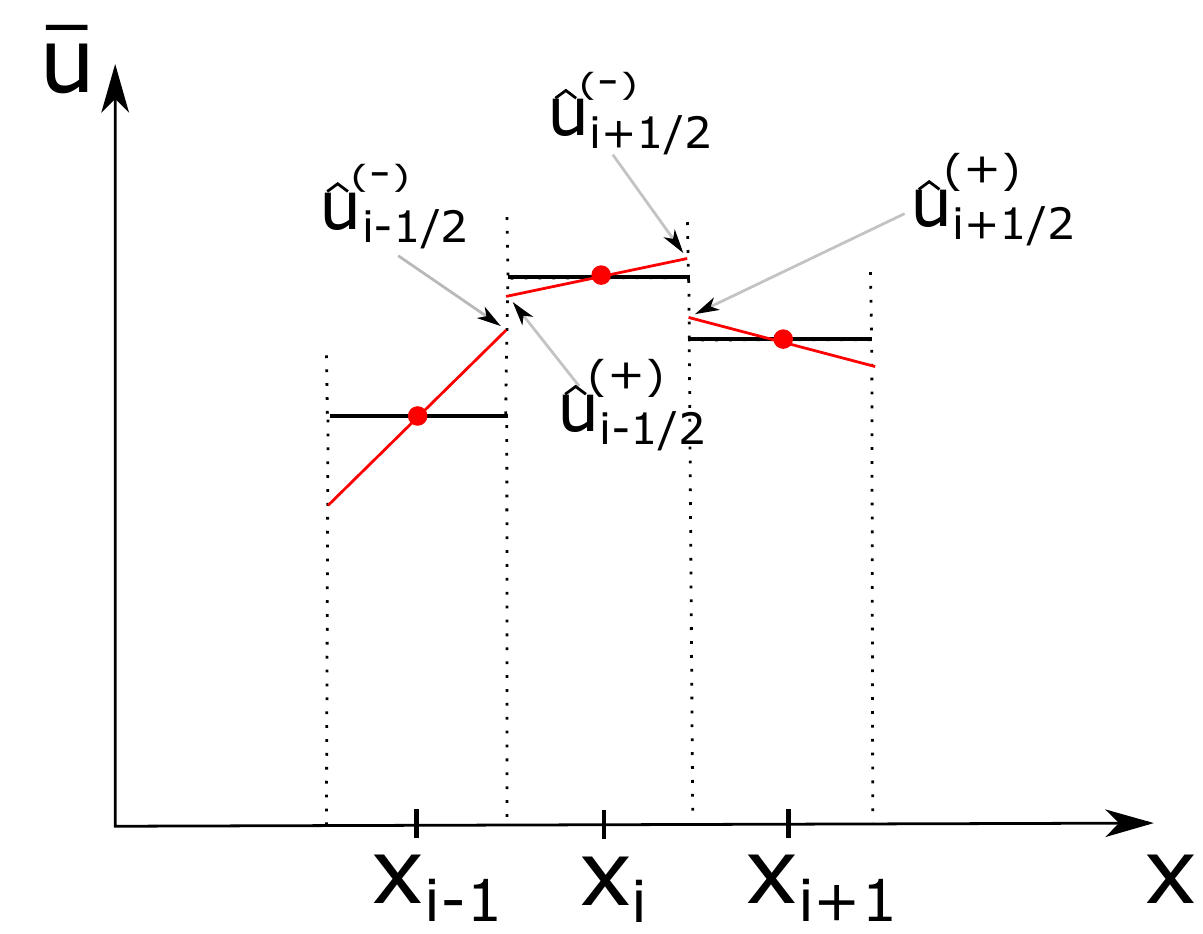}
			\caption{From the three cell averages (black horizontal lines), functions are reconstructed (red lines, a linear function in this case), these functions are then evaluated at the left- and right-sided cell boundaries.}
		\end{subfigure}
		\begin{subfigure}{.49\textwidth}
			\includegraphics[width=0.9\textwidth]{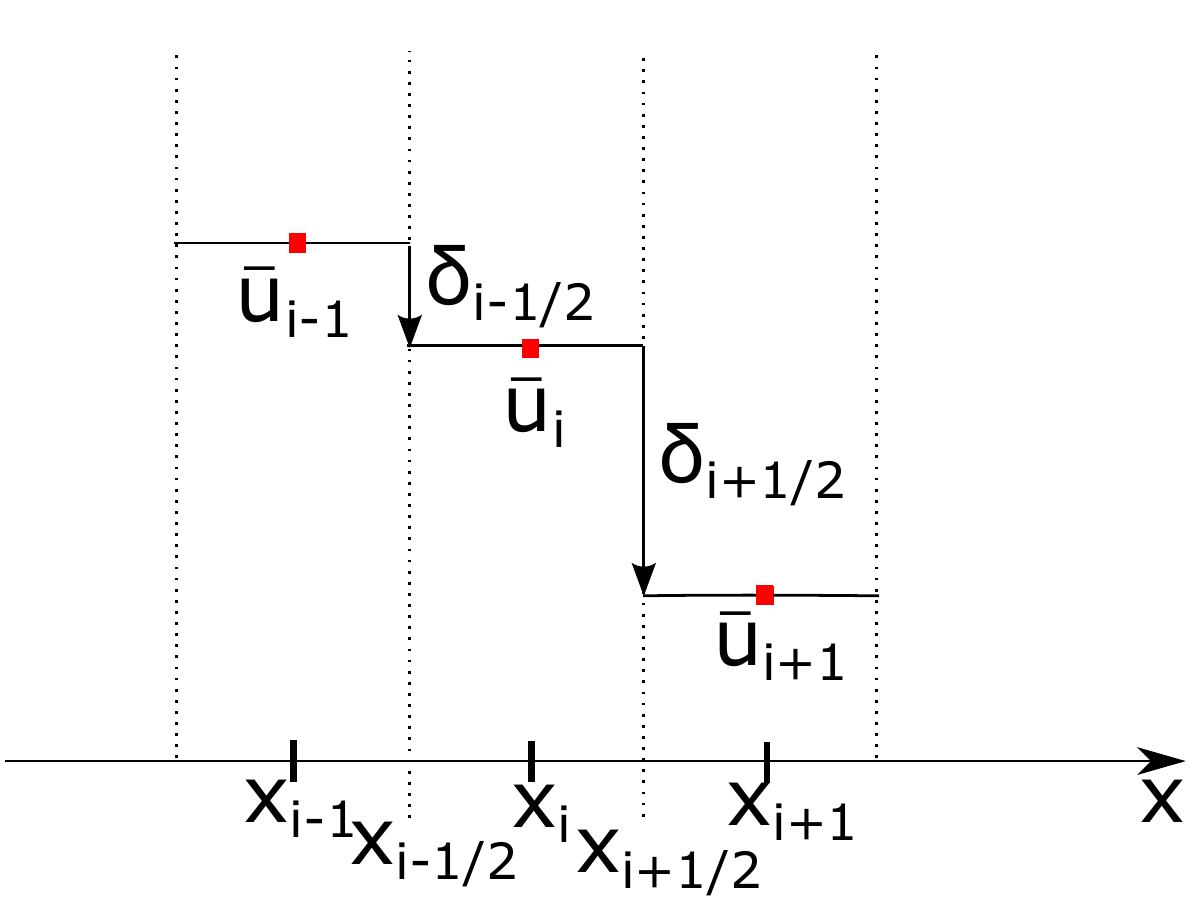}
			\caption{The normalized gradients are computed as the differences between neighboring cell averages, $\delta_{i+\frac{1}{2}}=\bar u_{i+1}-\bar u_i$ and $\delta_{i-\frac{1}{2}}=\bar u_i-\bar u_{i-1}$, respectively.}
		\end{subfigure}
		\caption{Basic setting for the reconstruction of the interface values $u^{(\pm)}(x_{i\pm 1/2})$ on a 3-point-stencil.}
	\label{fig:reconstruction}
	\end{figure}
	
Since we consider three cells for the reconstruction, we define the left and right interface values to be determined by functions $L$ and $R$:
	\begin{subequations}
		\label{eq:general_interface_values}
		\begin{align}			
			u_{i+\frac{1}{2}}^{(-)} &= L(\bar u_{i-1},\bar u_{i},\bar u_{i+1})\\
			u_{i-\frac{1}{2}}^{(+)} &= R(\bar u_{i-1},\bar u_{i},\bar u_{i+1}).
		\end{align}
	\end{subequations}	
	%
	%
\section{Finite Volume $3^\text{rd}$ order Limiting}\label{sec:LimO3}
%
In this section, we review limiting in FV methods using the three-point stencil for the reconstruction of $\hat u^{(+)}_{i\pm 1/2}$ and $\hat u^{(-)}_{i\pm 1/2}$. For the computation of the numerical flux functions, the approximate solution at the cell interfaces $x_{i\pm 1/2}$ is needed. For the cell $x_i$, we define the left and right interface values by
	\begin{subequations}
    	\label{eq:reconstructionNonStandard}
		\begin{align}
	 	  	u^{(-)}_{i+1/2} &= \bar u_i + \frac{1}{2}\; \phi(\theta_i)\delta_{i+\frac{1}{2}}, \\
  		  	u^{(+)}_{i-1/2} &= \bar u_i - \frac{1}{2}\; \phi(\theta^{-1}_i)\delta_{i-\frac{1}{2}},
		\end{align}
 	\end{subequations}
	respectively.
	Here, $\phi$ is a univariate, non-linear limiter function depending on the local smoothness measure
	\begin{subequations}
		\begin{align}
			\label{eq:theta}
			&\theta_i = \frac{\delta_{i-\frac{1}{2}}}{\delta_{i+\frac{1}{2}}},\\
			\text{where}\;\;&\delta_{i+\frac{1}{2}}=\bar u_{i+1}-\bar u_i,\label{eq:deltaP}\\
			\text{and}\;\;&\delta_{i-\frac{1}{2}} =\bar u_i-\bar u_{i-1}\label{eq:deltaM}
		\end{align}
	\end{subequations}	
	are the differences between neighboring cell averages, cf. Fig.~\ref{fig:reconstruction}. Limiter functions switch the reconstruction to high-order accuracy in smooth parts of the solution and to lower-order reconstructions near discontinuities \cite{LeVeque1992}.
This means, limiter functions contain some sort of smoothness measure. In Eq.~\eqref{eq:reconstructionNonStandard}, the choice of $\phi(\theta_i)$ determines the order of accuracy of the reconstruction and therefore of the resulting scheme.
	\begin{property}\label{prop:generalProperties} \textbf{General Properties of Limiter Functions}\\
		(i) If $\phi$ passes continuously through $\theta = 1$ with $\phi(1) = 1$, then the resulting scheme is at least $2^\text{nd}$ order accurate in sufficiently smooth, monotonous regions of the solution.\\
		(ii) If $\phi$ satisfies the conditions
		\begin{equation}
			\label{eq:strict_tvd_bounds}
			0\le\phi(\theta)\le \max(0,\min(2,2\theta))\;,
		\end{equation}
		the numerical scheme is total variation diminishing (TVD), and thus does not create spurious oscillations.\\
	\end{property}	
	\begin{proof}
		See e.g. \cite{LeVeque2002}.
	\end{proof}
	There is a variety of schemes which reconstruct based on three points and obtain $2^{\text{nd}}$-order accuracy. These are the classical schemes in the Sweby setting, such as Superbee, van Leer and others \cite{LeVeque1992}. These schemes use the information of the three cells to compute a linear reconstruction function, see e.g. \cite{VanLeer1979}. Indeed, the full $2^\text{nd}$-order reconstruction $u^{(-)}_{i+1/2}=\bar u_i + \frac{\Delta x}{2}	\left(\frac{\bar u_{i+1}-\bar u_{i-1}}{2\Delta x} \right)$ can be rewritten in form of Eq.~\eqref{eq:reconstructionNonStandard} with the limiter function $\phi (\theta) = \frac{1+\theta}{2}$. One can easily check that this limiter function satisfies the property $\phi(\theta^{-1})=\theta^{-1}\phi(\theta)$ and therefore, for this limiter, Eq.~\eqref{eq:reconstructionNonStandard} can be rewritten in the formulation	
	\begin{subequations}
    	\label{eq:reconstruction}
		\begin{align}
	 	  	u^{(-)}_{i+1/2} &= \bar u_i + \frac{\Delta x}{2}\; \sigma_i \\
  		  	u^{(+)}_{i-1/2} &= \bar u_i - \frac{\Delta x}{2}\; \sigma_i,
		\end{align}
  	\end{subequations}
	with the right-sided slope $\sigma_i = \phi (\theta_i)\delta_{i+\frac{1}{2}} / \Delta x$, see e.g. \cite{LeVeque2002}. Even though this formulation is widely used, we will continue with the formulations \eqref{eq:general_interface_values} and \eqref{eq:reconstructionNonStandard}, as introduced in \cite{Dubois1990, CadaTorrilhon2009}.	
	The aim of this work is to discuss schemes which use the three-point stencil to achieve $3^{\text{rd}}$ order accurate reconstructions of the cell-interface values. One possibility is to construct a quadratic polynomial $p_i(x)$ in each cell $C_i$. Applying the computed polynomial to $x_{i\pm 1/2}$, we obtain the interface values
	\begin{subequations}
		\begin{align}
			\label{eq:thirdOrderPoly}
			u^{(-)}_{i+1/2} = p_i(x_{i+1/2})=\frac{1}{3}u_{i+1}+\frac{5}{6}u_{i}-\frac{1}{6}u_{i-1}\\
			u^{(+)}_{i-1/2} = p_i(x_{i-1/2})=\frac{1}{3}u_{i-1}+\frac{5}{6}u_{i}-\frac{1}{6}u_{i+1}.
		\end{align}
	\end{subequations}
It turns out that these expressions can be written in the form \eqref{eq:reconstructionNonStandard} to obtain the non-limited $3^{\text{rd}}$-order reconstruction
	\begin{align}
		\label{eq:3rdOrder}
		\phi_{3}(\theta_i) := \frac{2+\theta_i}{3}
	\end{align}
	with $\theta_i$ given by Eq.~\eqref{eq:theta}. This formulation results in a full-$3^\text{rd}$-order-accurate scheme for smooth solutions, however, causes oscillations near discontinuities. This can be seen either by noticing that $\phi_{3}(\theta_i)$ does not lie in the TVD region, see Fig.~\ref{fig:ASandCTlimiter}, or as a direct consequence of Godunov's Theorem \cite{Godunov1959}, since a linear scheme of more than $1^\text{st}$ order cannot be monotone. 
Since oscillations should be avoided, limiter functions have been introduced, which apply the full $3^{\text{rd}}$-order reconstruction \eqref{eq:3rdOrder} at smooth parts of the solution and switch to a lower-order reconstruction close to large gradients, shocks, and discontinuities. In \cite{ArtebrantSchroll2005}, Artebrant and Schroll present a limiter function, which can be formulated as $\phi_{\text{AS}}(\theta_i, q)$, see \cite{CadaTorrilhon2009}, based on a local-double-logarithmic reconstruction. This function does not solely depend on $\theta_i$ but also contains a parameter $q$ which significantly changes the reconstruction function. The authors recommend $q=1.4$ and demonstrate that for $q\to 0$, the logarithmic limiter function reduces to $\phi_{3}(\theta_i)$. Their limiter function reads
	\begin{align*}
		&\phi_{\text{AS}}(\theta_i, q) =\frac{2 p[(p^2-2 p\,\theta_i+1)\log(p)-(1-\theta_i)(p^2-1)]}{(p^2-1)(p-1)^2},\\
		&p=p(\theta_i, q)=2\frac{|\theta_i|^q}{1+|\theta_i|^{2q}}.
	\end{align*}
	The downside of $\phi_{\text{AS}}(\theta_i, q)$ is its complexity, which renders the evaluation in each cell expensive. 
{\v{C}}ada and Torrilhon \cite{CadaTorrilhon2009} derive, in an ad-hoc fashion, a limiter function $\phi_{\text{CT}}(\theta_i)$ which is based on $\phi_{\text{AS}}$. This function overcomes the drawbacks by approximating the properties of $\phi_{\text{AS}}$ and reducing the computational cost. It reads
	\begin{align}
		\label{eq:limiter_extended}		
		\phi_{\text{CT}}(\theta_i)=\max\left(0,\min\left( \phi_{3}(\theta_i), \max\left(-\frac{1}{2}\theta_i,\min\left(2 \,\theta_i,\phi_{3}(\theta_i),1.6 \right)\right)\right)\right)
	\end{align}
	and is shown in Fig.~\ref{fig:ASandCTlimiter} together with $\phi_{\text{AS}}(\theta_i, 1.4)$ and $\phi_{3}(\theta_i)$.
	\begin{figure}
		\centering
		\includegraphics[scale=0.5]{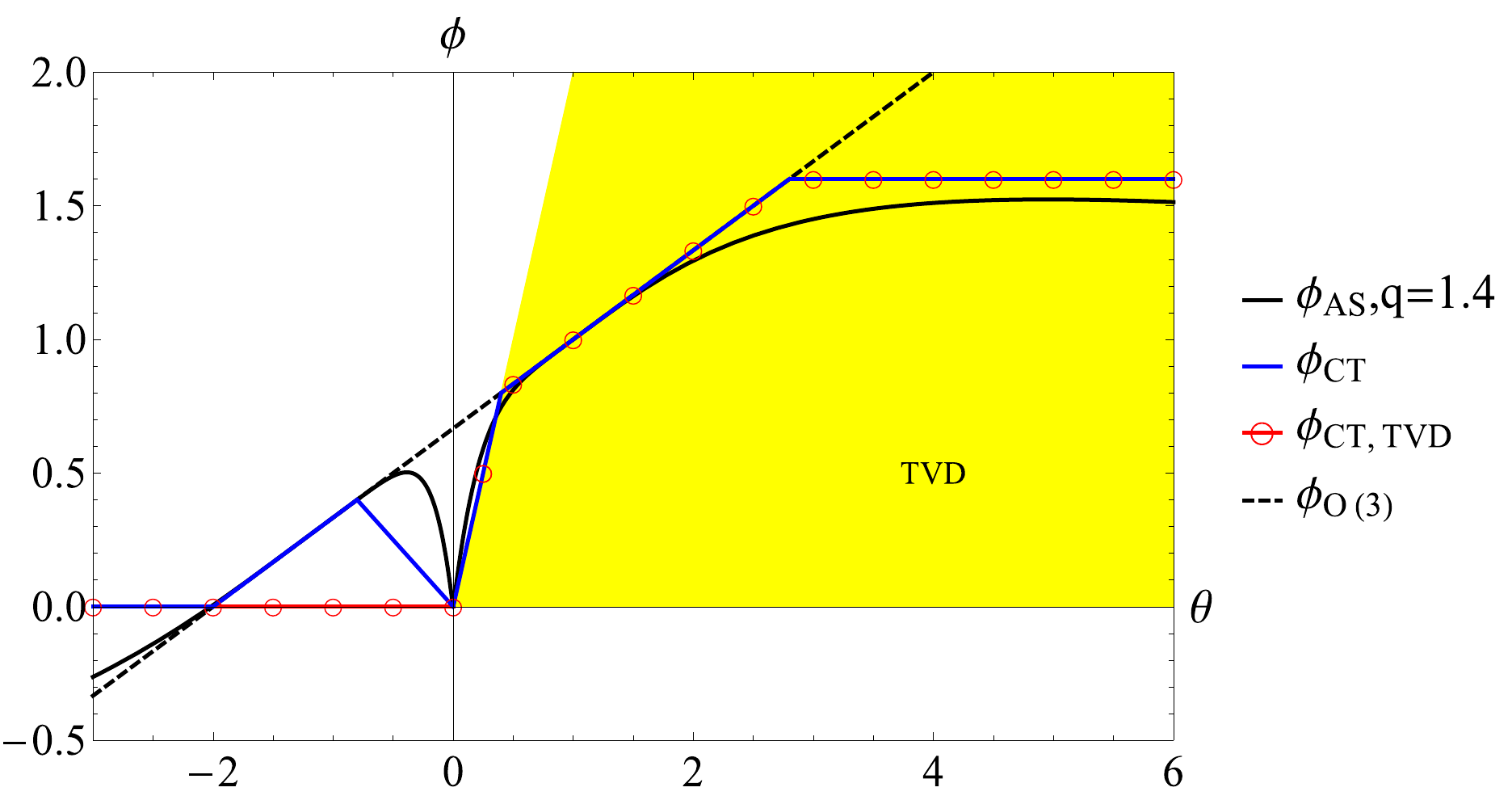}
		\caption{Logarithmic limiter function $\phi_{\text{AS}}$ with $q=1.4$ (black solid line), its approximation $\phi_{\text{CT}}$ (blue solid line), the TVD-version of $\phi_{\text{CT}}$ which does not have the non-zero part for $\theta<0$, denoted by $\phi_{\text{CT, TVD}}$ (red solid line with red circles) and the full-$3^\text{rd}$-order reconstruction $\phi_{\text{3}}$ (black dashed line).}
		\label{fig:ASandCTlimiter}
	\end{figure}
	Note that $\phi_{\text{CT}}$ does not lie within the strict TVD bounds, i.e. it breaks with \ref{prop:generalProperties} (ii).
The TVD property \ref{prop:generalProperties} (ii) can be achieved by considering only those parts of $\phi_{\text{CT}}$ where $\theta \geq 0$, and setting it to $0$ elsewhere, leading to the limiter function
	\begin{equation}
		\label{eq:limiter_tvd}
		\phi_\text{CT, TVD}(\theta) = \max\prn{0,\min\prn{2\theta,\phi_{3}(\theta), 1.6}}\;.
	\end{equation}
	However, the motivation for keeping the non-zero part in the construction of $\phi_{\text{CT}}(\theta)$ for $\theta \in [-2, 0]$ is to avoid the so-called extrema clipping. This is the effect occurring close to minima and maxima, where the normalized slopes $\delta_{i\pm 1/2}$ are of the same order of magnitude but have opposite signs, i.e. $\theta_i \approx -1$. In this case, classical limiter functions that fully lie in the strict TVD bounds yield zero and thus generate a $1^{\text{st}}$-order accurate scheme. This undesirable reduction in accuracy is avoided when the non-zero part of $\phi$ is included.
Therefore, $\phi_{\text{CT}}$ possesses better smoothness properties near $\theta = -1$ than $\phi_{\text{CT, TVD}}$ since here, $\phi_\text{CT}(\theta) = \phi_{3}(\theta)$ but $\phi_\text{CT, TVD}(\theta)=0$. For more details, see \cite{CadaTorrilhon2009}, where the non-zero part was first introduced.
	
	Nonetheless, it might occur that the discretization of an extremum is such that one of the consecutive slopes is approximately zero. In this case, $\theta \to 0$ or $\theta\to \pm \infty$ and the interface values $u_{i\pm 1/2}$ are again approximated by the cell mean values $\bar u_i$, which yields a $1^\text{st}$-order scheme. That is, a zero-slope is interpreted as the onset of a discontinuity, even though it might in fact be the magnified view of a smooth extremum. This undesired case demonstrates that a criterion is needed which can differentiate a smooth extremum from a discontinuity or steep gradient.
In the framework considered in this paper, the criterion should only depend on the available information of the compact three-point stencil. Furthermore, it has to detect cases when switching to the $3^\text{rd}$ order reconstruction is safe, even though one of the normalized slopes is zero.
We assume that using the $3^\text{rd}$ order reconstruction is safe if the non-zero slope is 'small'. In turn, if the non-zero slope is not small, we assume to be near a discontinuity or large gradient, and the order should be reduced.
The main focus of Section~\ref{subsec:2parameters} is to determine what 'small' means and to define a suitable smoothness indicator $\eta$.
%
\subsection{Interpretation in 2D Slope Domain}\label{subsec:2parameters}
	From the discussion above, it is clear that such a switch function $\eta$ has to explicitly depend on both normalized slopes $\delta_{i\pm1/2}$, Eq.~(\ref{eq:deltaP}, \ref{eq:deltaM}). The classical approach of only considering the ratio $\theta_i$, Eq.~\eqref{eq:theta}, of neighboring slopes is overly restrictive because part of the information (the actual magnitude of the two slopes) is discarded. This is why we reformulate all limiter functions $\phi$ in a two-parameter framework and obtain the new formulation for the reconstructed interface values (see Eq.~\eqref{eq:reconstructionNonStandard}) as
	\begin{subequations}
    	\label{eq:update_finite_volume_2d}
		\begin{align}
	 	  	u^{(-)}_{i+1/2} &= \bar u_i + \tfrac{1}{2}\; H(\delta_{i-\frac{1}{2}},\delta_{i+\frac{1}{2}}), \\
  		  	u^{(+)}_{i-1/2} &= \bar u_i - \tfrac{1}{2}\; H(\delta_{i+\frac{1}{2}},\delta_{i-\frac{1}{2}})
		\end{align}
  	\end{subequations}
	with the limiter function $H$ explicitly depending on both normalized slopes. The old limiter function $\phi(\theta_i)$ can of course be rewritten in the new form of the two-parameter function $H$ by setting
 	\begin{align}
 		\label{eq:h_finite_volume}
 		H(\delta_{i-\frac{1}{2}}, \delta_{i+\frac{1}{2}}):= \phi(\theta_i)\delta_{i+\frac{1}{2}}=\phi\left(\frac{\delta_{i-\frac{1}{2}}}{\delta_{i+\frac{1}{2}}}\right)\delta_{i+\frac{1}{2}}.
	\end{align}
	In this setting, the full-$3^\text{rd}$-order reconstruction $\phi_3(\theta_i)=(2+\theta_i)/3$, given by Eq.~\eqref{eq:3rdOrder}, now reads
	\begin{align}
		\label{eq:3rdOrder2param}
		H_{3}(\delta_{i-\frac{1}{2}}, \delta_{i+\frac{1}{2}}) =\frac{2 \delta_{i+\frac{1}{2}} + \delta_{i-\frac{1}{2}}}{3}.
	\end{align}
	This formulation has the advantage that there is no division by the normalized slope $\delta_{i\pm 1/2}$. Thus, a possible division by a number close to zero is avoided. 
	\begin{figure}
		\begin{subfigure}{.49\textwidth}
			\includegraphics[width=\textwidth]{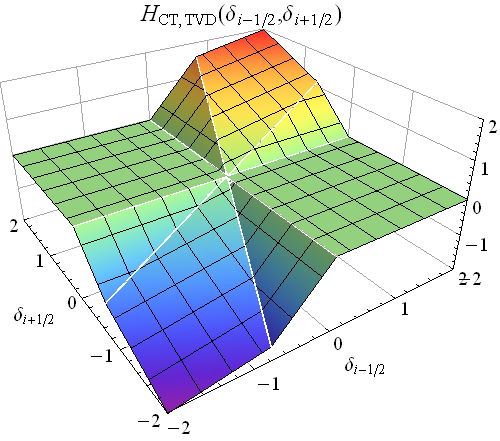}
			\caption{Limiter function $H_\text{CT, TVD}(\delta_{i-\frac{1}{2}}, \delta_{i+\frac{1}{2}})$ satisfying the strict TVD bounds.}
			\label{fig:limiter_tvd}
		\end{subfigure}
		\vspace{1em}
		\begin{subfigure}{.49\textwidth}
			\includegraphics[width=\textwidth]{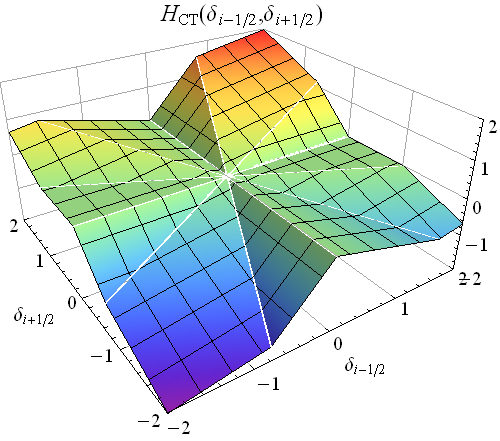}
			\caption{Extended limiter function $H_\text{CT}(\delta_{i-\frac{1}{2}}, \delta_{i+\frac{1}{2}})$.}
			\label{fig:limiter_ext}
		\end{subfigure}
		\caption{Different limiter functions in the two-parameter framework.}
	\end{figure}
	Fig.~\ref{fig:limiter_tvd} shows the limiter function $H_{\text{CT,TVD}}$ which satisfies the strict TVD bounds. This can clearly be seen by the zero parts for $\sgn(\delta_{i-\frac{1}{2}})\neq \sgn(\delta_{i+\frac{1}{2}})$. Fig.~\ref{fig:limiter_ext} shows the extended version $H_{\text{CT}}$ in the two-parameter setting. On the coordinate axis where $\delta_{i-\frac{1}{2}}=0$, i.e. $\theta_i=0$, the limiter function $H_{\text{CT}}$ returns zero, meaning that it yields a $1^{\text{st}}$-order method. The same holds for the coordinate axis where $\delta_{i+\frac{1}{2}}=0$. For two consecutive slopes of approximately the same order of magnitude, i.e. around the diagonals, $H_{\text{CT}}$ recovers the $3^{\text{rd}}$-order reconstruction $H_{3}$. This is the case for $\delta_{i-\frac{1}{2}}\approx-\delta_{i+\frac{1}{2}}$ as well as for $\delta_{i-\frac{1}{2}}\approx\delta_{i+\frac{1}{2}}$, contrary to $H_{\text{CT,TVD}}$ which returns $H_{3}$ only in the latter case and $0$ for $\delta_{i-\frac{1}{2}}\approx-\delta_{i+\frac{1}{2}}$.\\
	\begin{figure}[b]
   	\centering
   		\begin{subfigure}{0.49\textwidth}
   			\centering
   			\includegraphics[width=0.8\textwidth]{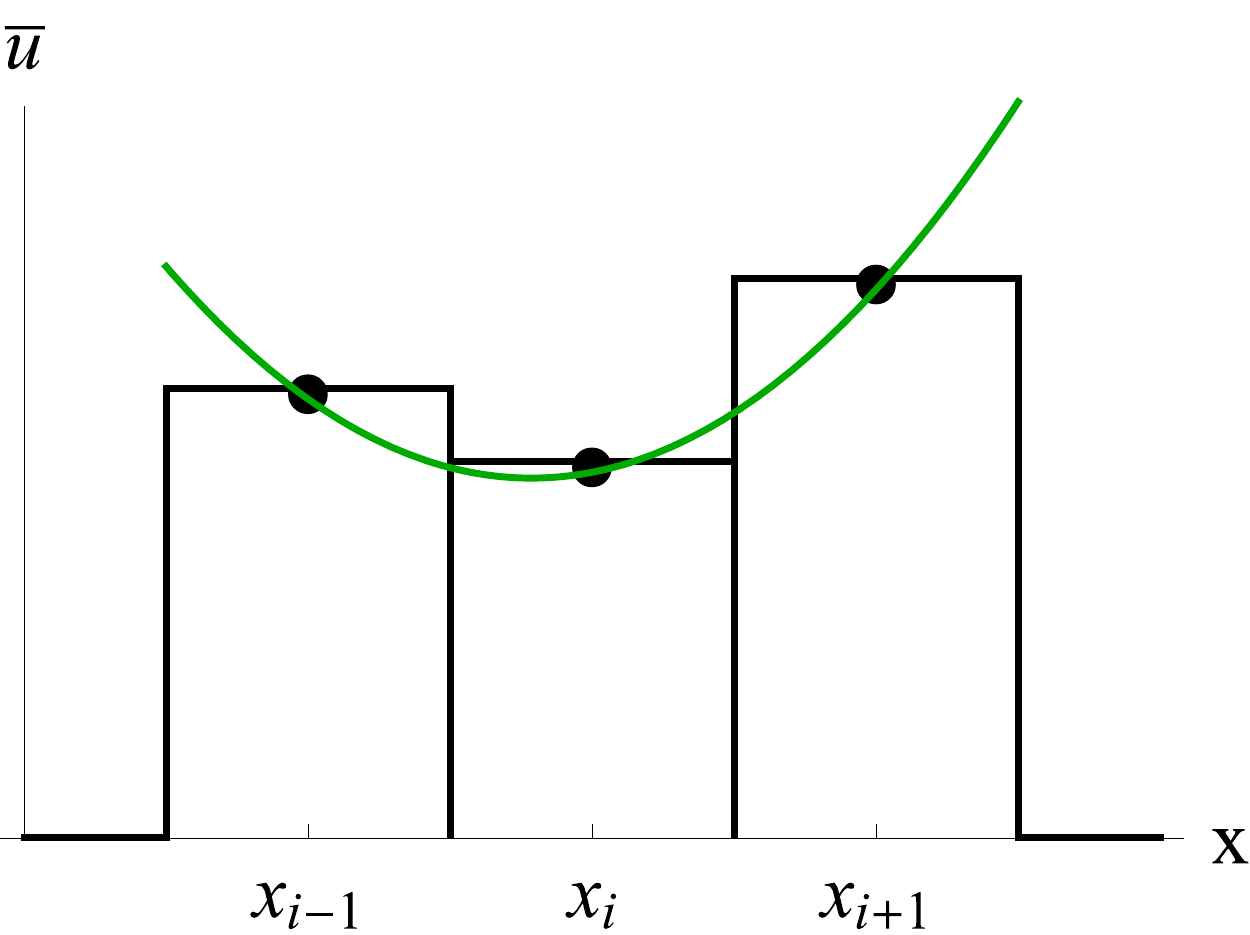}
	   		\caption{This situation is treated as a possible discontinuity: $H_{\text{CT}}\neq H_{3}$.}
   		\end{subfigure}
   		\hfill
   		\begin{subfigure}{0.49\textwidth}
   			\centering
   			\includegraphics[width=0.8\textwidth]{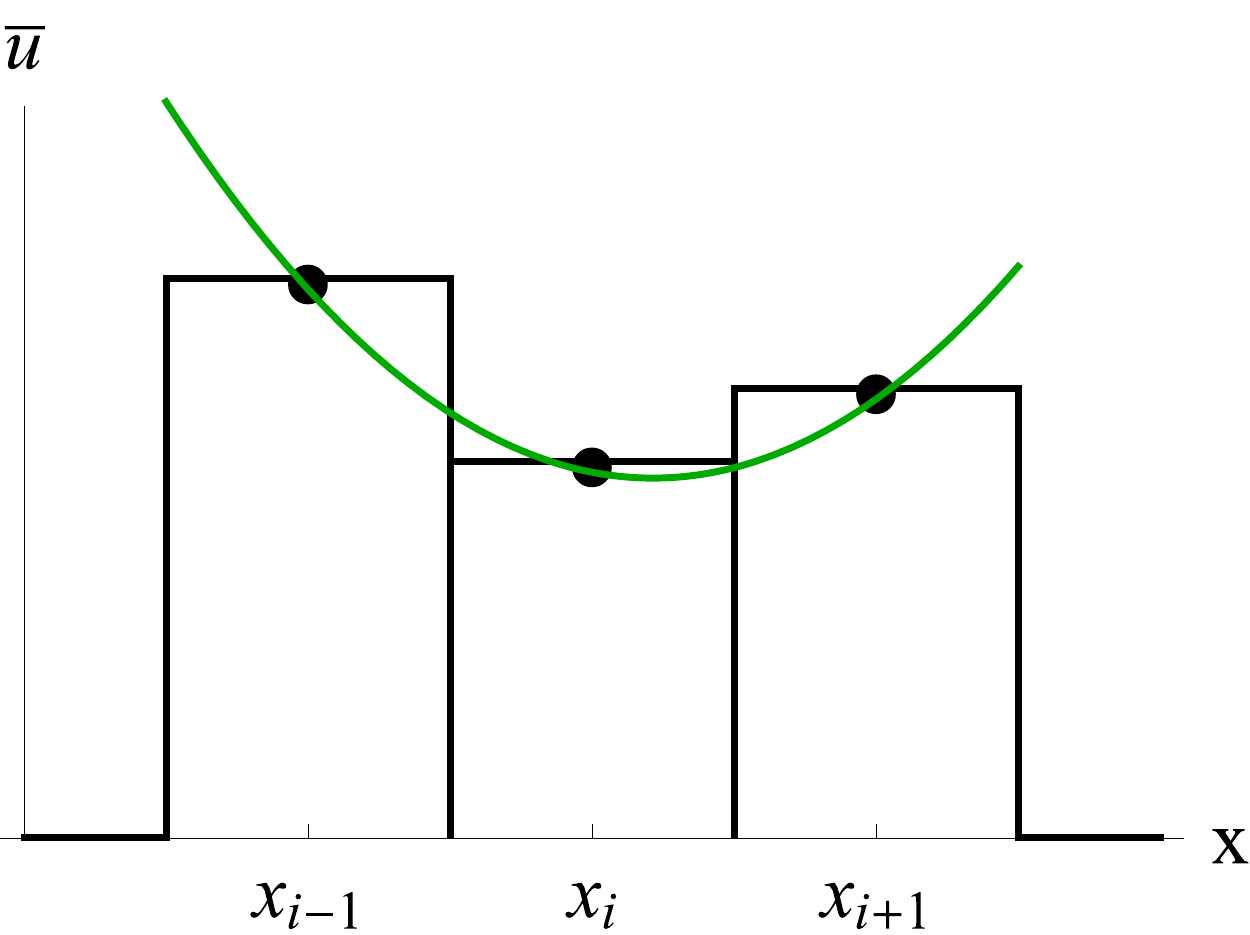}   		
	   		\caption{This situation is classified as smooth: $H_{\text{CT}}=H_{3}$.}
   		\end{subfigure}
   		\caption{Two situations which are reflections of each other, treated differently by $H_{\text{CT}}$.}
   		\label{fig:unsymmetricDeltas}
   \end{figure}
	The limiter function presented in \cite{CadaTorrilhon2009} is not symmetric with respect to the diagonals. This means that in some cases $H_{\text{CT}}(\delta_1,\delta_2) = H_{3}(\delta_1,\delta_2)$ but $H_{\text{CT}}(-\delta_2, -\delta_1) \neq H_{3}(-\delta_2, -\delta_1)$, cf. Fig.~\ref{fig:unsymmetricDeltas}.
This asymmetry is not natural because these two situations are equally smooth or non-smooth.
We therefore correct this feature by defining the following $3^{\text{rd}}$-order limiter function $H_{3\text{L}}(\delta_{i-\frac{1}{2}}, \delta_{i+\frac{1}{2}})$
	\begin{align}
		\label{eq:newFVlimiter}
		H_{3\text{L}}(\delta_{i-\frac{1}{2}},\delta_{i+\frac{1}{2}}) = 
		\,&\sgn(\delta_{i+\frac{1}{2}})\;\max(0,\min(\sgn(\delta_{i+\frac{1}{2}})\,H_{3}, \max(-\sgn(\delta_{i+\frac{1}{2}})\delta_{i-\frac{1}{2}}, \\
		&\min(2\,\sgn(\delta_{i+\frac{1}{2}})\delta_{i-\frac{1}{2}}, \sgn(\delta_{i+\frac{1}{2}})\,H_{3}, 1.5 |\delta_{i+\frac{1}{2}}|)))) \nonumber
	\end{align}
	This new limiter function treats symmetric situations in the same manner, i.e. if $H_{3\text{L}}(\delta_1,\delta_2) = H_{3}(\delta_1,\delta_2)$ then also $H_{3\text{L}}(-\delta_2, -\delta_1) = H_{3}(-\delta_2, -\delta_1)$, cf. Fig.~\ref{fig:unsymmetricDeltas}. The difference between the functions $H_{\text{CT}}$and $H_{3\text{L}}$ can also be seen in Fig.~\ref{fig:comparePhiCTandPhi3L}. \\
\\
	Expressing the interface values in the more general form \eqref{eq:general_interface_values}, we can determine some properties for $L(\cdot,\cdot,\cdot)$ and $R(\cdot,\cdot,\cdot)$, as done in \cite{Dubois1990}. These properties are valid for all limiter functions presented so far, i.e. for $H_{3},\, H_{\text{CT}},\, H_{3\text{L}}$.
	\begin{property}\label{prop:Homogeneity}\textbf{Homogeneity}\\
		Multiplying the arguments of $L$ and $R$ in Eq.~\eqref{eq:general_interface_values} by the same real number $\lambda$ multiplies the interface values $u_{i+\frac{1}{2}}^{(-)}$ and $u_{i+\frac{1}{2}}^{(+)}$, respectively, by the same constant $\lambda$. This is, $L$ and $R$ are called \textbf{homogeneous}, i.e.~linear along each line through the origin in the $(\delta_{i-\frac{1}{2}},\delta_{i+\frac{1}{2}})$ plane.
		\begin{align}
			J(\lambda u, \lambda v, \lambda w) = \lambda J(u, v, w),\quad J\in \{L,\, R\},\, \lambda\in\mathbb{R}
		\end{align}
	\end{property}
	\begin{property}\label{prop:Translation}\textbf{Translational invariance}\\
		Adding a constant $\lambda$ to the arguments of \eqref{eq:general_interface_values}
		adds the same constant $\lambda$ to
		the interface values. This is, $L$ and $R$ are called \textbf{translationally invariant}.
		\begin{align}
			J(u+\lambda, v+\lambda, w+\lambda) = J(u, v, w)+\lambda,\quad J\in \{L,\, R\},\, \lambda\in\mathbb{R}
		\end{align}
	\end{property}
	\begin{property}\label{prop:LeftRight}\textbf{Left-Right symmetry}\\
		Exchanging the first and third argument of \eqref{eq:general_interface_values} interchanges the left
		and right interface values, (cf. \cite{Dubois1990} for more details and figures)
		\begin{align}
			R(w,v,u)=L(u,v,w).
		\end{align}
	\end{property}
	\begin{lemma}\label{prop:limiterExistance}{\ }\newline
	If properties \ref{prop:Homogeneity} to \ref{prop:LeftRight} are satisfied, there exists an appropriate limiter function $\psi: \mathbb{R}\to \mathbb{R}$ such that
		\begin{align}
			L(u,v,w)&= v + \frac{1}{2}\psi\left(\frac{u-v}{w-v}\right)(w-v)\\
			R(u,v,w)&= v - \frac{1}{2}\psi\left(\frac{w-v}{u-v}\right)(v-u).
		\end{align}
	\end{lemma}
	\begin{lemma}
		\label{lemma:equivLimiters}
		With properties \ref{prop:Homogeneity} to \ref{prop:LeftRight}, it is easy to verify that
		\begin{itemize}
			\item[(i)] computing the interface values $u^{(\pm)}_{i\pm1/2}$ with $L$ and $R$, given by Eq.~\eqref{eq:general_interface_values}, or in the form of Eq.~\eqref{eq:reconstructionNonStandard}, are equivalent.
			\item[(ii)] For $\delta_{i+\frac{1}{2}}\neq 0$ the formulations of the cell interfaces in the one-parameter framework , Eq.~\eqref{eq:reconstructionNonStandard} and in the two-parameter framework, Eq.~\eqref{eq:update_finite_volume_2d}, are equivalent.
		\end{itemize}
	\end{lemma}
	\begin{proof}
		We will only show $(i)$ for $u^{(-)}_{i+1/2}$; the other cases are similar.
		\begin{itemize}
			\item[(i)] Setting $(u, v, w) = (\bar u_{i-1}, \bar u_i, \bar u_{i+1})$ yields
				\begin{align*}
					u^{(-)}_{i+1/2}\stackrel{\eqref{eq:general_interface_values}}=L(\bar u_{i-1}, \bar u_i, \bar u_{i+1})\stackrel{\eqref{prop:limiterExistance}}=\bar u_i+\frac{1}{2}\psi\left(\frac{\bar u_{i-1}-\bar u_{i}}{\bar u_{i+1}-\bar u_{i}}\right)(\bar u_{i+1}-\bar u_{i}) = \bar u_i+\frac{1}{2}\phi\left(\theta_i\right)\delta_{i+\frac{1}{2}}
				\end{align*}
				where $ \delta_{i+\frac{1}{2}}=\bar u_{i+1}-\bar u_{i}$ and $\psi\left(\bar u_{i-1}-\bar u_{i}/\bar u_{i+1}-\bar u_{i}\right) = \psi(-\theta_i)=:\phi(\theta_i)$ applying Eq.~\eqref{eq:theta}.
			\item[(ii)] Due to the homogeneity of $L$ and $R$, we can easily see with Eq.~\eqref{eq:update_finite_volume_2d} that
			\begin{align}
				H(\delta_{i-\frac{1}{2}}, \delta_{i+\frac{1}{2}}) = H(\frac{\delta_{i-\frac{1}{2}}}{\delta_{i+\frac{1}{2}}}, 1)\;\delta_{i+\frac{1}{2}} = H(\theta_i, 1)\;\delta_{i+\frac{1}{2}} = \phi(\theta_i)\;\delta_{i+\frac{1}{2}}.
			\end{align}
		\end{itemize}		
	\end{proof}
	%
\subsection{A New Smoothness Indicator}\label{subsec:smoothnessIndicator}
%
	\begin{figure}[t]
		\centering
		\includegraphics[scale=0.6]{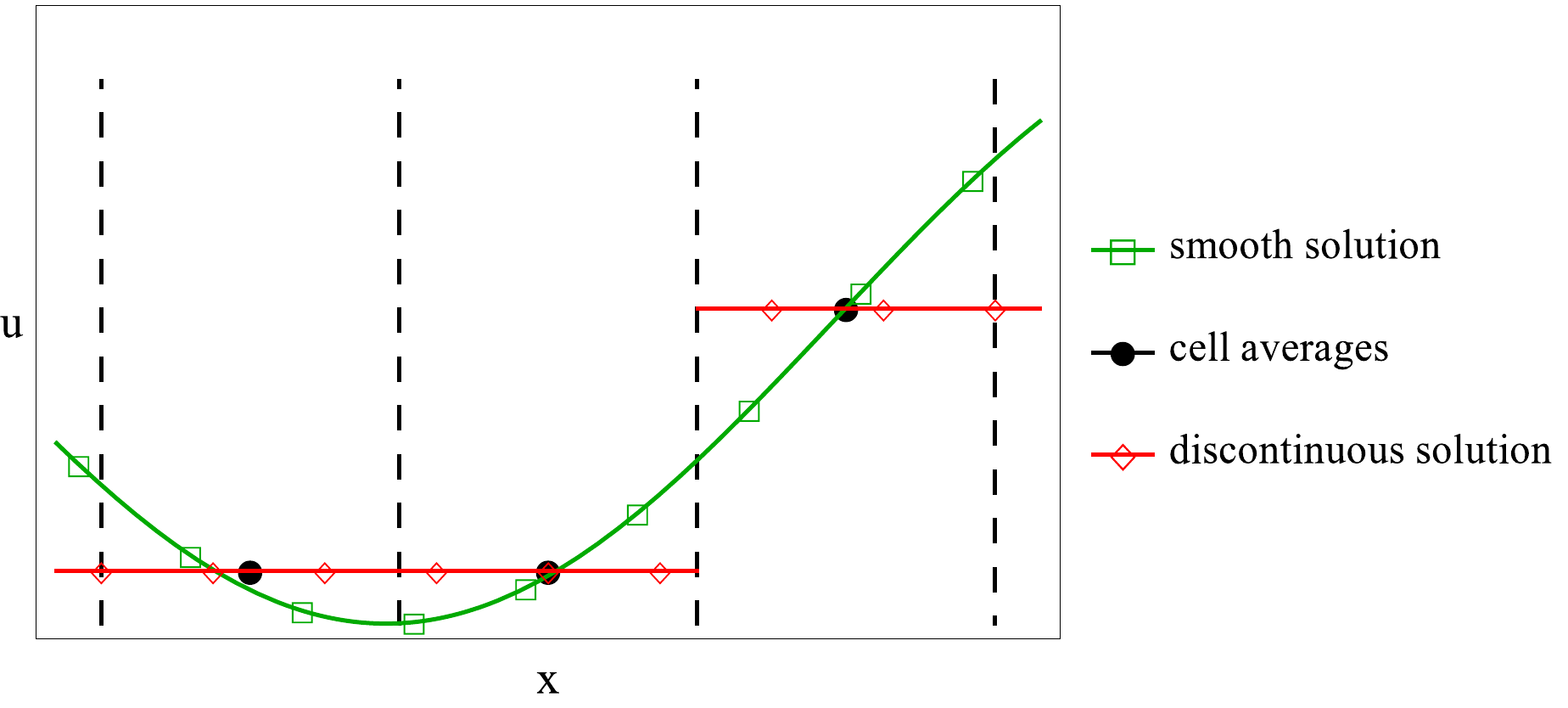}
		\caption{Identical cell mean values, corresponding to a smooth solution, or to a discontinuous solution.}
		\label{fig:smoothVSdisc}
	\end{figure}
	Given three cell averages, it is in general not possible to determine whether these points represent the onset of a discontinuity or the magnified view of an extremum. An example is shown in Fig.~\ref{fig:smoothVSdisc} where one set of cell mean values could be obtained by the smooth function (green, square markers) or the discontinuous function (red, diamond markers). As stated in Sec.~\ref{subsec:2parameters}, the two-parameter setting is the necessary prerequisite for the definition of such a criterion. {\v{C}}ada and Torrilhon \cite{CadaTorrilhon2009} proposed the function
	\begin{align}
		\label{eq:etaCT}
 			\eta_{\text{CT}}(\delta_{i-\frac{1}{2}}, \delta_{i+\frac{1}{2}})=\frac{\delta_{i-\frac{1}{2}}^2+\delta_{i+\frac{1}{2}}^2}{(r\Delta x)^2}.
	\end{align}
	This switch function defines an asymptotic region of radius $r$ around the origin in the $(\delta_{i-\frac{1}{2}}, \delta_{i+\frac{1}{2}})$-plane. Within this region the limiter switches to the $3^\text{rd}$-order reconstruction.

The authors of \cite{CadaTorrilhon2009} then modified the structure of the limiter function $\phi_{\text{CT}}$ to include the asymptotic region around the origin. The combination, denoted by the superscript $(c)$, is defined as
	\begin{equation}
		\label{eq:HCTcomb}
		\phi^{\text{(c)}}_{\text{CT}}(\theta_i) :=
		\begin{cases}
			\phi_{3}(\theta_i) \quad\qquad &\text{if}\;\eta_{\text{CT}} < 1\\
	  		\phi_{\text{CT}}(\theta_i)\quad\; &\text{if}\;\eta_{\text{CT}} \geq 1.	  		
	  	\end{cases}
	\end{equation}
	It is possible to make this function Lipschitz continuous, by introducing a small transition region and a linear function, cf.~\cite{CadaTorrilhon2009} for more details. This limiter function, combining $\phi_{\text{CT}}$ and the switch function $\eta_{\text{CT}}$, has been successfully employed in e.g. \cite{KeppensPorth2014, Mignone2010, Kemm2011}.
In \cite{CadaTorrilhon2009}, the authors did not provide a general formulation of the parameter $r$ which determines the size of the asymptotic region. Instead it was chosen ad hoc, in a problem-specific way.
To obtain some generic idea about suitable choices for $r$, we conduct a numerical test, applying $\phi^{\text{(c)}}_{\text{CT}}$ to the advection equation $u_t+u_x=0$ with smooth initial condition $u_0(x)=\sin(\pi x)$ for different values of $r$. Fig.~\ref{fig:cadaL1} shows the double logarithmic plot of the $L_1$- and $L_\infty$-errors versus the number of grid cells of the solution advected until $t_\text{end}=1$ with CFL number $\nu=0.9$. For this
smooth test case, we see that larger values of $r$, corresponding to larger asymptotic regions are favorable. For smooth solution this makes sense because increasing values of $r$ corresponds to increasing the region of directly applying the full-$3^\text{rd}$ order reconstruction $\phi_{3}$. From Fig.~\ref{fig:cadaL1} we can deduce that for smaller $r$, a finer space discretization is needed to obtain $3^\text{rd}$ order accuracy.
	\begin{figure}
		\centering
		\begin{subfigure}	{.49\textwidth}
			\includegraphics[width=\textwidth]{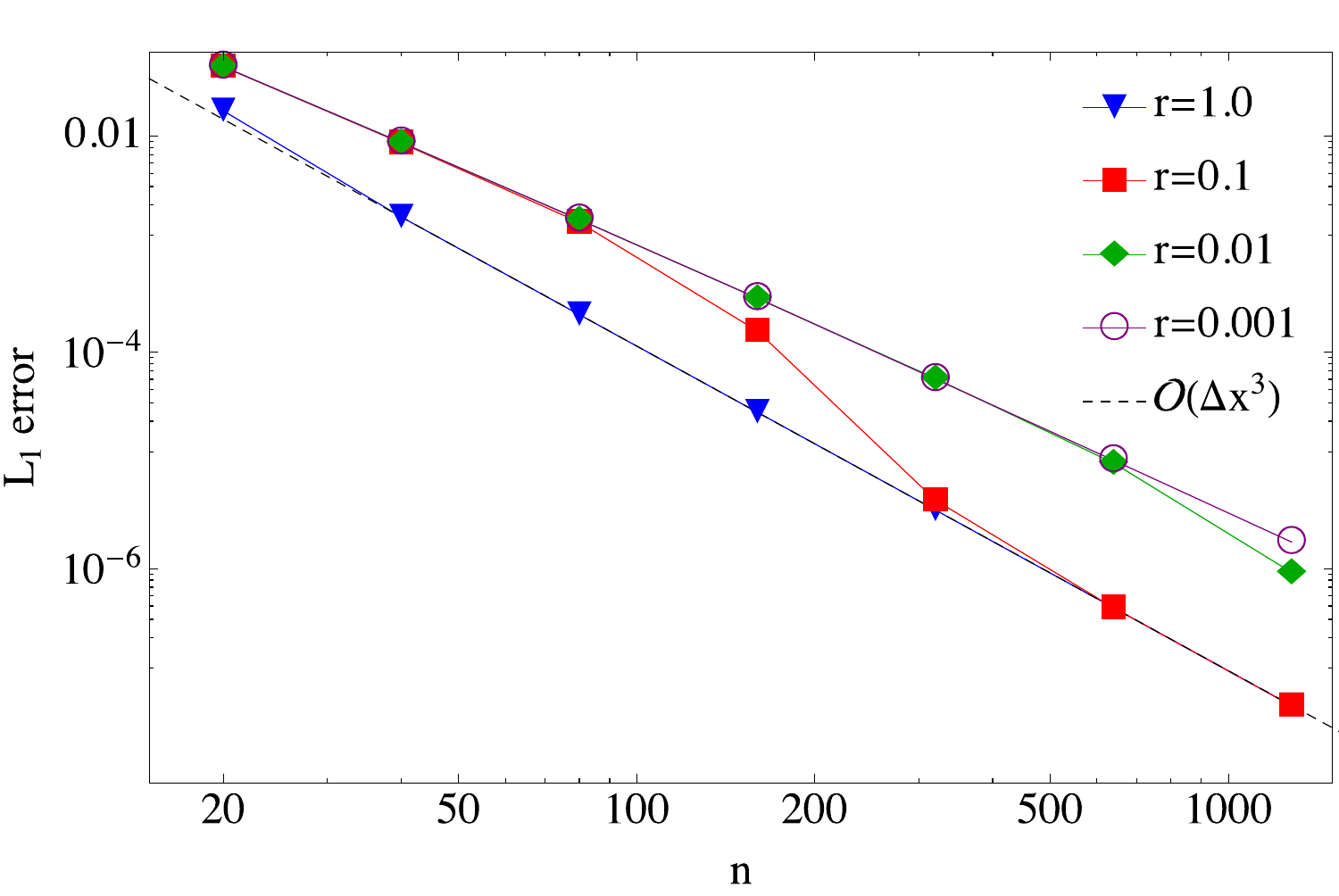}
		\end{subfigure}				
		\begin{subfigure}{.49\textwidth}
			\includegraphics[width=\textwidth]{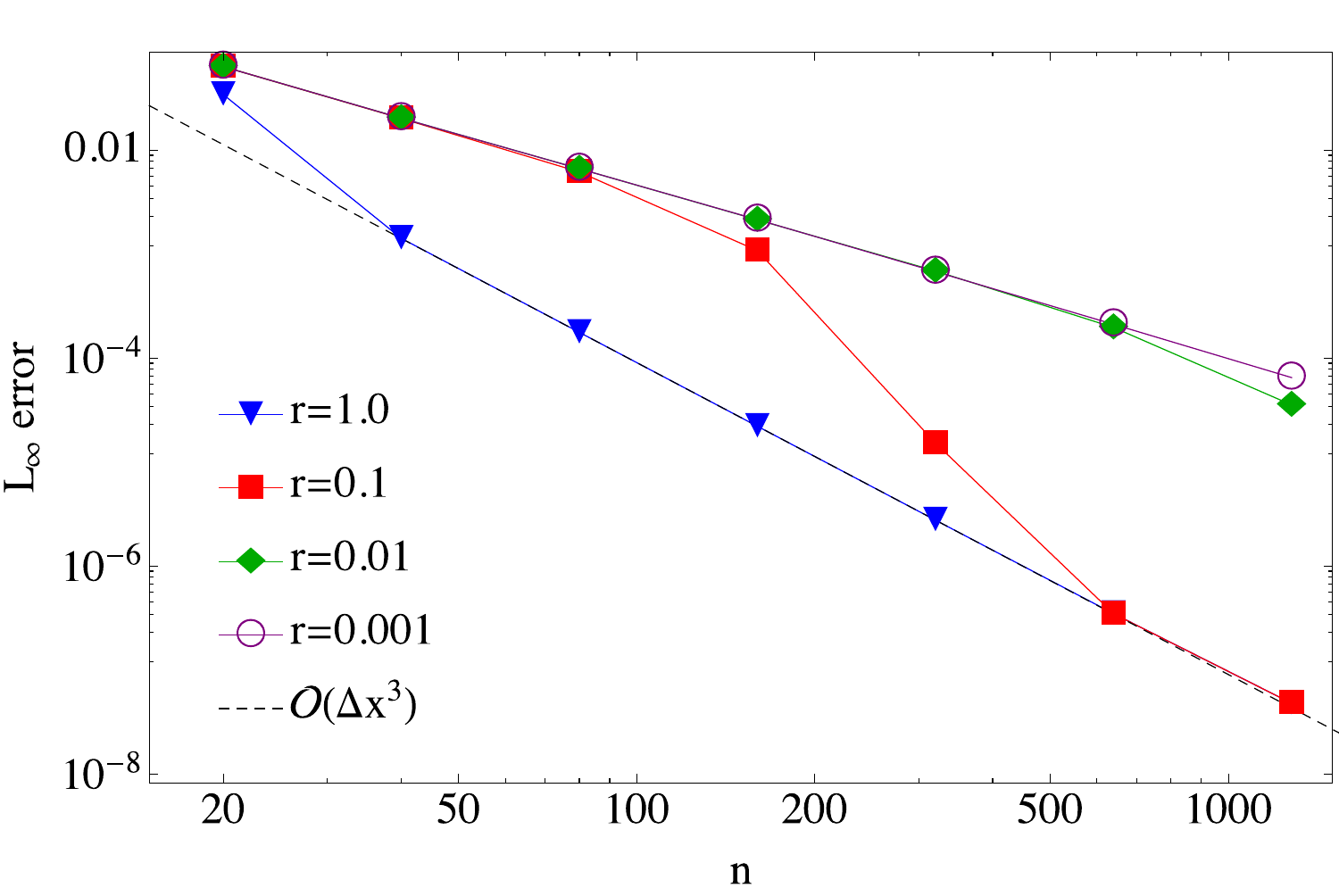}
		\end{subfigure}				
		\caption{Double logarithmic plot of the $L_1$- and $L_\infty$-errors versus the number of grid cells of the solution obtained with $\phi^{\text{(c)}}_{\text{CT}}$ advected until $t_\text{end}=1$ with CFL number $\nu=0.9$. The numerical solutions are shown for different values of $r$, see Eq.~\eqref{eq:etaCT}. Left: $L_1$-error, right: $L_\infty$-error.}
		\label{fig:cadaL1}
	\end{figure}
	\newline
	Motivated by the limiter function of \cite{CadaTorrilhon2009}, we want to define a new smoothness indicator without artificial parameters. As already sketched in \cite{SchmidtmannHYP2014}, a promising potential to distinguish discontinuities from smooth extrema is by measuring the magnitude of the vector $(\delta_{i-\frac{1}{2}}, \delta_{i+\frac{1}{2}})$. If this vector is sufficiently small in some appropriate norm, the reconstruction is switched to the full-$3^\text{rd}$ order reconstruction, even if one of the lateral derivatives may be vanishing.
	\begin{lemma}\label{lemma:magnitudeEta}
		In the vicinity of an extremum $\xi_0$, for $|x_i-\xi_0| \leq \Delta x$, the following relations hold for each time $t^n$:
		\begin{subequations}		
		\begin{align}
		\label{eq:lemma1}
			\begin{Vmatrix}			
				\left( \delta_{i-\frac{1}{2}}, \delta_{i+\frac{1}{2}}\right)
			\end{Vmatrix}_2
			&\leq \,\sqrt{c}\,\max_{x \in \Omega \backslash \Omega_d}  |u^{\prime\prime}(x,t^n)|\,\Delta x^2\quad \text{with}\;c=\frac{5}{2}+\O(\Delta x),
			\\
			\label{eq:lemma2}
			\begin{Vmatrix}			
				\left( \delta_{i-\frac{1}{2}}, \delta_{i+\frac{1}{2}}\right)
			\end{Vmatrix}_1
			&\leq \,c\,\max_{x \in \Omega \backslash \Omega_d}  |u^{\prime\prime}(x,t^n)|\,\Delta x^2\quad \text{with}\;c=2+\O(\Delta x).
		\end{align}			
		\end{subequations}	
		Here, $\Omega$ is the computational domain, and $\Omega_d$ is a set of points where the solution is discontinuous.
	\end{lemma}
	\begin{proof}
		Let us recall the definition of $\bar{U}(x,t)$, Eq.~\eqref{eq:exSol}, 
		\begin{align*}
			\bar U(x, t) = \frac{1}{\Delta x} \int_{x-\Delta x/2}^{x+\Delta x/2} u(s,t) ds,
		\end{align*}
		where $u$ is the exact solution. The following properties hold for $\bar{U}(x,t)$
		\begin{subequations}
		\label{eq:propertiesUbar}
			\begin{align}
				\bar{U}^\prime(x,t)&=\frac{u(x+\frac{\Delta x}{2}, t)-u(x-\frac{\Delta x}{2}, t)}{\Delta x} \\
				\bar{U}^{\prime\prime}(x,t)&=\frac{u^\prime(x+\frac{\Delta x}{2}, t)-u^\prime(x-\frac{\Delta x}{2}, t)}{\Delta x}\\				
				\bar{u}_i^n &\approx \bar{U}_i^n=\bar U(x_i, t^n),
			\end{align}
		\end{subequations}	
		where $\bar U(x_i,t^n)$ is the cell average of the exact solution at time $t^n$ in cell $C_i$ with cell center $x_i$.\\
		Eq.~\eqref{eq:lemma1} can be proven regarding the following formulation of definitions \eqref{eq:deltaP}, \eqref{eq:deltaM}, with a constant $\alpha$ which is going to be specified. For the sake of simplicity we shall neglect $t^n$ for the rest of the proof.
		\begin{subequations}
	  		\label{eq:etaBdiscMVlong}
			\begin{align}
				\label{eq:startingFormulation}
				&\frac{(\bar{U}(x_i+\Delta x)-\bar{U}(x_i))^2 + (\bar{U}(x_i)-\bar{U}(x_i-\Delta x))^2}{(\alpha\,\Delta x^2)^2}= \\
				&= \frac{1}{\alpha^2}\left(\frac{\bar U(x_{i}+\Delta x)-2\bar U(x_i)+ \bar U(x_{i}-\Delta x)}{\Delta x^2}\right)^2
			    +\frac{2}{\alpha^2 \Delta x^2}\left(\frac{\bar U(x_{i}+\Delta x)-\bar U(x_i)}{\Delta x}\right)
			    \left(\frac{\bar U(x_i)-\bar U(x_{i}-\Delta x)}{\Delta x}\right)
		\intertext{a Taylor expansion of the functions $\bar U(x_{i}\pm \Delta x)$ around $x_i$ yields}
		&= \frac{1}{2}\left(\frac{\bar U^{\prime\prime}(x_i)}{\alpha}\right)^2
	    +\frac{2}{\Delta x^2}\left(\frac{\bar U^\prime(x_i)}{\alpha}\right)^2
	    + \frac{2}{3}\frac{\bar U^\prime(x_i) \bar U^{\prime\prime\prime}(x_i)}{\alpha^2}
   	    +\mathcal{O}(\Delta x^2).
	\end{align}
	\end{subequations}  	 
    In the vicinity of an extremum $\xi_0$, for $|x_i-\xi_0| \leq \Delta x$, the derivative satisfies
    \begin{subequations}
		\begin{align}
			\bar U^\prime(x_i) &= \bar U^\prime(\xi_0)+ \bar U^{\prime\prime}(\xi_0)(x_i-\xi_0)+\O((x_i-\xi_0)^2)\\
				&=\bar U^{\prime\prime}(\xi_0)(x_i-\xi_0)+\O(\Delta x^2)	\\
			\Rightarrow |\bar U^\prime(x_i)| &\leq|\bar U^{\prime\prime}(\xi_0)|\Delta x+\O(\Delta x^2).\\				
			\intertext{Since we are interested in $\left(\bar U^\prime(x_i)\right)^2$, we find that}
			\left(\bar U^\prime(x_i)\right)^2 &\leq  \left(\bar U^{\prime\prime}(\xi_0)\right)^2\Delta x^2+\O(\Delta x^3).		 \\
    \intertext{Therefore, Eq.~\eqref{eq:etaBdiscMVlong} reduces to}
  		\label{eq:etaBdiscMV}	
		&\hspace{-1.5cm}\frac{(\bar{U}(x_i+\Delta x)-\bar{U}(x_i))^2 + (\bar{U}(x_i)-\bar{U}(x_i-\Delta x))^2}{(\alpha\,\Delta x^2)^2}
		\leq \frac{1}{2}\left(\frac{\bar U^{\prime\prime}(x_i)}{\alpha}\right)^2
	    +2\left(\frac{\bar U^{\prime\prime}(\xi_0)}{\alpha}\right)^2+\mathcal{O}(\Delta x).
	\end{align}
	\end{subequations}  	 
    Consider the computational domain $\Omega$, and the set of points $\Omega_d$, where the solution is discontinuous. Setting $\alpha \equiv \max_{x_i\in\Omega\backslash \Omega_d} |\bar U^{\prime\prime}(x_i,t^n)|$ with the cell centers $x_i$, leads to
    \begin{align}
		\frac{(\bar{U}(x_i+\Delta x)-\bar{U}(x_i))^2 + (\bar{U}(x_i)-\bar{U}(x_i-\Delta x))^2}{\left(\max_{x_i} |\bar U^{\prime\prime}(x,t^n)|\Delta x^2\right)^2}		
	\leq \frac{5}{2}+\O(\Delta x).
	\end{align}
    Since the numerical solution $\bar u_i^n$ is a $3^\text{rd}$-order-accurate approximation of the true solution, i.e. $\bar u_i^n=\bar{U}(x_i,t^n)+\O(\Delta x^3)$ and 
	\begin{align*}
		\alpha =&\max_{x_i} |\bar U^{\prime\prime}(x_i,t^n)| \stackrel{\mathrm{\eqref{eq:propertiesUbar}}}\leq\max_{x_i}\max_{\xi\in(x_i-\tfrac{\Delta x}{2}, x_i+\tfrac{\Delta x}{2}) } |u^{\prime\prime}(\xi, t)|
	\end{align*}	    
    holds true, this shows Eq.~\eqref{eq:lemma1}. In a similar manner, Eq.~\eqref{eq:lemma2} can be proven.
	\end{proof}
	\begin{remark} Often, the exact value of $\max_{x} |u^{\prime\prime}(x,t^n)|$ is not known or it is too expensive to compute. In these cases, a different estimator needs to be found.
		\begin{itemize}
			\item[(a)] In many applications, one has some estimate of the largest second derivative of the solution, even if one does not know the solution itself.
			\item[(b)] If a good estimate of the solution at time $t^n$ is unavailable, one can use the initial conditions $u_0(x)$ as an approximation of $\alpha$. Note that for a conservation law of the form $u_t + f(u)_x =
0$, certain information about $u_{xx}$ can in fact be inferred from the initial data (where the solution is smooth). For instance, second derivatives are actually preserved at extremal points. To see this, consider the equations that $u_x$ and $u_{xx}$ satisfy, namely $(u_x)_t+ f'(u)(u_x)_x = -f''(u)(u_x)^2$ and $(u_{xx})_t + f'(u)(u_{xx})_x =-f'''(u)u_x^3 - 3f''(u)u_xu_{xx}$. In both equations, the left hand side represents the convective derivative along characteristics. Therefore, extremal points ($u_x=0$) are just moved with the characteristics, and $u_{xx}$ remains unchanged from its value at initial time.
		\end{itemize} 
		From now on, we will use $\alpha\equiv\max_{x\in \Omega \backslash \Omega_d} |u_0^{\prime\prime}(x)|$.
	\end{remark}
	Lemma \ref{lemma:magnitudeEta} makes a statement on the magnitude of the normalized slopes.
Note that the upper bound only depends on the grid size $\Delta x$ and the initial condition $u_0$. From this lemma we can now deduce the definition of the smoothness indicator.
	\begin{definition}\label{def:eta}
		The switch function $\eta$ which marks the limit between smooth extrema and discontinuities is defined by
		\begin{align}
		\label{eq:eta}
 			\eta=\eta(\delta_{i-\frac{1}{2}},\delta_{i+\frac{1}{2}}) = \frac{\sqrt{\delta^2_{i-1/2}+\delta^2_{i+1/2}}}{\sqrt{\frac{5}{2}}\,\alpha\,\Delta x^2}	
 		\end{align}
 		with
 		\begin{align}
 		\label{eq:alpha}
 			\alpha \equiv \max_{x \in \Omega \backslash \Omega_d} |u_{0}^{\prime\prime}(x)|
 		\end{align} 		 		
 		$\Omega$ and $\Omega_d$ defined as in Lemma \ref{lemma:magnitudeEta}.
	\end{definition}
	Note that $\alpha$ is proportional to $1/\Delta x^2$ and therefore, $\eta$ is a non-dimensional quantity, so that the overall scheme is not affected by changes of units.

	The idea of using the largest second derivative of the initial conditions to relax limiting near extrema has already been proposed by Cockburn and Shu \cite{CockburnShu1989} in the context of discontinuous Galerkin methods. They use the constant $M_2 = \max_{x} |u_{0}^{\prime\prime}(x)|$ to overcome the degeneration to first order at critical points: $M_2$ is used to estimate the magnitude of the reconstructions $u^{(-)}_{i\pm 1/2}, u^{(+)}_{i\pm 1/2}$.
If the reconstruction is smaller than a certain bound, the high order reconstruction is used, otherwise, a limiter function is applied. Note that in \cite{CockburnShu1989} the switch function is based on the reconstructed values $u^{(-)}_{i\pm 1/2}, u^{(+)}_{i\pm 1/2}$, whereas in this work, we consider the normalized slopes $\delta_{i\pm 1/2}$. 
	\newline\newline
	With Def.~\ref{def:eta}, Lemma~\ref{lemma:magnitudeEta} states that in the vicinity of smooth extrema, $\eta \leq 1$ holds. Combining this information with the new limiter function $H_{3\text{L}}$, we use this result to define the combined limiter, denoted by the superscript $(c)$,
	\begin{align}
		\label{eq:HnewComb}
			&H^{\text{(c)}}_{3\text{L}}(\delta_{i-\frac{1}{2}},\delta_{i+\frac{1}{2}}) :=
	  		\begin{cases}
		  		H_{3}(\delta_{i-\frac{1}{2}},\delta_{i+\frac{1}{2}}) \quad\qquad &\text{if}\;\eta <1  \\
	  	  		H_{3\text{L}}(\delta_{i-\frac{1}{2}},\delta_{i+\frac{1}{2}})\quad\; &\text{if}\;\eta \geq 1 .
	  		\end{cases}
	\end{align}	
	The resulting two-dimensional limiter function for this approach is shown in Fig.~\ref{fig:limiter_new_asy}. Note that in the same manner as for $H^{\text{(c)}}_\text{CT}$, Lipschitz-continuity of the combined limiter $H^{\text{(c)}}_{3\text{L}}(\delta_{i-\frac{1}{2}},\delta_{i+\frac{1}{2}})$, Eq.~\eqref{eq:HnewComb}, can be achieved by using a continuous switch function for the transition between the limited and non-limited reconstruction.
	\begin{figure}
	\centering
		\begin{subfigure}{.48\textwidth}
			\includegraphics[width=\textwidth]{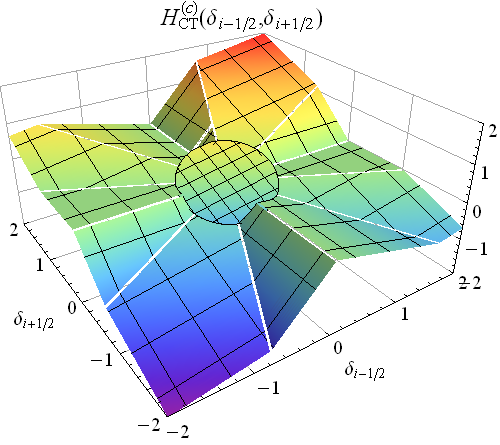}
			\caption{Combined limiter function $H^{\text{(c)}}_\text{CT}$, \cite{CadaTorrilhon2009}.}
			\label{fig:limiter_ext_asy}
		\end{subfigure}
		\begin{subfigure}{.48\textwidth}
			\includegraphics[width=\textwidth]{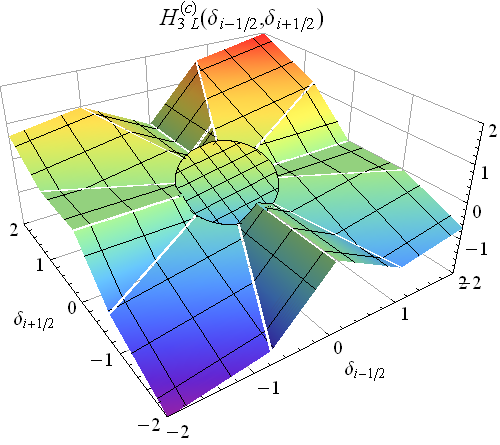}
			\caption{New combined limiter function $H^{\text{(c)}}_{3\text{L}}$.}
			\label{fig:limiter_new_asy}			
		\end{subfigure}		
		\caption{Limiter functions combined with the full-$3^\text{rd}$ order reconstruction in a region around the origin.}
		\label{fig:comparePhiCTandPhi3L}
	\end{figure}
	\begin{remark}
%
	The proposed limiter $H^{\text{(c)}}_{3\text{L}}(\delta_{i-\frac{1}{2}},\delta_{i+\frac{1}{2}})$ still satisfies properties \ref{prop:Homogeneity} to \ref{prop:LeftRight}, i.e. it is homogeneous, translationally invariant and fulfills the left-right symmetry. For the homogeneity, this can be seen by the fact that the stretched vector $(k\,u_{i-1}, k\,u_i, k\,u_{i+1})$ leads to the undivided differences $(k\,\delta_{i-\frac{1}{2}},k\,\delta_{i+\frac{1}{2}})$, which are stretched by the same factor. Also, observe how $(k\,u_{i-1}, k\,u_i, k\,u_{i+1})$ affects $\alpha$, Eq.~\eqref{eq:alpha},
	\begin{subequations}
		\begin{align}
			\alpha(k\,u_0)=& \max_{x \in \Omega \backslash \Omega_d} |k\,u_{0}^{\prime\prime}(x)|= |k|\max_{x \in \Omega \backslash \Omega_d} |u_{0}^{\prime\prime}(x)|=|k| \alpha(u_0).
		\intertext{With definition \eqref{eq:eta}, we obtain that the decision criterion is homogeneous, i.e.,}
				&\eta(k\,\delta_{i-\frac{1}{2}},k\,\delta_{i+\frac{1}{2}})=\eta(\delta_{i-\frac{1}{2}},\delta_{i+\frac{1}{2}}).\label{eq:homogenityOfEta}
		\intertext{Using the fact that $H_{3}(\delta_{i-\frac{1}{2}},\delta_{i+\frac{1}{2}})$ and $H_{3\text{L}}(\delta_{i-\frac{1}{2}},\delta_{i+\frac{1}{2}})$ are homogeneous functions, finally leads to }
			\Rightarrow &H^{\text{(c)}}_{3\text{L}}(k\,\delta_{i-\frac{1}{2}},k\,\delta_{i+\frac{1}{2}})=k\;H^{\text{(c)}}_{3\text{L}}(\delta_{i-\frac{1}{2}},\delta_{i+\frac{1}{2}}).
		\end{align}		
	\end{subequations}

	For the translational invariance, note that the shifted vector $(k\,+u_{i-1}, k\,+ u_i, k\,+u_{i+1})$ leads to the original formulation $(\delta_{i-\frac{1}{2}},\delta_{i+\frac{1}{2}})$ and does not affect $\alpha$ either.
	Contrary to $H^{\text{(c)}}_{3\text{L}}$, the limiter function $H^{\text{(c)}}_{\text{CT}}(\delta_{i-\frac{1}{2}},\delta_{i+\frac{1}{2}})$ is not homogeneous, because $\eta_\text{CT}(k\delta_{i-\frac{1}{2}},k\delta_{i+\frac{1}{2}})=k^2\eta_\text{CT}(\delta_{i-\frac{1}{2}},\delta_{i+\frac{1}{2}})\neq \eta_\text{CT}(\delta_{i-\frac{1}{2}},\delta_{i+\frac{1}{2}})$, compare to Eq.~\eqref{eq:homogenityOfEta}.
	\end{remark}
%
\section{Third order WENO}\label{sec:WENO}
In this section we briefly review the $3^\text{rd}$-order weighted essentially non-oscillatory (WENO) scheme for one-dimensional scalar conservation laws. We consider WENO schemes because they represent a very popular method for approximating the solution of hyperbolic conservation laws. Since we want to compare WENO with the limiter functions introduced in Section \ref{sec:LimO3}, we focus on WENO3. This is the family of schemes which only use three points for the reconstruction of the left and right cell-interfaces $u^{(+)}_{i-1/2}$ and $u^{(-)}_{i+1/2}$. WENO3 reconstructs an estimate of the solution $u(x_{i+\frac{1}{2}})$ from the cell averages $\bar u_{i-1}$, $\bar u_i$, and $\bar u_{i+1}$. The procedure, introduced by Jiang and Shu \cite{JiangShu1996}, is to start just as for the $r$-th-order ENO scheme \cite{HartenEngquistOsherChakravarthy1987}, where $r=2$. This is, we consider the two-point stencils $S_0=C_{i-1} \cup C_i$ and $S_1=C_i\cup C_{i+1}$ and define on each stencil $2^\text{nd}$-order accurate approximations:
	\begin{itemize}
		\item $p_0(x_{i+\frac{1}{2}}) = -\frac{1}{2}u_{i-1}+\frac{3}{2}u_i$, where $p_0$ is the linear interpolation of $(x_{i-1}, \bar u_{i-1})$ and $(x_i, \bar u_i)$ and
		\item $p_1(x_{i+\frac{1}{2}}) = \phantom{-}\frac{1}{2}u_i+\frac{1}{2}u_{i+1}$, where $p_1$ is the linear interpolation of $(x_i, \bar u_i)$ and $(x_{i+1}, \bar u_{i+1})$.
	\end{itemize}
	For the sake of simplicity, we only consider the procedure for the right cell interface $x_{i+\frac{1}{2}}$. The left interface follows in a similar manner. \newline
	Based on an $r$-th order ENO scheme, the best one can get is a $(2r-1)$-th order WENO scheme, i.e. in our case $3^\text{rd}$ order \cite{JiangShu1996}. To obtain this, the WENO3 estimate of $u(x_{i+\frac{1}{2}})$, called $\hat u_{i+\frac{1}{2}}$, is defined as a convex combination of the two $2^\text{nd}$-order estimates $p_0(x_{i-\frac{1}{2}})$ and $p_1(x_{i+\frac{1}{2}})$:
	\begin{align}
		\label{eq:weightedFormula}
		\hat u_{i+\frac{1}{2}}= w_{i-\frac{1}{2}}p_0(x_{i+\frac{1}{2}})+w_{i+\frac{1}{2}}p_1(x_{i+\frac{1}{2}}).
	\end{align}
	The weights satisfy $w_{i\pm\frac{1}{2}} \geq 0$ and $w_{i-\frac{1}{2}}+w_{i+\frac{1}{2}}=1$. There is a particular choice of weights that generates a $3^\text{rd}$-order-accurate approximation to $u(x_{i+\frac{1}{2}})$, namely $w_{i-\frac{1}{2}} = \gamma_{i-\frac{1}{2}} = \frac{1}{3}$, and $w_{i+\frac{1}{2}} = \gamma_{i+\frac{1}{2}} = \frac{2}{3}$. This approximation is obtained when interpolating $(x_{i-1},\bar u_{i-1})$, $(x_i, \bar u_i)$, and $(x_{i+1}, \bar u_{i+1})$ by a quadratic polynomial, and evaluating it at $x_{i+\frac{1}{2}}$, see Eq.~\eqref{eq:thirdOrderPoly}.

	The philosophy of WENO is to design formulas for the weights $w_{i\pm\frac{1}{2}}$, such that in smooth regions, one has $w_{i\pm\frac{1}{2}}\approx\gamma_{i\pm\frac{1}{2}}$, while in regions of large gradients, more weight is given to the approximation that generates fewer spurious oscillations. This is achieved using smoothness indicators $\beta_{i\pm 1/2}$. They are defined using the normalized slopes between neighboring cells $\delta_{i\pm\frac{1}{2}}$:
	\begin{equation*}
		\beta_{i\pm\frac{1}{2}} = (\delta_{i\pm\frac{1}{2}})^2.
	\end{equation*}
	The final weights $w_{i\pm\frac{1}{2}}$ are defined by
	\begin{equation}
		\label{eq:weno_w}
		w_{i\pm\frac{1}{2}} = \frac{\alpha_{i\pm\frac{1}{2}}}{\alpha_{i-\frac{1}{2}}+\alpha_{i+\frac{1}{2}}}\;,
	\end{equation}
	where
	\begin{equation}
		\label{eq:weno_alpha}
		\alpha_{i\pm\frac{1}{2}} = \frac{\gamma_{i\pm\frac{1}{2}}}{(\varepsilon+\beta_{i\pm\frac{1}{2}})^p}\;.
	\end{equation}
	In equation \eqref{eq:weno_alpha}, there are two parameters which need to be further detailed. The integer $p\in\mathbb{N}$, which Liu et at. \cite{LiuOsherChan1994} suggest to set $p = r$, the order of the base ENO scheme. Jiang and Shu \cite{JiangShu1996} claim that for $r=2,3$, setting $p$ to $2$ is adequate. In this work, we will set $p=2$. 
	
	The other parameter in Eq.~\eqref{eq:weno_alpha} is $\varepsilon$, a small positive number, originally introduced to avoid the division by zero \cite{JiangShu1996}. We suggest that rather than fixing $\varepsilon$ to some constant, it should depend on the spacial discretization $\Delta x$. This will be discussed in more detail in Sec.~\ref{subsec:WENOsmoothness}.
	%
\subsection{Interpretation of WENO3 in 2D Slope Domain}
	%
	A brief calculation yields for the WENO3 reconstruction
	\begin{align*}
		\hat u_{i+\frac{1}{2}}
		&=
		w_{i-\frac{1}{2}}p_0(x_{i+\frac{1}{2}})+w_{i+\frac{1}{2}}p_1(x_{i+\frac{1}{2}}) \\
		&=
		w_{i-\frac{1}{2}} \prn{-\tfrac{1}{2}\bar u_{i-1}+\tfrac{3}{2}\bar u_i}
		+w_{i+\frac{1}{2}}\prn{\tfrac{1}{2}\bar u_i+\tfrac{1}{2}\bar u_{i+1}} \\
		&=
		\bar u_i+\frac{1}{2}\prn{w_{i-\frac{1}{2}}(\bar u_i-\bar u_{i-1})+w_{i+\frac{1}{2}}(\bar u_{i+1}-\bar u_i)} \\
		&=
		\bar u_i+\frac{1}{2}\prn{w_{i-\frac{1}{2}}\delta_{i-\frac{1}{2}}
		+w_{i+\frac{1}{2}}\delta_{i+\frac{1}{2}}}\;.
	\end{align*}
	Since the weights $w_{i\pm\frac{1}{2}}$ themselves only depend on the slopes $\delta_{i\pm\frac{1}{2}}$, the WENO3 estimate can be written in the form \eqref{eq:update_finite_volume_2d}
	with the limiter function $H_\text{WENO3}$, whose particular form depends on the parameters $p$ and $\varepsilon$. Explicitly written, it is
	\begin{equation}
		\label{eq:weno3}
		H_\text{WENO3}(\delta_{i-\frac{1}{2}},\delta_{i+\frac{1}{2}}) =
		\frac{
		\frac{1}{3}(\varepsilon+(\delta_{i-\frac{1}{2}})^2)^{-p}\delta_{i-\frac{1}{2}}
		+\frac{2}{3}(\varepsilon+(\delta_{i+\frac{1}{2}})^2)^{-p}\delta_{i+\frac{1}{2}}
		}{
		\frac{1}{3}(\varepsilon+(\delta_{i-\frac{1}{2}})^2)^{-p}
		+\frac{2}{3}(\varepsilon+(\delta_{i+\frac{1}{2}})^2)^{-p}
		}\;.
	\end{equation}
	%
In the vicinity of $(0,0)$, i.e.~for $|\delta_{i\pm\frac{1}{2}}|\ll\varepsilon$, one has in leading order that
	\begin{equation}
	\label{eq:HwenoSmallDeltas}
		H_\ll(\delta_{i-\frac{1}{2}},\delta_{i+\frac{1}{2}}) =
		\tfrac{1}{3}\delta_{i-\frac{1}{2}}+\tfrac{2}{3}\delta_{i+\frac{1}{2}}\;.
	\end{equation}
	As mentioned before, this linear function is homogeneous and results in a $3^\text{rd}$-order-accurate approximation. On the other side of the spectrum, if $|\delta_{i\pm\frac{1}{2}}|\gg\varepsilon$,
	one has in leading order that
	\begin{equation}
		\label{eq:HwenoBigDeltas}
		H_\gg(\delta_{i-\frac{1}{2}},\delta_{i+\frac{1}{2}}) =
		\frac{
		\frac{1}{3}(\delta_{i-\frac{1}{2}})^{1-2p}
		+\frac{2}{3}(\delta_{i+\frac{1}{2}})^{1-2p}
		}{
		\frac{1}{3}(\delta_{i-\frac{1}{2}})^{-2p}
		+\frac{2}{3}(\delta_{i+\frac{1}{2}})^{-2p}
		}\;.
	\end{equation}
	This function is also homogeneous, i.e. $H_\gg(k\delta_{i-\frac{1}{2}},k\delta_{i+\frac{1}{2}}) = k\,H_\gg(\delta_{i-\frac{1}{2}},\delta_{i+\frac{1}{2}})$. That means, along each line through the origin in the $(\delta_{i-\frac{1}{2}},\delta_{i+\frac{1}{2}})$ plane, it is linear, thus resembling the behavior of traditional FV limiters.
\subsection{WENO Smoothness Indicators}\label{subsec:WENOsmoothness}
	The choice of the weights $\omega_{i\pm 1/2}$, which depend on $\alpha_{i\pm 1/2}$, is crucial for the order of accuracy of the resulting scheme. Its precise value influences the behavior of the limiting when $\beta_{i\pm 1/2}$ is close to zero. Clearly, this is of particular importance near extremal points of the solution. To point out the importance of $\varepsilon$, we repeated the numerical experiment conducted in Sec.~\ref{subsec:smoothnessIndicator} with the $3^\text{rd}$-order WENO scheme. This is, we apply $H_\text{WENO3}$ to the advection equation $u_t+u_x=0$ with smooth initial condition $u_0(x)=\sin(\pi x)$ for different values of $\varepsilon$. The result for $t_\text{end}=1$ and CFL number $\nu=0.9$ is depicted in Fig.~\ref{fig:WENOerror}.
	\begin{figure}[t]
		\centering
		\begin{subfigure}	{.49\textwidth}
			\includegraphics[width=\textwidth]{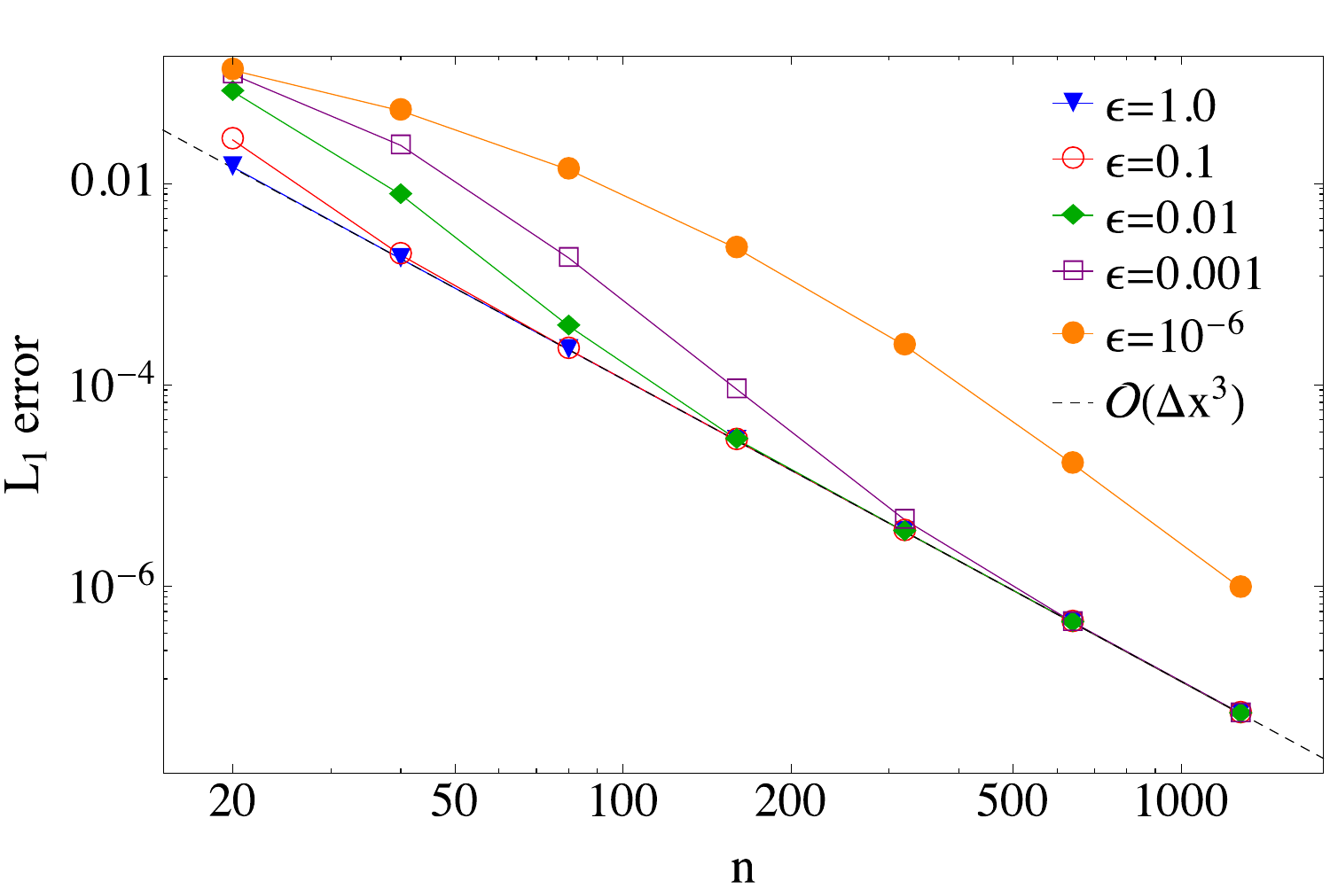}
		\end{subfigure}				
		\begin{subfigure}{.49\textwidth}
			\includegraphics[width=\textwidth]{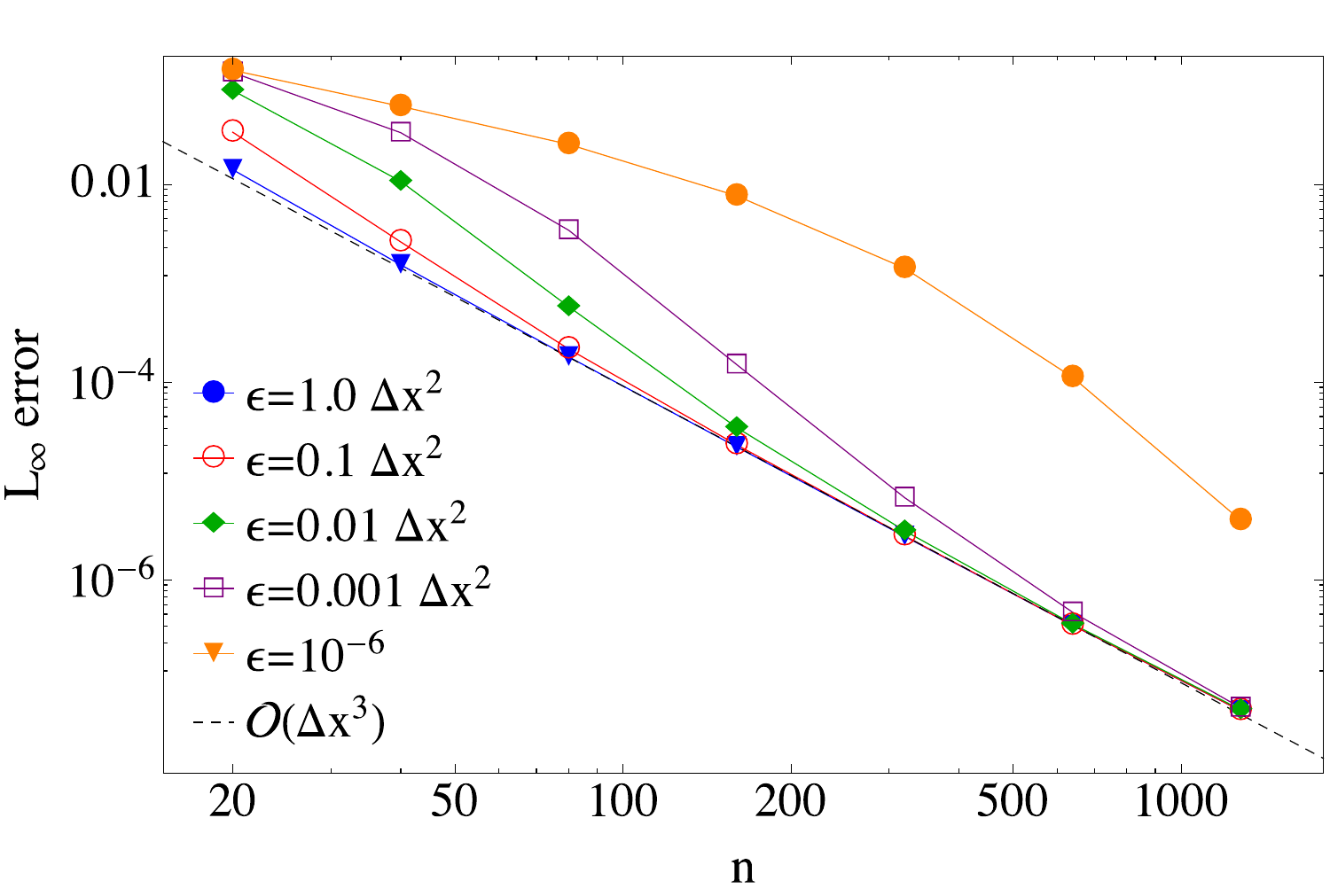}
		\end{subfigure}				
		\caption{Double logarithmic plot of the $L_1$- and $L_\infty$-errors versus number of grid cells of the solution obtained with $H_\text{WENO3}$ advected until $t_\text{end}=1$ with CFL number $\nu=0.9$. Numerical solution for different values of $\varepsilon$, see Eq.~\eqref{eq:weno3} and \eqref{eq:weno_alpha}. Left: $L_1$-error, right: $L_\infty$-error.}
		\label{fig:WENOerror}
	\end{figure}
	There is a strong similarity between this test case and the one shown in Fig.~\ref{fig:cadaL1}. This resemblance and also comparing the form of $\alpha_{i\pm 1/2}$,  Eq.~\eqref{eq:weno_alpha}, with the new FV smoothness indicator $\eta$, Eq.~\eqref{eq:eta}, strongly suggests to take a closer look at $\varepsilon$. In recent years, this parameter has attracted a lot of attention, see e.g. \cite{YamaleevCarpenter2009}, \cite{Arandiga2011}, \cite{ArandigaMarti2014}, \cite{Kolb2014} and references therein. Originally, $\varepsilon$ was introduced by Jiang and Shu to avoid the denominator to become zero \cite{JiangShu1996}. The authors called it a small positive number and set to $\varepsilon=10^{-6}$ for their test-case studies. Therefore, WENO3 with the fixed value $\varepsilon=10^{-6}$ will be called WENO-JS henceforth. There are several drawbacks by fixing $\varepsilon$ to some value $\varepsilon_0$. One of them is the following scenario which might occur since $\delta_{i\pm\frac{1}{2}} = \mathcal{O}(\Delta x)$ in smooth parts of the solution:
	\begin{enumerate}
		\item For large grid sizes, $|\delta_{i\pm\frac{1}{2}}|\gg \varepsilon_0$ holds, leading to the homogeneous function $H_\text{WENO3} \to H_\gg(\delta_{i-\frac{1}{2}},\delta_{i+\frac{1}{2}})$, Eq.~\eqref{eq:HwenoBigDeltas}, and thus to low order.
		\item Refining the grid leads to $|\delta_{i\pm\frac{1}{2}}|\ll \varepsilon_0$, and yields $H_\text{WENO3} \to H_\ll(\delta_{i-\frac{1}{2}},\delta_{i+\frac{1}{2}})$, Eq.~\eqref{eq:HwenoSmallDeltas}, which is the full-$3^\text{rd}$ order reconstruction.
	\end{enumerate}
	This phenomenon is also demonstrated in Fig.~\ref{fig:WENOerror}.
	Wanting $\varepsilon$ to be adaptive for all grid discretizations means that it has to depend on the grid size, i.e. $\varepsilon = \varepsilon(\Delta x)$. This is also what Yamaleev and Carpenter claim in \cite{YamaleevCarpenter2009, YamaleevCarpenter2009_2}. However, the authors do not simply replace $\varepsilon$ by a function depending on $\Delta x$, they rather define new weight functions
	\begin{subequations}
		\label{eq:weightsYC}
		\begin{align}
			&w_{k} = \frac{\alpha_{k}}{\sum_{i=0}^{r-1}\alpha_{i}}\;,\quad \alpha_{k}=\gamma_{k}\left(1+\frac{\tau}{\varepsilon+\beta_{k}}\right),\quad k=0,\ldots,r-1,\\
			&\varepsilon=\max_{x\not\in\Omega_d}(\|u_0^2\|_1, \|(u_0^\prime)^2\|_1,\ldots,\|(u_0^{(r-1)})^2\|_1)\Delta x^2,
			\label{eq:epsYC}
		\end{align}
	\end{subequations}
	%
	where $\Omega_d$ is a set of points where the initial condition is discontinuous. In this paper we consider $r=2$. Thus, $\gamma_0=\gamma_{i-1/2}, \gamma_1=\gamma_{i+1/2}$, etc. are defined as before and $\tau$ is the square of the undivided difference on the entire stencil. In case of the compact three-point stencil this is $\tau = (\bar u_{i+1}-2 \bar u_i + \bar u_{i-1})^2=(\delta_{i+\frac{1}{2}}-\delta_{i-\frac{1}{2}})^2$. Note that the definition of $\varepsilon$ is not invariant under translation of the initial condition $u_0$. This leads to a different limiting behavior even though $u_0$ may simply be shifted by a constant, cf. Sec.~\ref{subsec:discIC}.
	
	Ar{\'a}ndiga et al. \cite{ArandigaMarti2014} show that using the weights proposed by Yamaleev and Carpenter, the resulting scheme reaches the maximal order $(2 r-1)$ for sufficiently smooth solutions. However, near discontinuities, the scheme achieves order of accuracy $\mathcal{O}(\Delta x^2)$ which is worse or equal to $\mathcal{O}(\Delta x^r),\;r\geq2$, the order of accuracy of the underlying ENO scheme. Ar{\'a}ndiga et al. report to have fixed this issue by slightly changing the weight functions using
	\begin{subequations}
		\label{eq:weightsAMM}
		\begin{align}
		&\alpha_{k}=\gamma_{k}\left(1+\left(\frac{\tau}{\varepsilon+\beta_{k}}\right)^\mu\right),\quad k=0,\ldots,r-1,\quad\mu=\left\lceil\frac{r}{2}\right\rceil \\
		&\varepsilon=K\,\Delta x^q,\;\,\text{with}\; K>0,\;q\in\mathbb{N},\; q\leq 4\,r-4-r/\mu.
		\end{align}
	\end{subequations}	
	In the case of the three-point stencil, $\mu=1$ and therefore the weight formulation remains the same. In the numerical test cases carried out in \cite{ArandigaMarti2014}, $K$ is set to $1$, which makes $\varepsilon$ a dimensional quantity that might be affected by changes of units. We will not consider these weight functions in our numerical experiments in Sec.~\ref{sec:results} since the definition of $\varepsilon$ in Eq.~\eqref{eq:epsYC} is clearly more elaborate. Solely in Sec.~\ref{subsec:difficultIC} we compare the obtained results with the scheme setting $K=1$, i.e. $\varepsilon=\Delta x^2$ to show the importance of the definition of $\varepsilon$.
	%
\section{A Unifying View}\label{sec:unifying}	
In this section we want to place the different methods in a unifying view to point out their differences and especially their similarities.

	For the sake of simplicity, we do not discuss the three-dimensional plots of the limiter functions but rather sectional views at fixed values of $\delta_{i+\frac{1}{2}}$. This means, the limiter functions are depicted with a one-dimensional dependence on $\delta_{i-\frac{1}{2}}$. Lemma \ref{lemma:equivLimiters} states that for $\delta_{i+\frac{1}{2}}=1$ this is equivalent to the standard $(\theta,\phi(\theta))$-plots such as Fig.~\ref{fig:ASandCTlimiter}, and found in textbooks. 
	
	To give an overview, Fig.~ \ref{fig:unifyingView} shows the limiter functions treated in the previous sections for $\delta_{i+\frac{1}{2}}\in \lbrace 2, 1, 0.5, 0.1 \rbrace$.
	\begin{figure}[ht]
		\centering
		\begin{subfigure}{0.49\textwidth}
			\includegraphics[width=\textwidth]{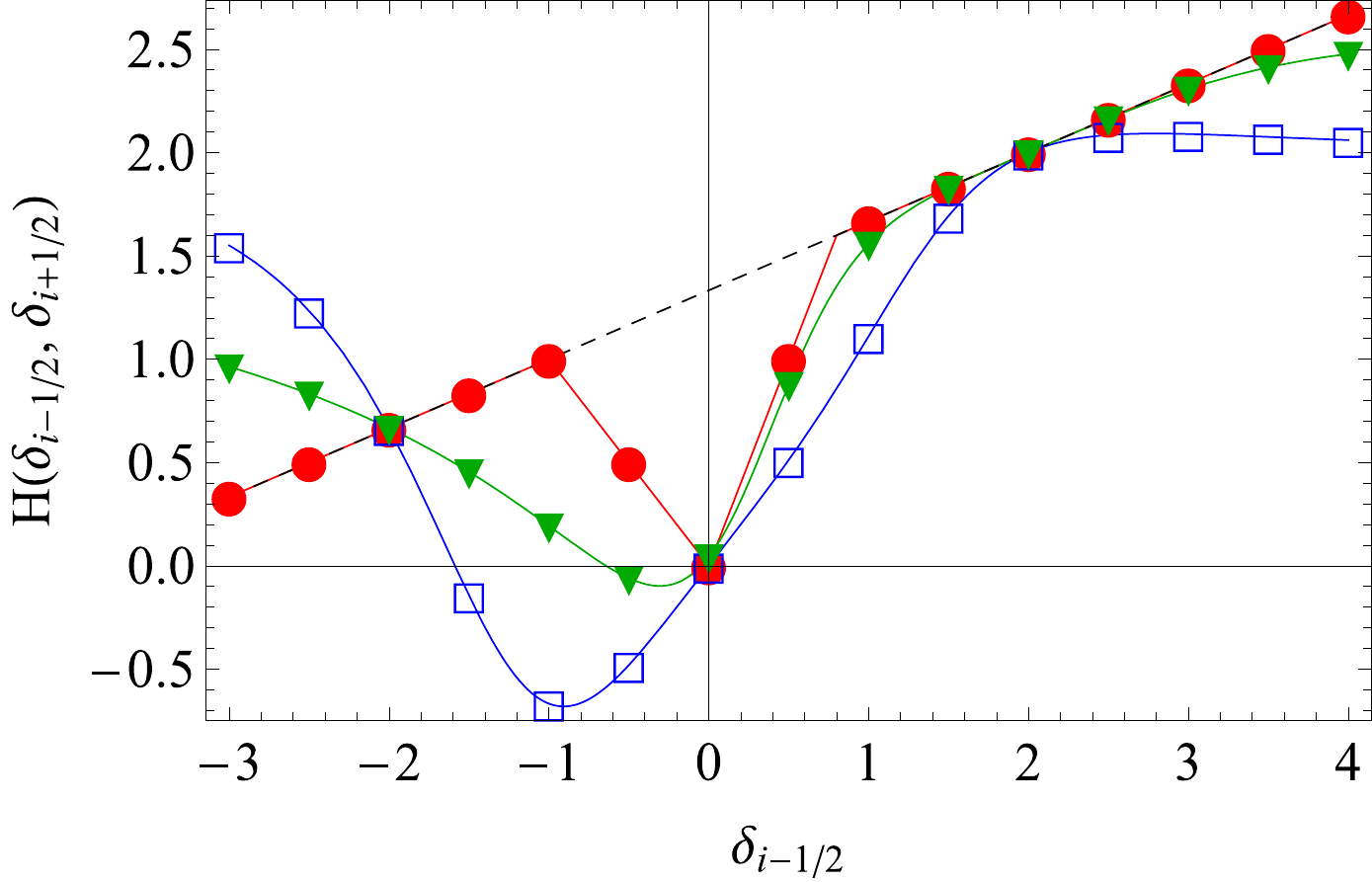}
			\caption{$\delta_{i+\frac{1}{2}}=2$.}
			\label{fig:deltap=2}
		\end{subfigure}
		\hfill
		\begin{subfigure}{0.49\textwidth}
			\includegraphics[width=\textwidth]{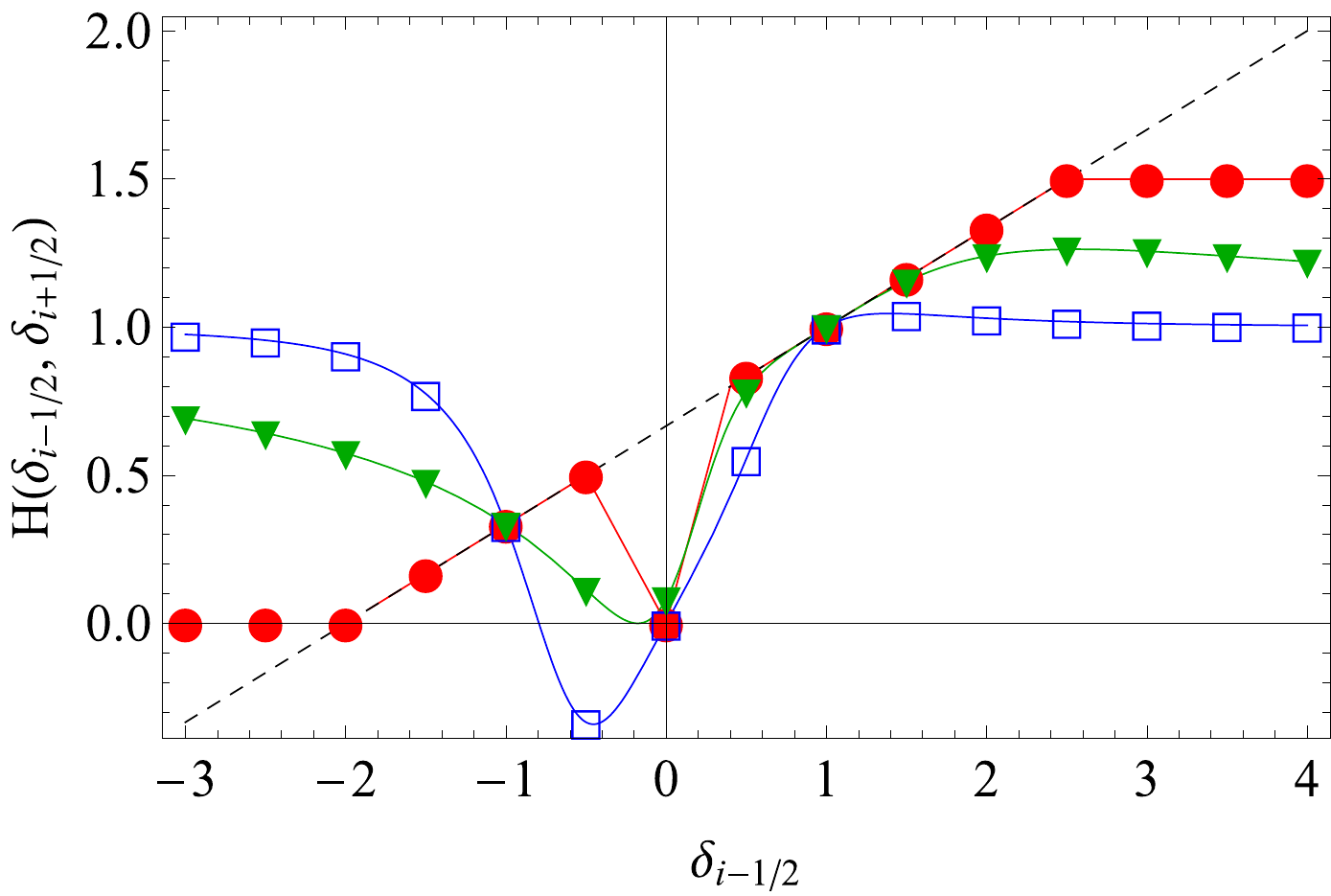}
			\caption{$\delta_{i+\frac{1}{2}}=1$.}
			\label{fig:deltap=1}
		\end{subfigure}

		\begin{subfigure}{0.49\textwidth}
			\includegraphics[width=\textwidth]{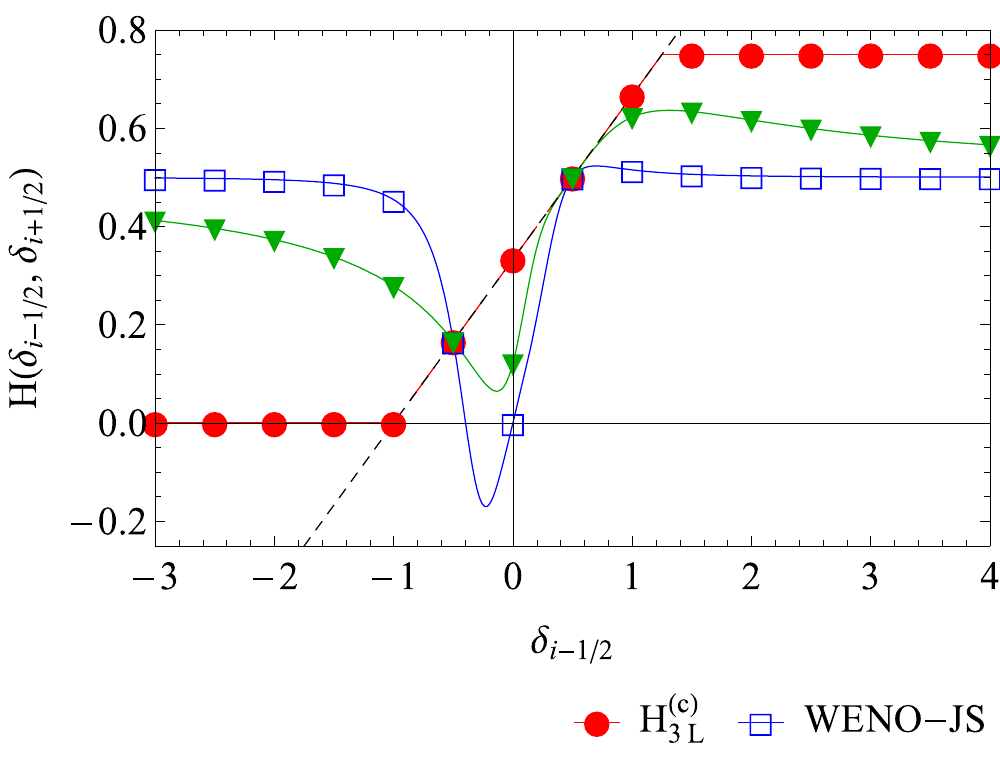}
			\caption{$\delta_{i+\frac{1}{2}}=0.5$.}
			\label{fig:deltap=0p5}
		\end{subfigure}
		\hfill
		\begin{subfigure}{0.49\textwidth}
			\includegraphics[width=\textwidth]{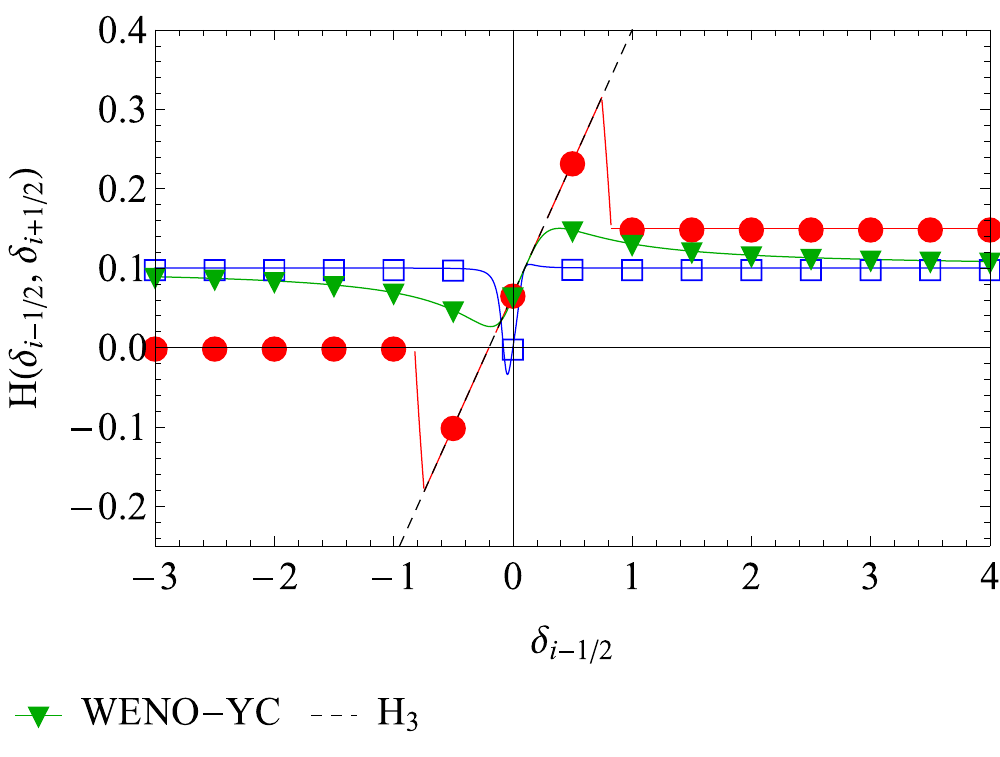}
			\caption{$\delta_{i+\frac{1}{2}}=0.1$.}
			\label{fig:deltap=0p1}
		\end{subfigure}
		\caption{Sectional view of different limiter functions for the fixed values $\delta_{i+\frac{1}{2}}\in \lbrace 2, 1, 0.5, 0.1 \rbrace$.}
		\label{fig:unifyingView}
	\end{figure}
	In Fig.~\ref{fig:unifyingView} we can nicely see that all limiter functions satisfy Property \ref{prop:generalProperties} (i). This is, for $\theta=1$, i.e.~$\delta_{i-\frac{1}{2}}=\delta_{i+\frac{1}{2}}$, the limiter functions continuously go through the point $(\delta_{i+\frac{1}{2}}, \delta_{i+\frac{1}{2}})$, which corresponds to $\phi(\theta)=1$. This situation represents smooth parts of the solution. It can be seen that WENO-JS has slope $\frac{1}{3}$  only close to the point $\delta_{i-\frac{1}{2}}=\delta_{i+\frac{1}{2}}$ meaning that as soon as the slopes (of same sign) differ from each other, the limiting takes effect, even though the function is still smooth. The smaller the values of $\delta_{i\pm 1/2}$, the stronger is this effect, see the trend from Fig.~\ref{fig:deltap=2} to \ref{fig:deltap=0p1}. 
	
	For $\delta_{i-\frac{1}{2}}\to\pm \infty$, i.e. $\theta\to\pm \infty$, all considered functions, except $H_3$, tend to a constant. For the WENO schemes, this constant is non-zero, contrary to the negative part of the FV limiter. However, for the reconstruction of the cell interface values, $H(\delta_{i-\frac{1}{2}},\delta_{i+\frac{1}{2}})\equiv \text{const.}$ leads to first-order accuracy, independent of the value of the constant. In Fig.~\ref{fig:deltap=0p5} and \ref{fig:deltap=0p1} we can see the switch of the FV limiter to the full-$3^\text{rd}$ order reconstruction, highlighted by the dashed black line. Clearly, the construction of the transition zone is rather ad-hoc for the FV limiter, while it comes more naturally in WENO3. However, the FV setup allows a much more systematic control about the shape of the function $H$ far away from the origin, particularly in the regions where $\delta_{i-\frac{1}{2}}$ and $\delta_{i+\frac{1}{2}}$ have opposite signs. It is here that the FV and WENO limiters show a very different behavior. In the point $\delta_{i-\frac{1}{2}}=-\delta_{i+\frac{1}{2}}$, i.e. at extrema, all proposed limiters equal $H_3$. However, while $H^{\text{(c)}}_{3\text{L}}$ has the same slope as the full-$3^\text{rd}$ order reconstruction, the WENO limiters have negative slopes. This indicates that the FV limiter reconstructs extrema with higher accuracy, since in general extremal points might not lie exactly at $\delta_{i-\frac{1}{2}}=-\delta_{i+\frac{1}{2}}$ but rather $\delta_{i-\frac{1}{2}}\approx -\delta_{i+\frac{1}{2}}$. In these cases, the new limiter function approximates the full-$3^\text{rd}$-order reconstruction, leading to higher order solutions.

Another interesting observation is the nature of $H^{\text{(c)}}_{3\text{L}}$ and WENO-YC at $\delta_{i-\frac{1}{2}}=0$. When considering the univariate limiter functions $\phi(\theta)$ with $\theta=\delta_{i-\frac{1}{2}}/\delta_{i+\frac{1}{2}}$, the point $\delta_{i-\frac{1}{2}}=0$ is equivalent to $\theta=0$. In this case, the conventional FV limiter functions in the MUSCL framework are set to $0$, yielding a first-order reconstruction. This is also the case for WENO-JS, whereas the limiters proposed in this work and by Yamaleev and Carpenter yield non-zero contributions for small $\delta_{i+\frac{1}{2}}$, see Fig.~\ref{fig:deltap=1}, \ref{fig:deltap=0p5}, and \ref{fig:deltap=0p1}. This is exactly the desired behavior close to extrema where one slope may be zero and the other slope small but non-zero. Further away from the origin in the $(\delta_{i-\frac{1}{2}}, \delta_{i+\frac{1}{2}})$-plane, it is clear that the limiter functions should yield $0$ for $\delta_{i-\frac{1}{2}}=0$ because this situation may correspond to a discontinuity. This feature can be observed in Fig.~\ref{fig:deltap=2}.

\section{Numerical Results}\label{sec:results}
In this section, we apply and compare the different limiter functions discussed in the previous sections. In all test problems proposed in this work, we compare
\begin{enumerate}
	\item the original WENO3 scheme as introduced in \cite{JiangShu1996}, i.e. fixing $\varepsilon=10^{-6}$, called WENO-JS;
	\item WENO3 with the weight functions proposed by Yamaleev and Carpenter \cite{YamaleevCarpenter2009}, called WENO-YC; 
	\item the full $3^{\text{rd}}$-order reconstruction $H_{3}$, Eq.~\eqref{eq:3rdOrder2param}; and
	\item the FV limiter function $H^{\text{(c)}}_{3\text{L}}$, as introduced in Sec.~\ref{sec:LimO3}.
\end{enumerate}
The time derivative is approximated using the $3^\text{rd}$-order Runge-Kutta method as described in \cite{ShuOsher1988}.

For all test cases, rather than presenting tables with errors and convergence rates, we plot the $L_1$- and $L_\infty$-errors. In these plots, we include reference slopes of the design order of accuracy, i.e. the order of accuracy the schemes obtain in theory.
\subsection{Preliminary Test Case}\label{subsec:preliminaryTest}
%
	\begin{figure}[t]
		\centering
		\includegraphics[scale=0.57]{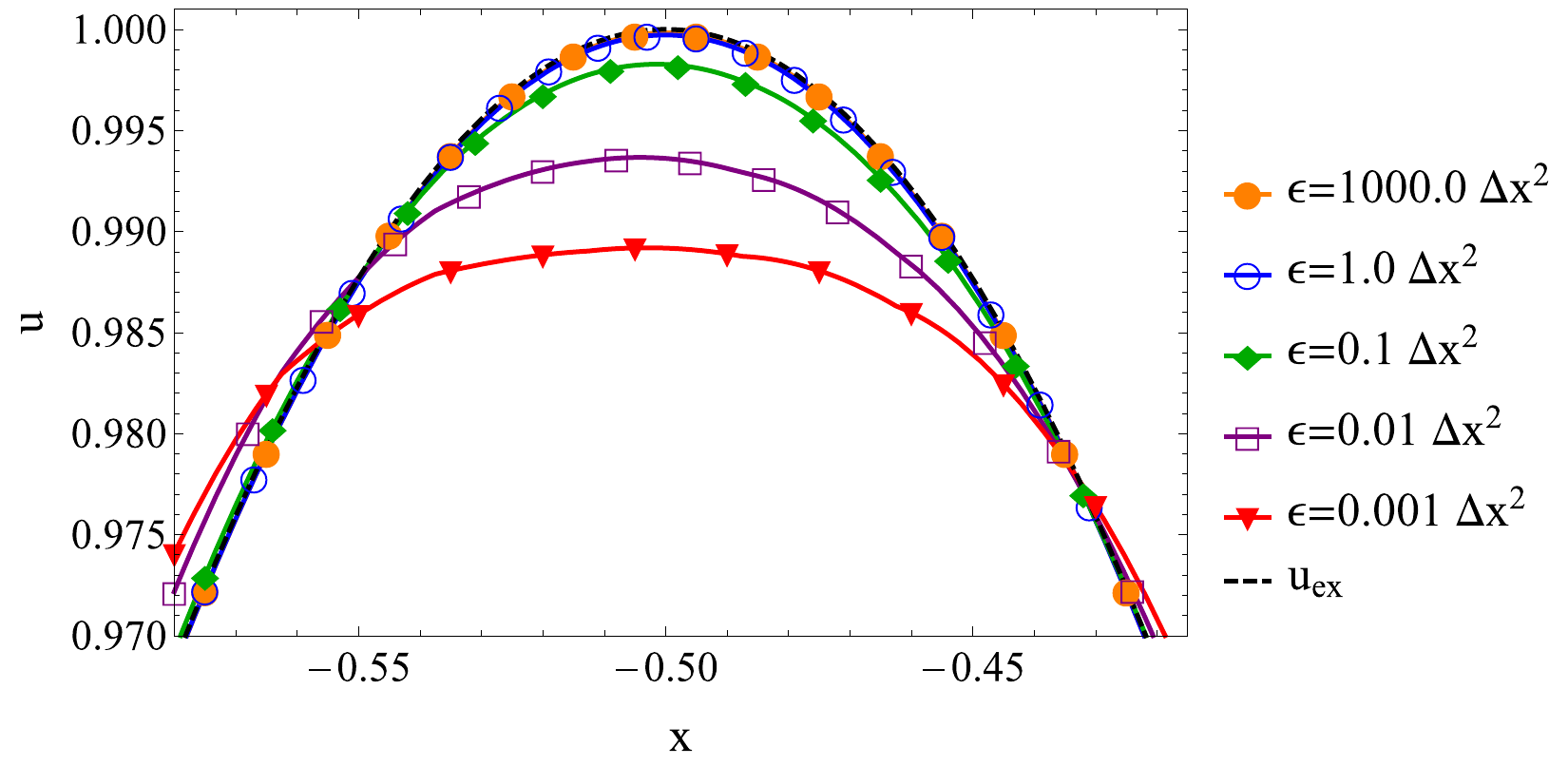}
		\caption{Zoom of the solution of the preliminary test case, Eq.~\eqref{eq:preliminaryTest}. This has been obtained with the WENO-YC scheme, $\varepsilon=C\,\Delta x^2,\; C\in\lbrace 10^3, 1.0, 10^{-1},10^{-2},10^{-3}\rbrace$ with $n=80$ grid cells and $t_\text{end}=1.0$.}
		\label{fig:zoomWENOYCdiffEps}
	\end{figure}
	In this test, we want to point out the importance of the choice of $\varepsilon$ for WENO-YC as we have done for $H^{\text{(c)}}_\text{CT}$ and WENO-JS in Sec.~\ref{subsec:smoothnessIndicator} and \ref{subsec:WENOsmoothness}, respectively. We set $\varepsilon=C\,\Delta x^2$. According to Eq.~\eqref{eq:weightsYC}, $C=\max_{x\not\in\Omega_d}(\|u_0^2\|, \|(u_0^\prime)^2\|)$. However, in this test case we set $C\in\lbrace 10^3, 1.0, 10^{-1},10^{-2},10^{-3}\rbrace$ to study the influence of the coefficient. The results for the repetition of the test
	%
	\begin{subequations}
		\label{eq:preliminaryTest}
		\begin{align}			
			u_t+u_x &=0\\
			u_0(x)&=\sin(\pi x),\quad x\in [-1,1]
		\end{align}
	\end{subequations}	
	with $n=80$ cells and $t_\text{end}=1.0$ is depicted in Fig.~\ref{fig:zoomWENOYCdiffEps} for different values of $\varepsilon$. One can clearly see the loss in accuracy if the constant for $\varepsilon$ is chosen too small. Note that the correct value for $\varepsilon$, as proposed by Yamaleev and Carpenter \cite{YamaleevCarpenter2009}, is $\varepsilon=\max(1,\,\pi^2)\;\Delta x^2=\pi^2\;\Delta x^2$. This indicates that for smooth functions a coefficient which is too small decreases the quality of the approximation. This can also be observed in Fig.~\ref{fig:WENOYCerror} which shows the $L_1$- and $L_\infty$-errors of the solution. We can see that the solution with $\varepsilon=1000\;\Delta x^2$ is $3^\text{rd}$-order-accurate in both norms starting at the lowest resolution. While the solution with $\varepsilon=1.0\;\Delta x^2$ is still directly $3^\text{rd}$-order-accurate in the $L_1$-norm starting from $n=20$ grid cells, we observe that the smaller the coefficient of $\varepsilon$, the larger $n$ must be chosen so that $3^\text{rd}$ order is achieved.
	
	Of course, the same conclusion holds true for the choice of $\alpha$ in $H^{\text{(c)}}_{3\text{L}}$. This test thus shows that we need to carefully evaluate Eq.~\eqref{eq:weightsYC} and \eqref{eq:eta}. A misinterpretation of the coefficients may lead to results which are significantly worse than what the scheme is capable to achieve.
	\begin{figure}[h!]
		\centering
		\begin{subfigure}{.49\textwidth}
			\includegraphics[width=\textwidth]{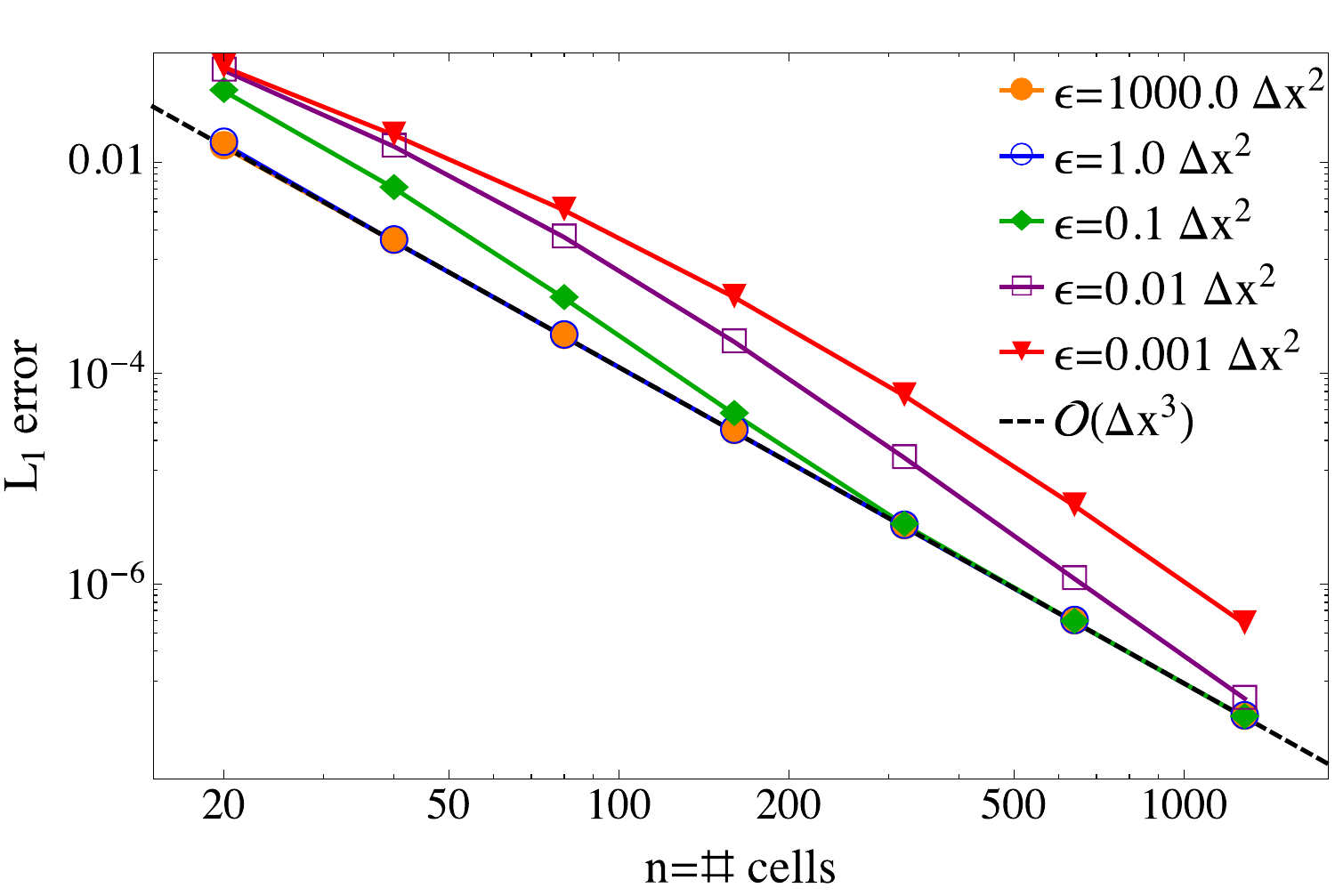}
		\end{subfigure}				
		\hfill
		\begin{subfigure}{.49\textwidth}
			\includegraphics[width=\textwidth]{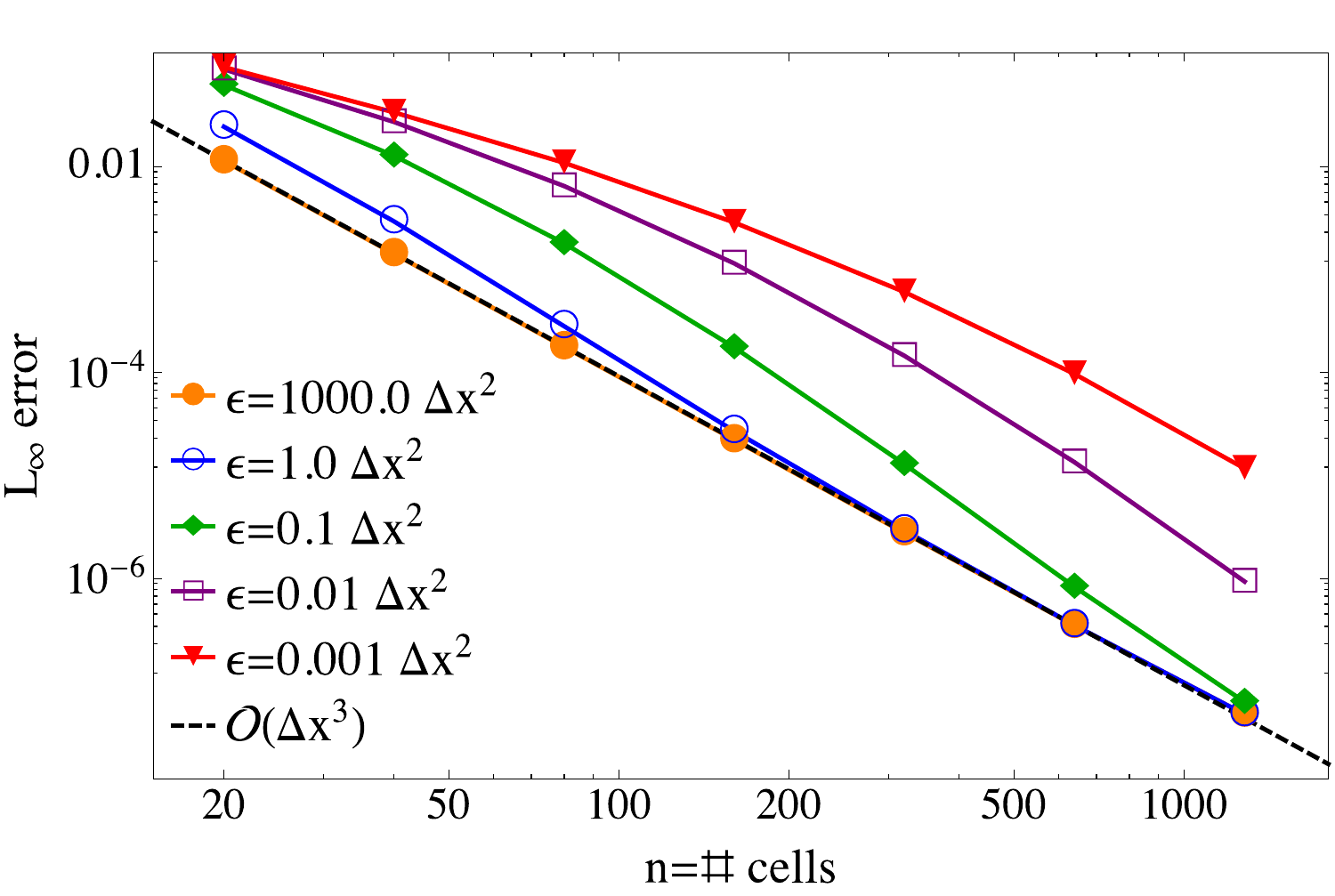}
		\end{subfigure}				
		\caption{Double logarithmic plot of the $L_1$- and $L_\infty$-errors versus number of grid cells of the solution obtained with $H_\text{WENO-YC}$, advected until $t_\text{end}=1$ with CFL number $\nu=0.9$. Numerical solution for different values of $\varepsilon=C\,\Delta x^2,\; C\in\lbrace 10^3, 1.0, 10^{-1},10^{-2},10^{-3}\rbrace$. Left: $L_1$-error, right: $L_\infty$-error.}
		\label{fig:WENOYCerror}
	\end{figure}
	%
\subsection{Advection Equation with Smooth Initial Condition}\label{subsec:smoothIC}
%
	\begin{figure}[t]
		\centering
		\begin{subfigure}{0.48\textwidth}
		\includegraphics[width=\textwidth]{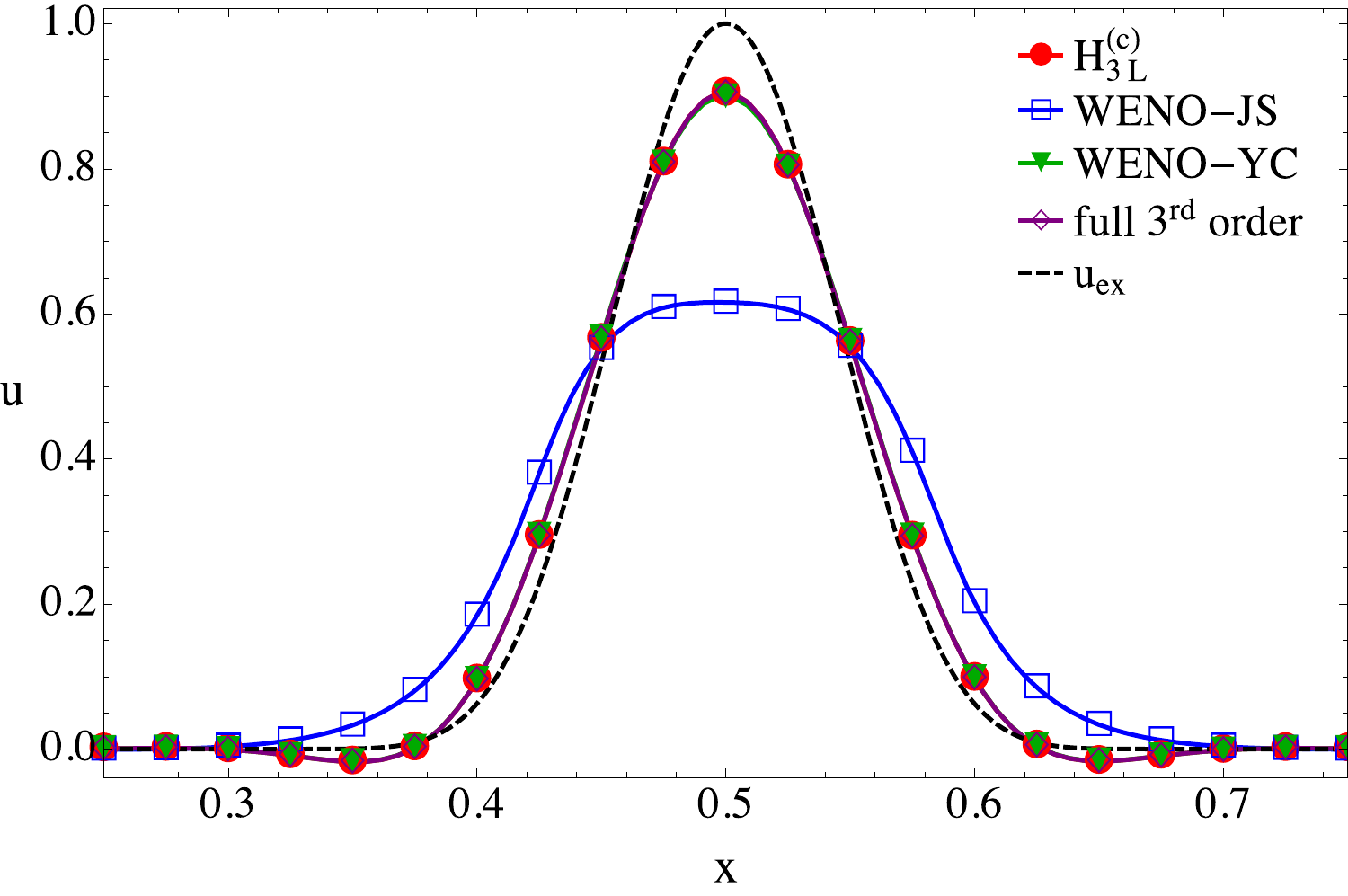}
		\caption{Solution of Eq.~\eqref{eq:advectionEq}, \eqref{eq:ICsmooth} with different numerical schemes, slightly zoomed in.}
		\label{fig:solTest0}
		\end{subfigure}
		\hfill
		\begin{subfigure}{0.48\textwidth}
			\includegraphics[width=\textwidth]{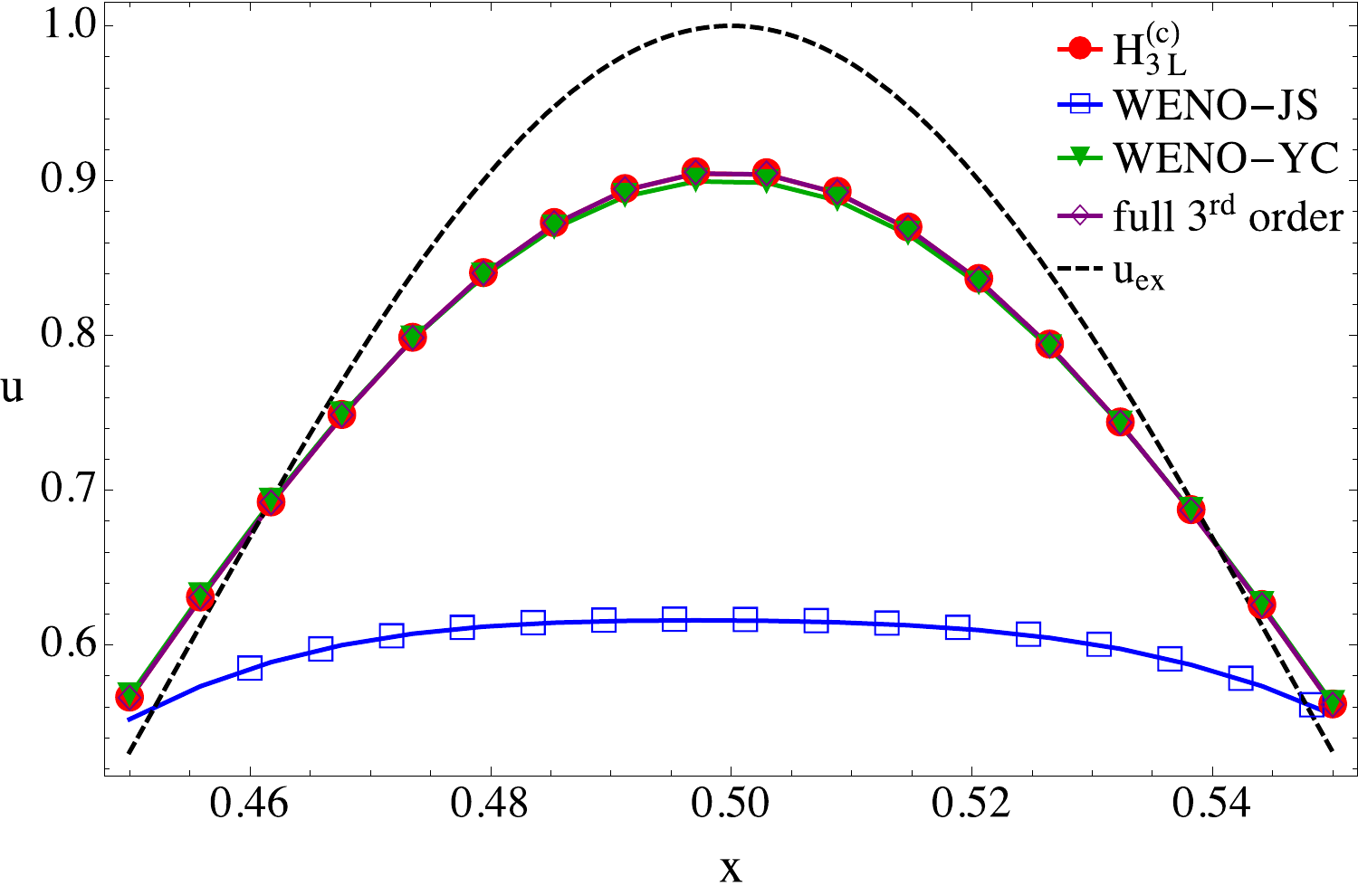}
			\caption{Enlarged view of the maximum of the solution.}
			\label{fig:solTest0zoom}
		\end{subfigure}
		\caption{Results of advection equation Eq.~\eqref{eq:advectionEq} with smooth initial condition, $\nu=0.8$, $n=170$ grid cells and $t_\text{end}=10$, i.e. the solution has been advected 10 times around the domain.}
	\end{figure}
	With this first test case we aim to verify that all considered schemes are $3^{\text{rd}}$-order accurate for smooth solutions. We solve the linear advection equation,
	\vspace{-.3cm}
	\begin{align}
		\label{eq:advectionEq}
		u_t + u_x = 0
		\vspace{-.2cm}
	\end{align}
	\vspace{-.2cm}
	with the smooth initial condition
	\begin{align}
		\label{eq:ICsmooth}
		u_0(x)=\begin{cases}
			(0.5+0.5\cos(5\pi (x-0.5)))^4\quad&\text{if}\quad 0.3 \leq x\leq 0.7 \\
			0\quad &\text{else}
		\end{cases}
	\end{align}
	and periodic boundary conditions. The computational domain is $[0, 1]$, the CFL number $\nu=0.8$ and the solution is advected until $t_\text{end}=10$. The spatial resolution is the sequence of refined uniform grids with $n=20, 40, 50, 100, 120, 170, 200, 300, 500, 700, 1000, 1500, 3000$ cells. For WENO-YC, according to Eq.~\eqref{eq:weightsYC}, we set $\varepsilon = 20.67 \;\Delta x^2$ with $\Delta x = 1/n$. For the FV limiter function Eq.~\eqref{eq:HnewComb} and \eqref{eq:eta} we fix $\alpha=\max_{x \in \Omega \backslash \Omega_d} |u_{0}^{\prime\prime}(x)| =493.48$.	
	
	Fig.~\ref{fig:solTest0} and \ref{fig:solTest0zoom} show zooms of the solution at $t_\text{end}$ with different numerical schemes on a $170$-cell grid. It can be seen that $H_{3\text{L}}^{(c)}$ and WENO-YC, as well as the full-$3^\text{rd}$-order scheme, perform much better than the conventional WENO-JS scheme in terms of accuracy. This can also be observed in Fig.~\ref{fig:errorTest0} which shows the $L_1$- and $L_\infty$-errors	
	\begin{figure}[h!]
		\centering
		\setlength{\abovecaptionskip}{10pt plus 3pt minus 2pt}
		\begin{subfigure}{0.49\textwidth}
			\includegraphics[width=\textwidth]{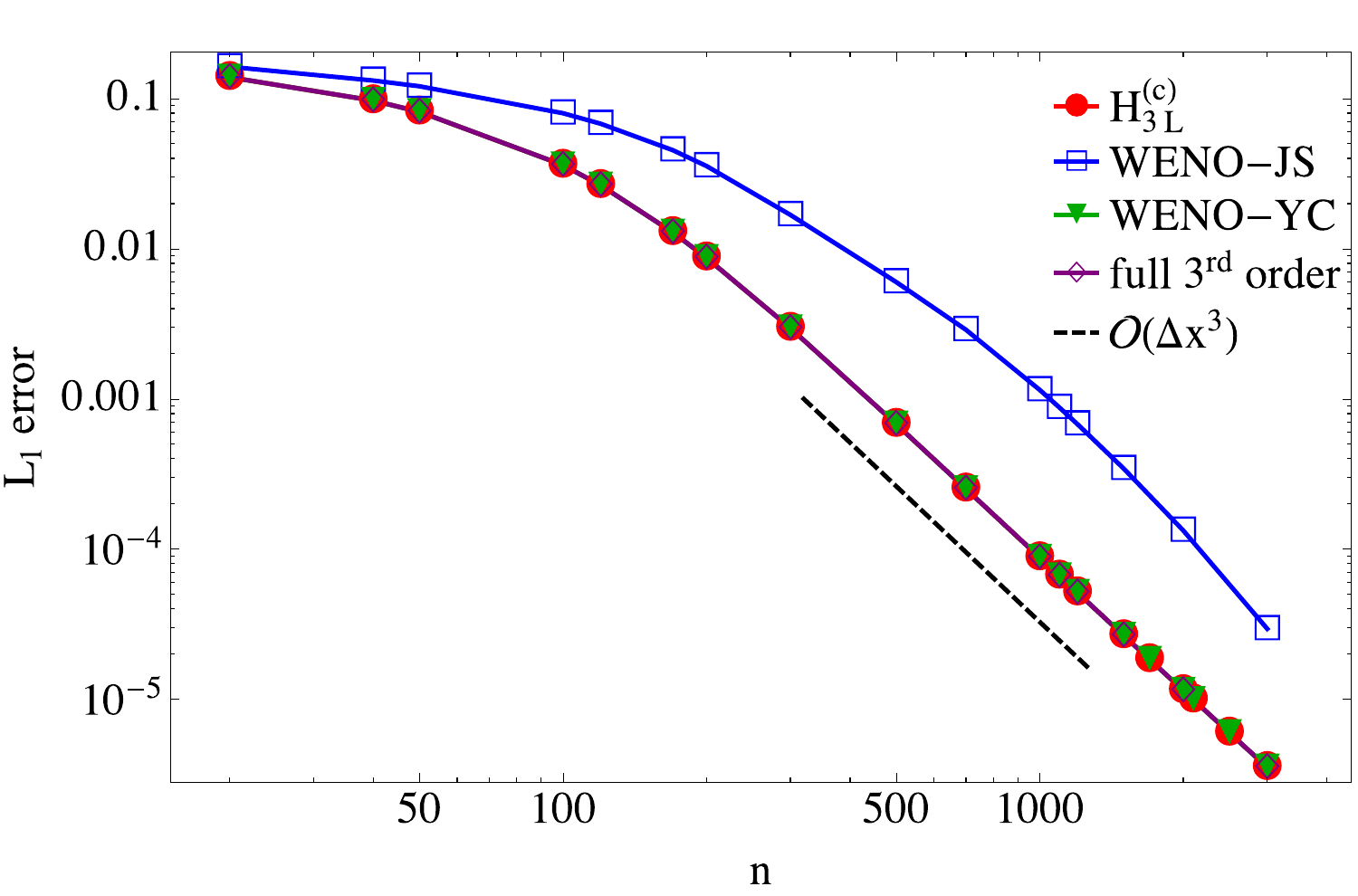}
			\label{fig:L1errorTest0}
		\end{subfigure}		
		\hfill
		\begin{subfigure}{0.49\textwidth}
			\includegraphics[width=\textwidth]{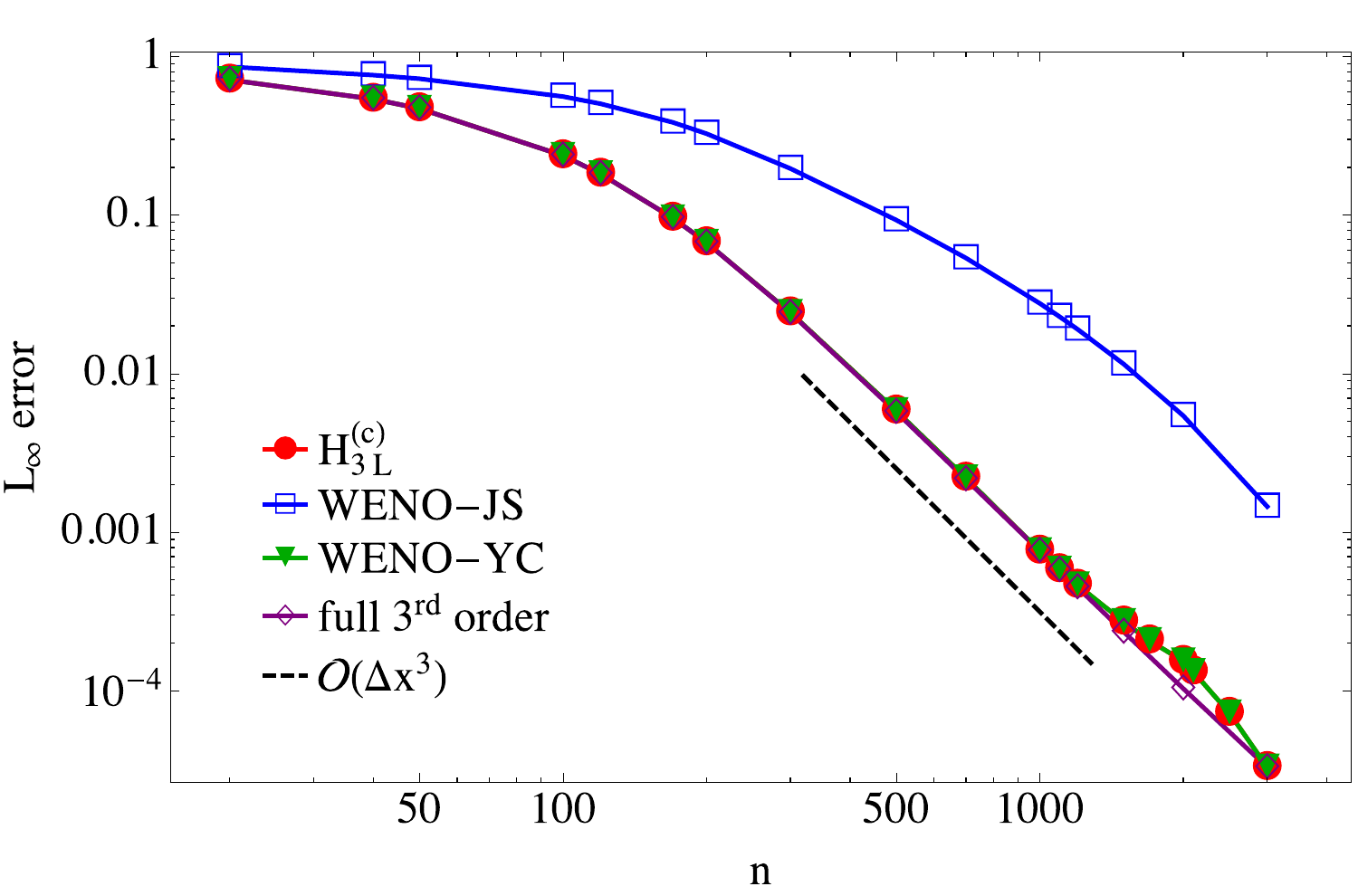}
			\label{fig:LinfErrorTest0}
		\end{subfigure}		
		\caption{Double-logarithmic plot of the $L_1$- and $L_\infty$-errors vs. number of grid cells for the advection equation with smooth initial condition, Eq.~\eqref{eq:advectionEq}, \eqref{eq:ICsmooth}.}
		\label{fig:errorTest0}
	\end{figure}
	obtained at $t_\text{end}$.
	The convergence rates of our proposed FV scheme, WENO-YC, and the unlimited $3^\text{rd}$-order scheme reach $3^{\text{rd}}$ order starting at $n=150$ grid cells whereas WENO-JS shows this order of convergence only for very large numbers of cells and with a larger error constant.
	%
\subsection{Advection Equation with Discontinuous Initial Condition}\label{subsec:discIC}
	%
	In this section we want to discuss the behavior of the numerical schemes for solutions containing discontinuities. Therefore, we consider the advection equation, Eq.~\eqref{eq:advectionEq} with $a=1$ and square wave initial condition
	\begin{align}
	\label{eq:discontIC}
		u_0(x)= \begin{cases}
		1 \quad \text{for}\;\; -0.5<x<0.5 \\
		0\quad \text{else.}
		\end{cases}
	\end{align}
	The computational domain is $[-1, 1]$, the CFL number $\nu=0.8$ and the solution is advected until $t_\text{end}=10$, which corresponds to $5$ periods in time. Due to the large gradients contained in the initial condition, the limiter functions have to take effect in order to avoid spurious oscillations to appear. Solving this test case with the full-$3^\text{rd}$-order reconstruction Eq.~\eqref{eq:3rdOrder} generates oscillations. This is the reason the FV limiter functions presented in Sec.~\ref{sec:LimO3} restrict the reconstruction to $1^\text{st}$-order in these situations and WENO3 reduces to Eq.~\eqref{eq:HwenoBigDeltas}. The WENO-YC parameter $\varepsilon$ is given by $\Delta x^2$ (see Eq.~\eqref{eq:epsYC}) and $H_{3L}^{(c)}$ reduces to $H_{3L}$ because $\alpha=\max_{x \in \Omega \backslash \Omega_d} |u_{0}^{\prime\prime}(x)| = 0$.
	
	This test case nicely shows the already mentioned drawback of the definition of $\varepsilon$ in WENO-YC, Eq.~\eqref{eq:epsYC}, namely the coefficient of $\varepsilon$ which is not translationally invariant if the initial condition is shifted. To point this out, a second test case has been chosen, where the initial condition has simply been shifted by $+100$, i.e 
	\begin{align}
	\label{eq:discontICplus100}
		\text{IC}_{+100}:\;\, u_{0, +100}(x)= \begin{cases}
		101 \quad \text{for}\;\; -0.5<x<0.5 \\
		100\quad \text{else.}
		\end{cases}
	\end{align}
	When applying WENO-JS to this new initial condition, the solution is does not show more oscillations, because $\varepsilon$ is fixed to $10^{-6}$. In fact, the solution is the same as for $u_0(x)$, only shifted by $+100$. Also, in the proposed FV limiter, the constant $\alpha=\max_{x \in \Omega \backslash \Omega_d} |u_{0, +100}^{\prime\prime}(x)| = 0$ does not change. Thus, the scheme yields the exact same results, shifted by $+100$. However, for WENO-YC, $\varepsilon$, as given by Eq.~\eqref{eq:epsYC}, is no longer $\Delta x^2$ but $20201\,\Delta x^2$. The higher value of $\varepsilon$ leads to augmented oscillations in the solution, as can be seen in Fig.~\ref{fig:solTest3}. Here, for better comparison, the solution $u_{+100}(x,t_\text{end})$ of the test with shifted initial condition, which lies in the range $[99.9, 101.1]$, has been shifted to the range $[-0.1, 1.1]$. Thus, both test cases, Eq.~\eqref{eq:discontIC} and \eqref{eq:discontICplus100} lie between $-0.1$ and $1.1$ in the plots and the magnitude of the oscillations can be compared. 
	\begin{figure}[t]
		\centering
		\begin{subfigure}{0.49\textwidth}
			\includegraphics[width=\textwidth]{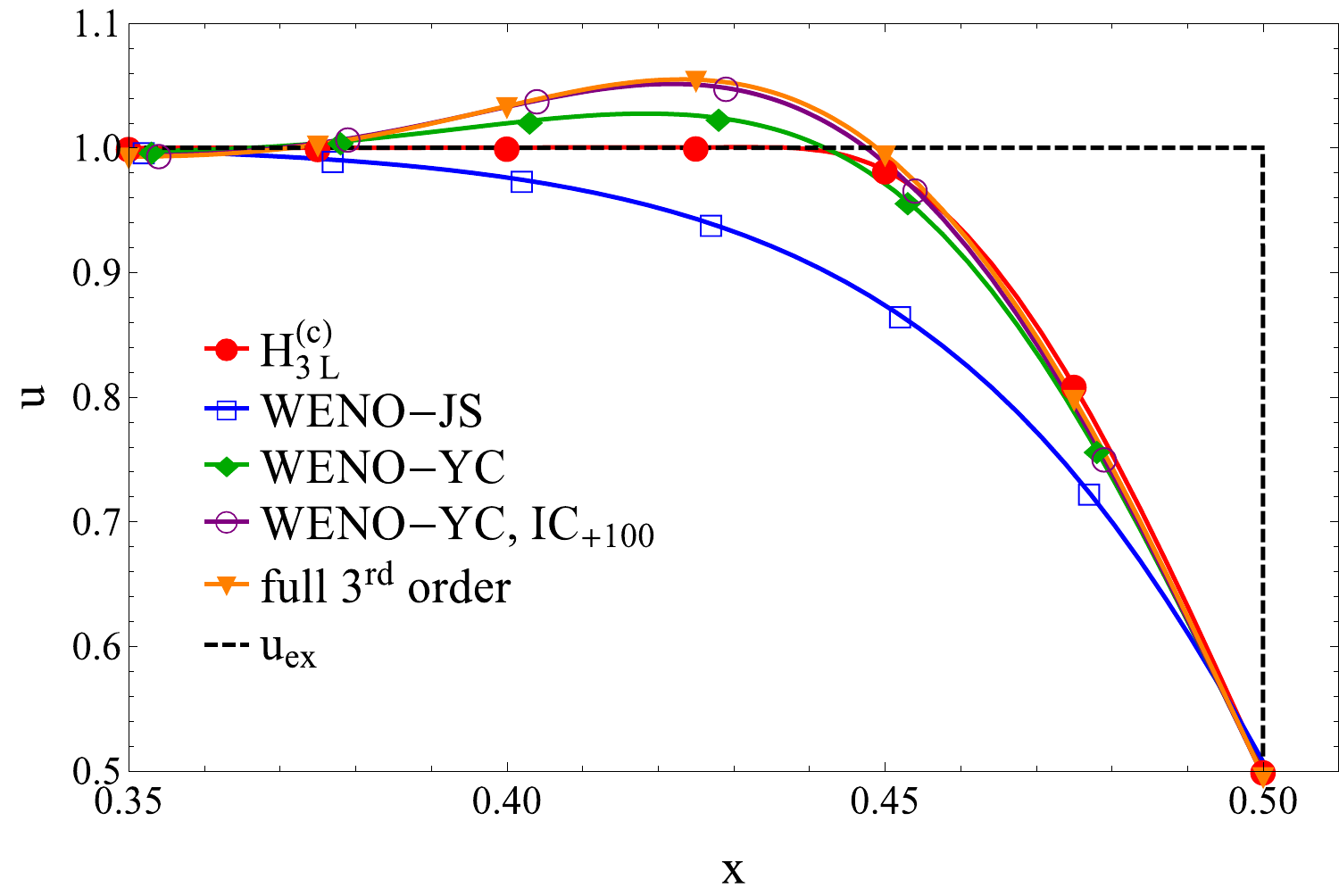}
		\end{subfigure}
		\begin{subfigure}{0.49\textwidth}
			\includegraphics[width=\textwidth]{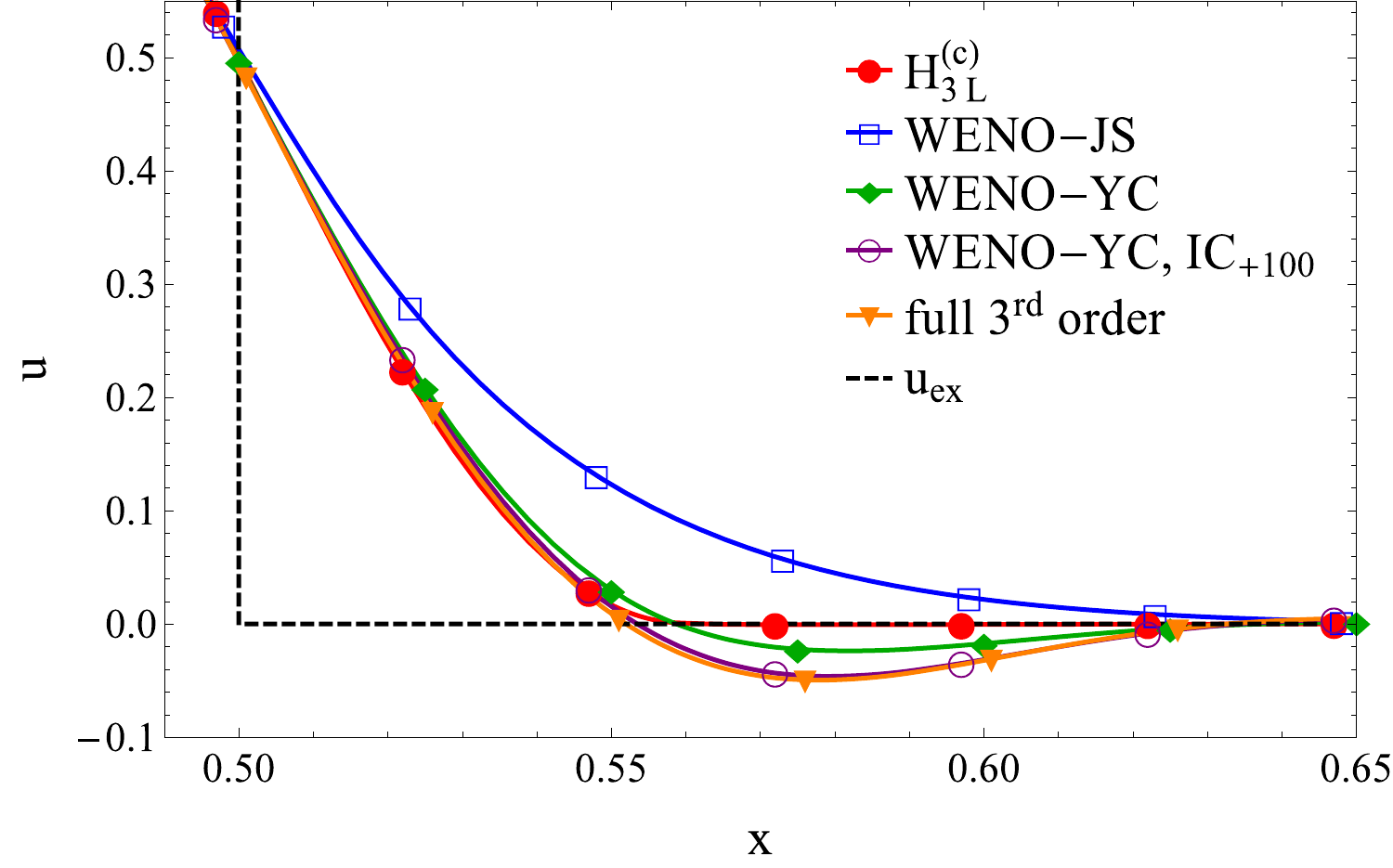}
		\end{subfigure}			
		\caption{Enlarged view of the discontinuities of the solution of Eq.~\eqref{eq:advectionEq}, \eqref{eq:discontIC} with different numerical schemes using $n=320$ cells. $\nu=0.8, t_\text{end}=10$.}
		\label{fig:solTest3}
	\end{figure}
	As seen in Fig.~\ref{fig:solTest3}, the original WENO3 scheme does not cause any oscillation but it is rather dissipative. Our proposed FV limiter does not produce overshoots either. Additionally, it approximates the sharp structure of the initial condition better than WENO-JS. WENO-YC with correctly chosen parameter approximates the discontinuity almost as good as $H^{\text{(c)}}_{3\text{L}}$, however, it causes oscillations. These are even larger for the badly chosen $\varepsilon$.
	\begin{figure}[t]
		\centering
		\begin{subfigure}{0.49\textwidth}
			\includegraphics[width=\textwidth]{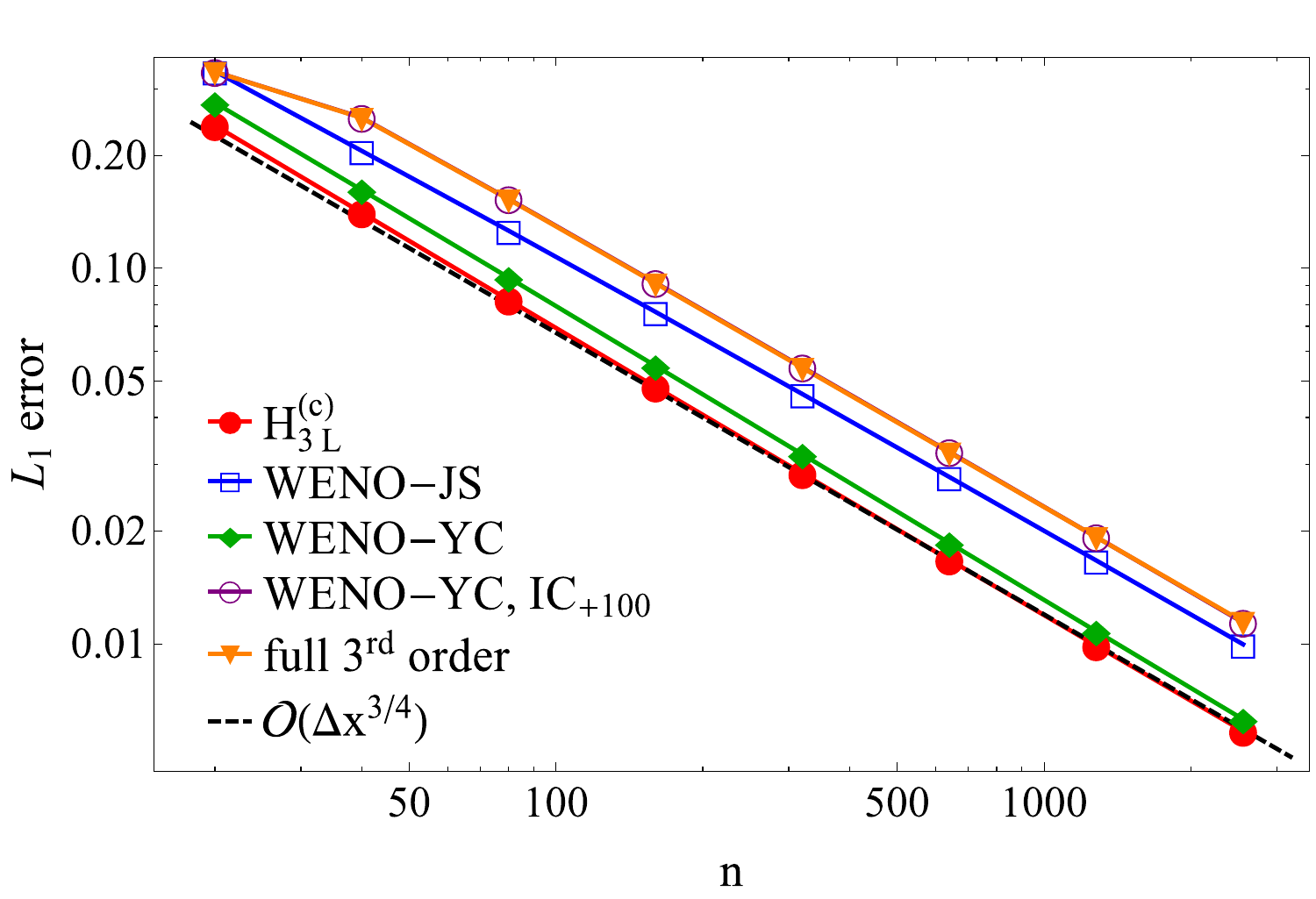}
			\caption{Double-logarithmic plot of the $L_1$-error vs. number of grid cells.}
			\label{fig:errorTest3}
		\end{subfigure}
		\begin{subfigure}{0.49\textwidth}
			\includegraphics[width=\textwidth]{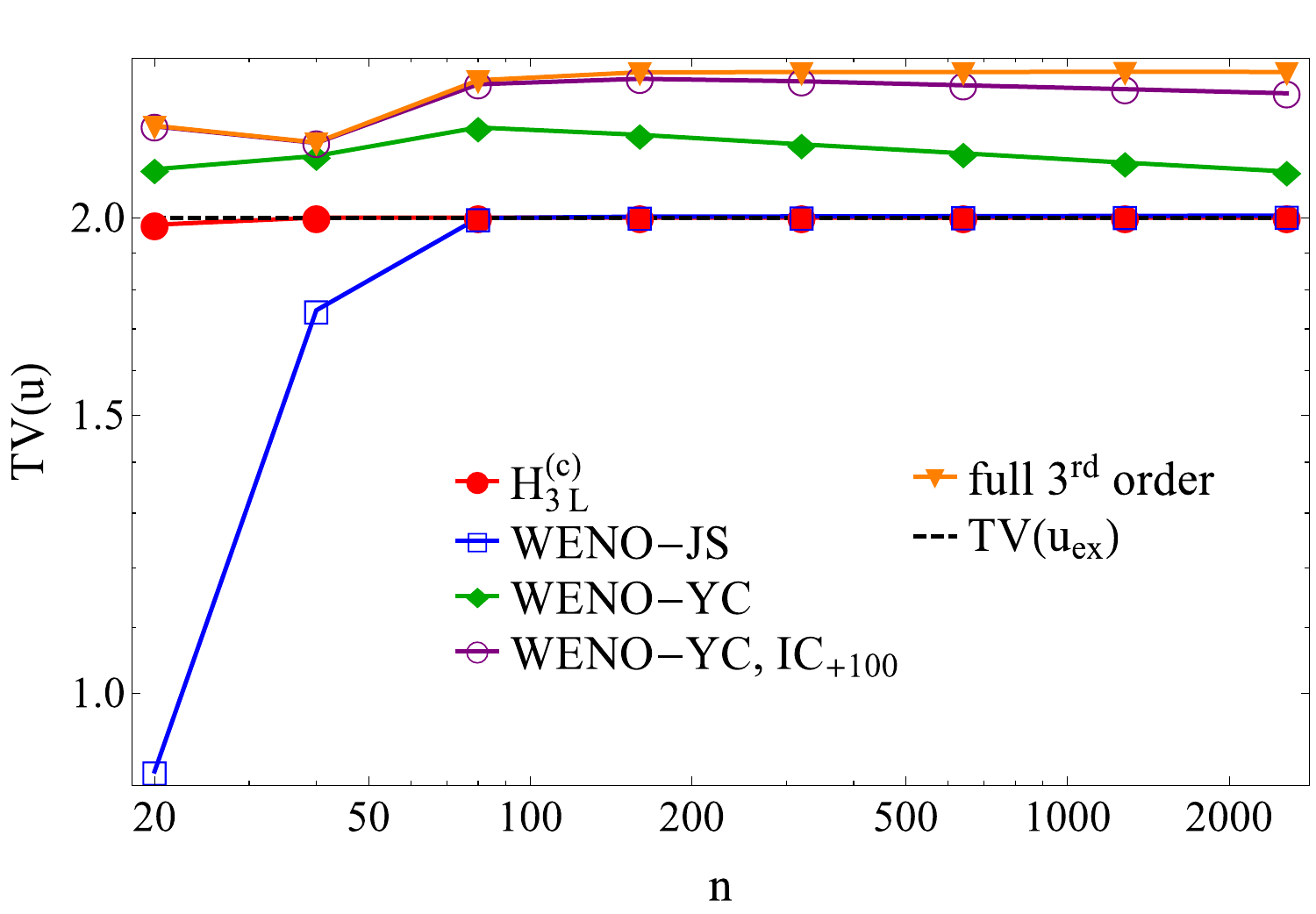}
			\caption{Total Variation vs. number of grid cells of the different schemes.}
			\label{fig:TVTest3}
		\end{subfigure}			
		\caption{Results of advection equation Eq.~\eqref{eq:advectionEq} with discontinuous initial condition., $\nu=0.8$ and $t_\text{end}=10$, i.e. the solution has been advected 10 around the domain.}
	\end{figure}
	This behavior can also be observed in Fig.~\ref{fig:errorTest3} and \ref{fig:TVTest3} which show the $L_1$-error and the total variation (TV), respectively. 
	The best order of convergence we can expect from a $3^\text{rd}$-order scheme and a solution containing discontinuities is $3/4$. This can be shown by the Fourier analysis of the modified equation with self-similar initial condition, cf. \cite{LeVeque2002}. \\
	Even though, both solutions computed with WENO-YC cause oscillations, they obtain the order of accuracy $3/4$. The more dissipative WENO-JS scheme is also of order $3/4$ but with a larger error constant. Among the tested schemes, the best error constant is obtained with our proposed FV limiter function. The total variation of all schemes represents their behavior as seen in the solution. WENO-JS attains the TV of the exact solution $TV(u_{ex})$ from below, meaning that is does not cause overshoots at all, whereas WENO-YC is larger than $TV(u_{ex})$ for all spatial discretizations. $H^{\text{(c)}}_{3\text{L}}$ is closest to TV(u$_\text{ex}$) and lies never above TV(u$_\text{ex}$) for all time steps, i.e. it does not cause oscillations at any time.
	%
	\subsection{Initial Condition Containing Smooth and Non-Smooth Features}\label{subsec:difficultIC}
	%
	\begin{figure}[b]
		\centering
		\includegraphics[scale=0.38]{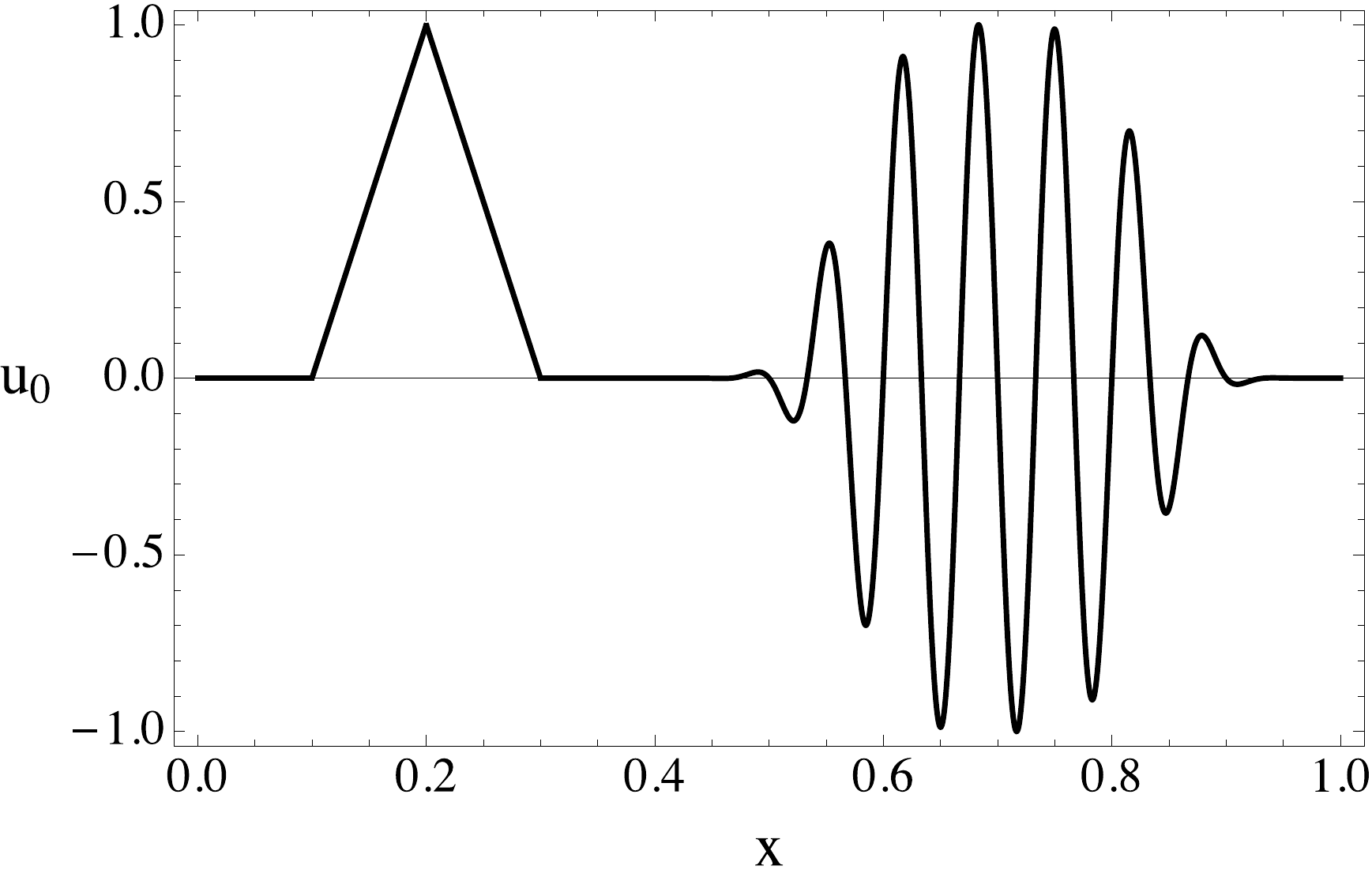}
		\caption{Initial condition \eqref{eq:test1IC} containing smooth and non-smooth features.}
	\end{figure}
	We consider the same setup as in Sec.~\ref{subsec:smoothIC} with  CFL number $\nu = 0.8$, $t_\text{end}=10.0$ and initial condition
	\begin{align}
		\label{eq:test1IC}
		u_0(x)=\max \left(\min \left(\frac{x}{0.1}-2, -\left(\frac{x}{0.1}-2\right)\right) + 1, 0\right)+\exp \left(-	\left(\frac{x-0.7}{0.15}\right)^4\right) \sin (30 \pi  x).
	\end{align}
	In this problem, we want to test how accurate the different schemes resolve small features in a larger setting of a more complex solution. The spatial discretizations are $n=20\cdot 2^i,\; i=0,\ldots, 7$ grid cells. For WENO-YC, we set $\varepsilon = 1042.83 \;\Delta x^2$, $\Delta x = 1/n$, according to Eq.~\eqref{eq:epsYC}. For the FV limiter function we fix $\alpha=\max_{x \in \Omega \backslash \Omega_d} |u_{0}^{\prime\prime}(x)|=8887.87$. We run a third case with WENO-YC but setting $\varepsilon = \Delta x^2$ to show the difference in performance as described in Sec.~\ref{subsec:WENOsmoothness} and \ref{subsec:preliminaryTest}. This test case corresponds to the weight functions proposed in \cite{ArandigaMarti2014, Kolb2014}, where $\varepsilon=K\,\Delta x^2$ was used with $\;K=1$. Furthermore, a test with the full-$3^\text{rd}$-order reconstruction $H_3$, Eq.~\eqref{eq:3rdOrder2param}, was performed and compared to the numerical schemes discussed in this paper.
	
	Fig.~\ref{fig:test1} shows a close up view of two significant regions of the solution with $n=640$ grid cells. We observe that our proposed FV limiter function and WENO-YC with correctly-chosen parameter $\varepsilon$ perform much better than WENO-YC with $\varepsilon = \Delta x^2$. The results of these two schemes are very close to the ones of the full-$3^\text{rd}$ order reconstruction. Also, all four schemes outperform the conventional WENO-JS scheme.
	\begin{figure}[t]
	\begin{subfigure}{0.4\textwidth}
		\includegraphics[width=\textwidth]{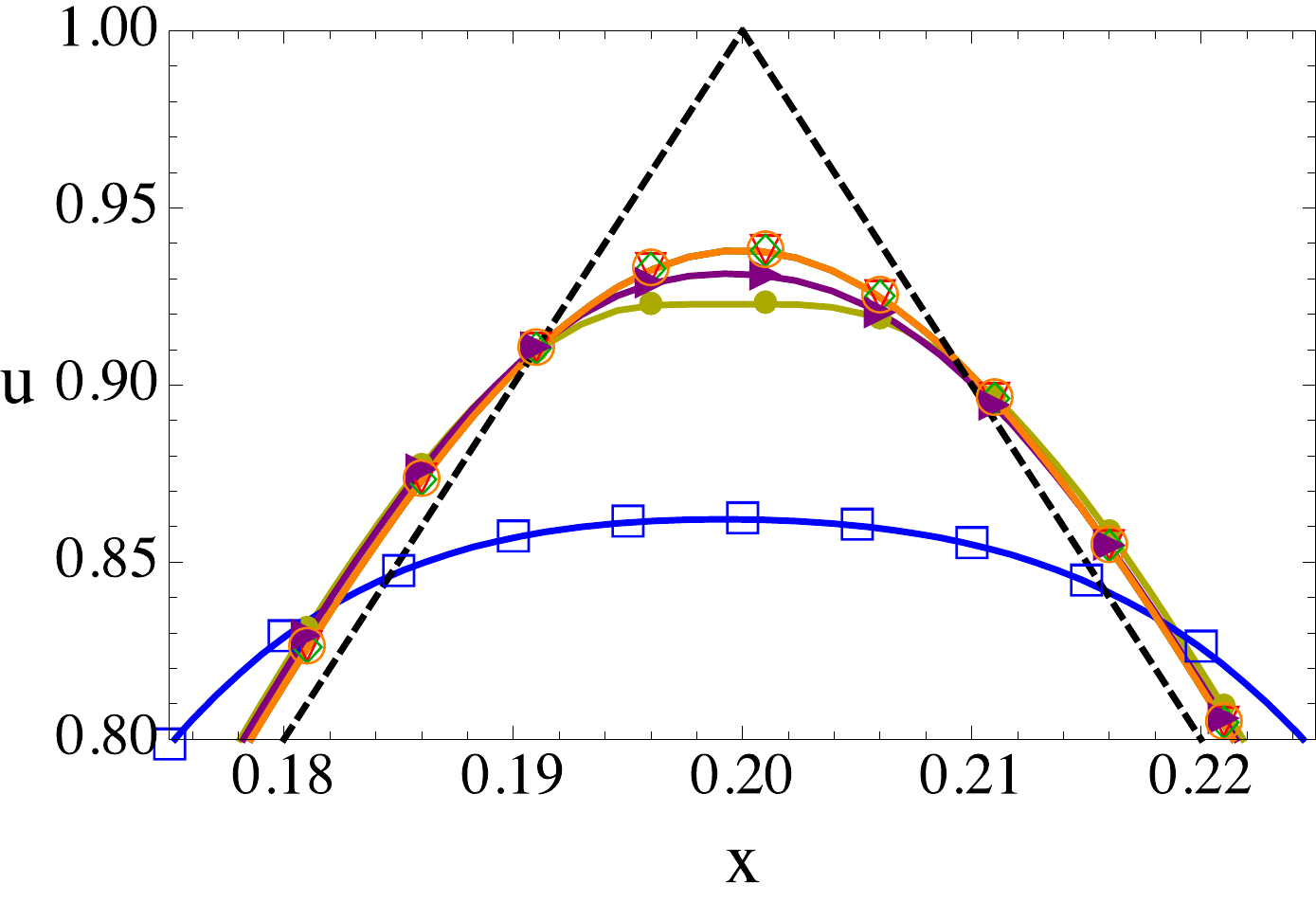}
   		\end{subfigure}
   		\begin{subfigure}{0.59\textwidth}
   			\centering
   			\includegraphics[width=\textwidth]{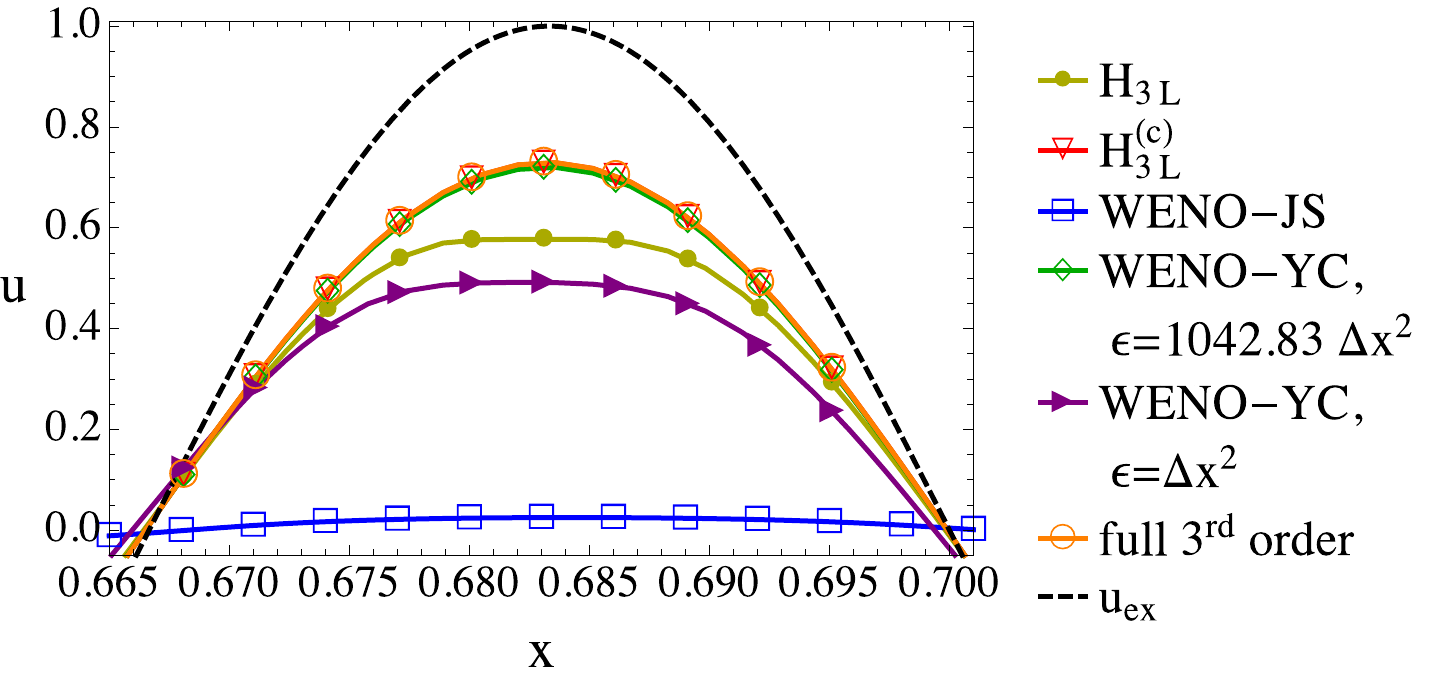}
   		\end{subfigure}
   		\caption{Results of different schemes for the advection equation with more elaborate initial condition: Zoom of two significant regions of the solution with $n=640$ grid cells at $t_\text{end}=10, \nu=0.8$.}
   		\label{fig:test1}
	\end{figure}   	
	This can also be seen in Fig.~\ref{fig:test1L1error}, which shows the double-logarithmic plot of the $L_1$- and $L_\infty$-errors versus the number of grid cells in the smooth part of the solution, $x \in [0.4, 1]$. The solution cannot be accurately represented with few grid cells by any of the treated schemes. Even the full $3^{\text{rd}}$-order reconstruction does not resolve the details and therefore has a large error constant. As soon as a reasonable space discretization is reached, the order of convergence reaches $3^{\text{rd}}$ order if only the range $x \in [0.4, 1]$, i.e. the smooth part of the solution, is regarded. Solely WENO-JS does not reach this order even at the highest resolution. 
	\begin{figure}[b]
		\centering
		\includegraphics[width=0.7\textwidth]{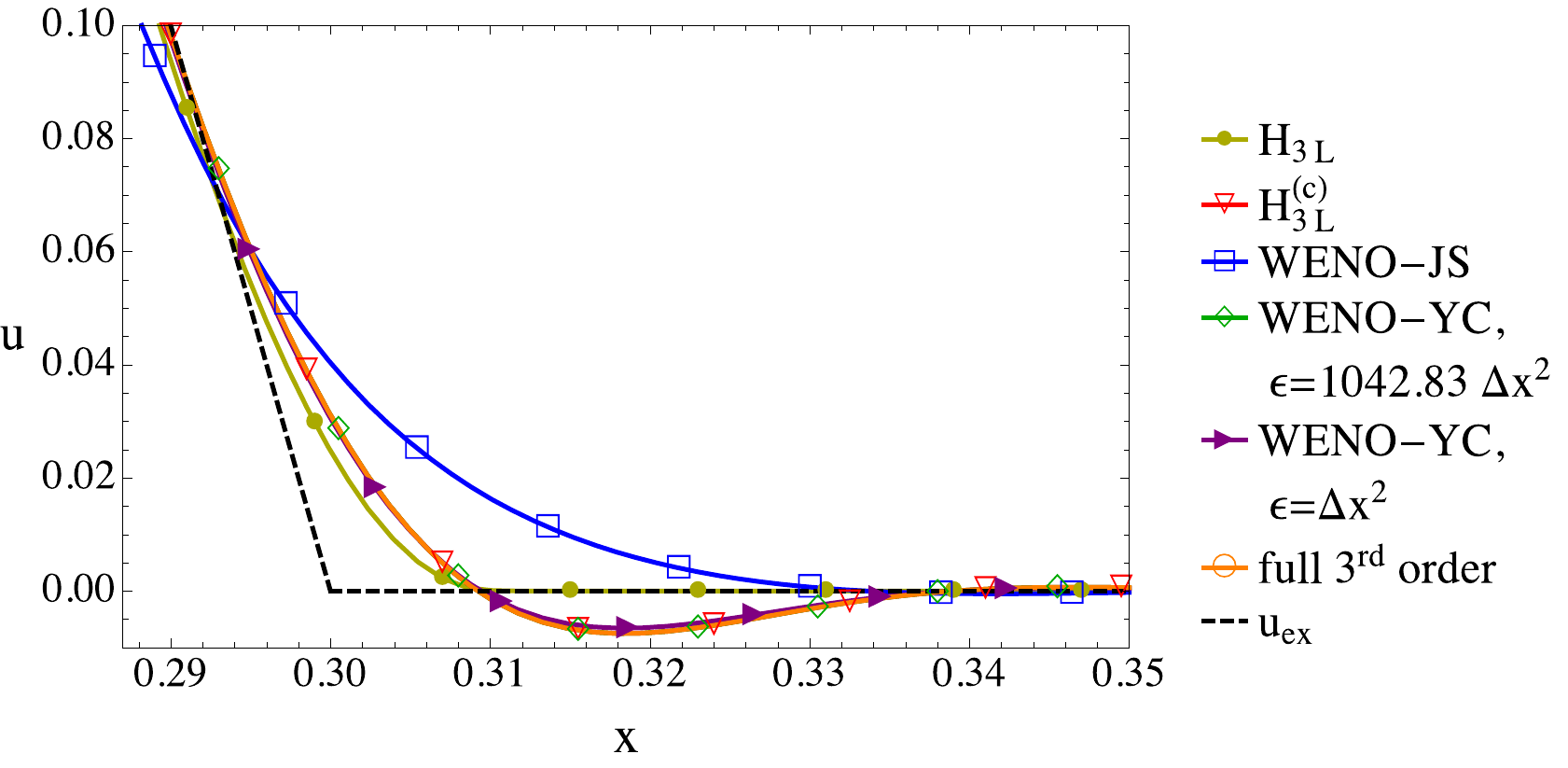}
		\caption{Zoom of the kink of the solution with $n=640$ grid cells at $t_\text{end}=10, \nu=0.8$.}
		\label{fig:test1zoom3}
	\end{figure}

 	\begin{figure}[t]
	 	\centering	 	
   		\begin{subfigure}{0.49\textwidth}
	 	 	\includegraphics[width=\textwidth]{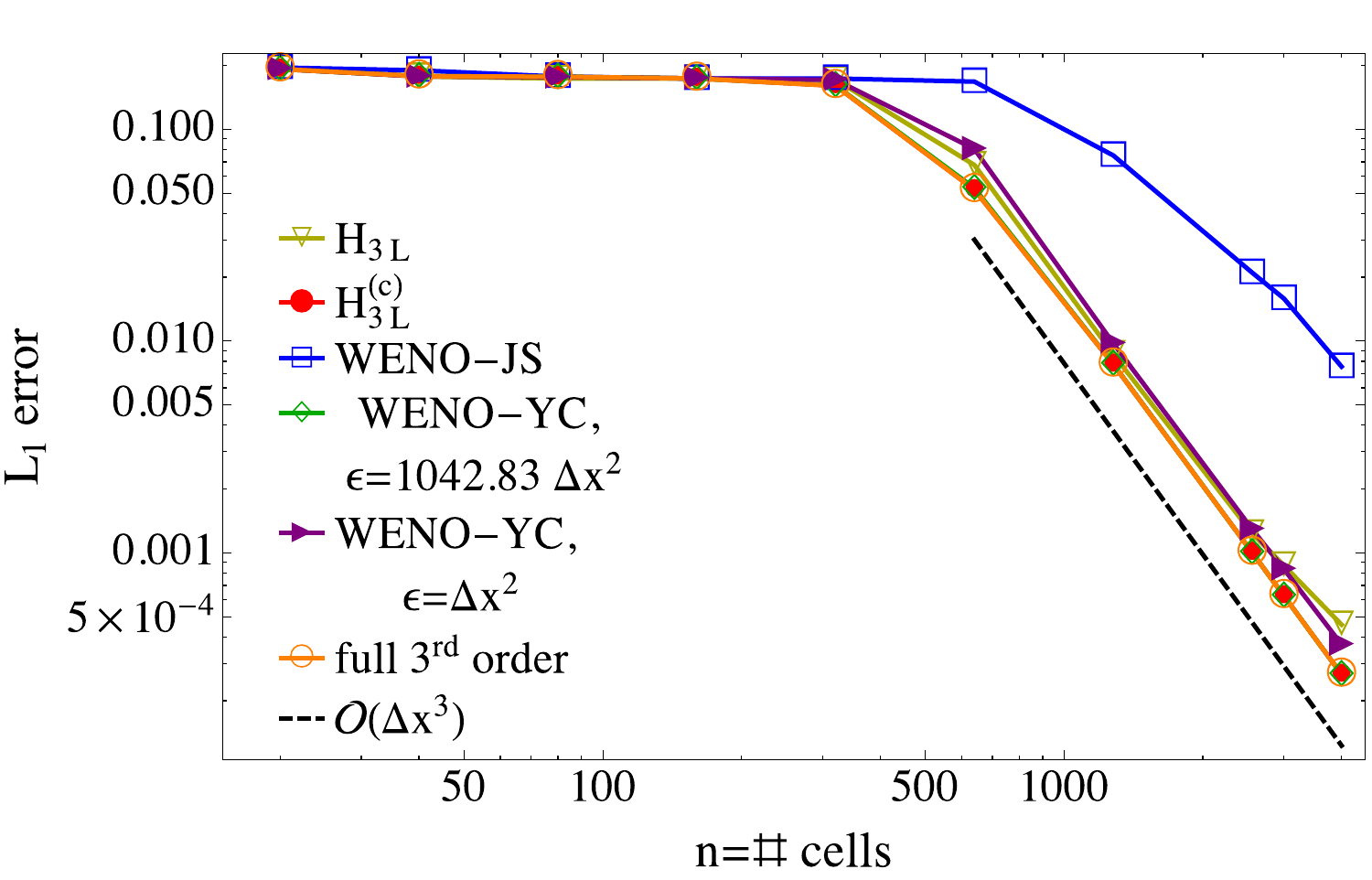}
 	 	\end{subfigure}
		\hfill
   		\begin{subfigure}{0.49\textwidth}
	   		\includegraphics[width=\textwidth]{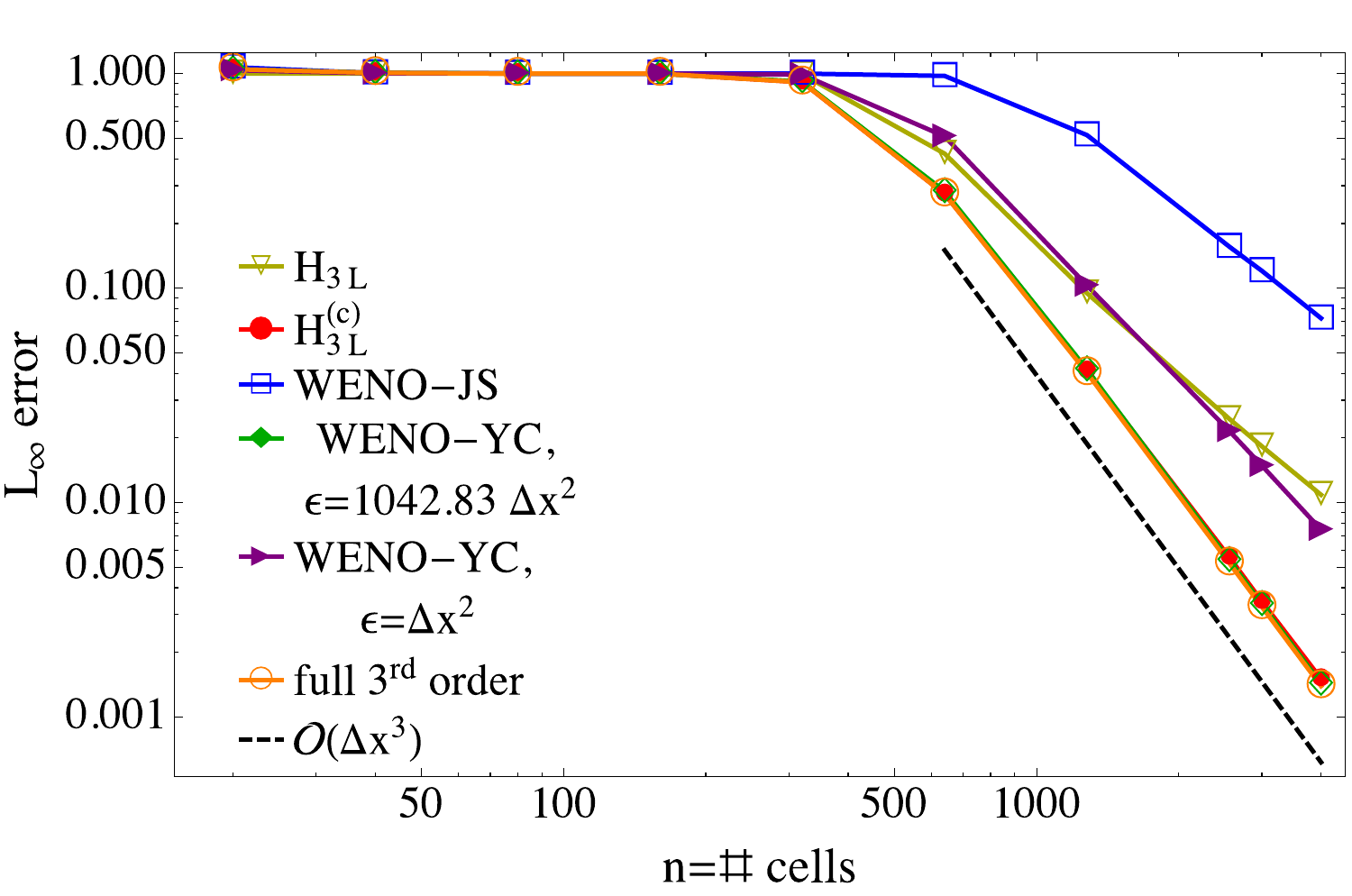}
   		\end{subfigure}
	 	 	\caption{Double-logarithmic plot of the $L_1$- and $L_\infty$-norm of the error vs. number of grid cells in the range $x \in [0.4, 1]$, i.e. the smooth part of the solution of \eqref{eq:advectionEq}, \eqref{eq:test1IC}.}
	 	 	\label{fig:test1L1error}
 	\end{figure} 		 	

	If the error is considered on the whole domain $[0, 1]$, the convergence rate of all schemes degenerates for higher resolutions. This is due to the fact that at higher resolutions the kinks in $x\in [0.1, 0.3]$ become well-resolved and thus recognized as non-smooth. The initial conditions are an interesting test case combining smooth and non-smooth features and therefore testing the capabilities of limiter functions. As shown in Fig.~\ref{fig:test1zoom3}, near the position of the kinks, all schemes - with the exception of WENO-JS and $H_{3\text{L}}$ - generate undershoots. For our proposed FV limiter $H^{(c)}_{3\text{L}}$ this can be explained with the large asymptotic region. Since the maximal second derivative of the initial condition is very large, $\alpha=\max_{x \in \Omega \backslash \Omega_d} |u_{0}^{\prime\prime}(x)|=8887.87$, the region where $H^{(c)}_{3\text{L}}$ reconstructs with full-$3^\text{rd}$ order is relatively large. At the same time, the discrete second derivative of the kinks of the triangle are small compared to the extreme regions of the smooth part, and at these points, the solution is reconstructed using $H_3$ with no limiting. As a result, $H^{(c)}_{3\text{L}}$ causes undershoots, just as $H_3$ does, because this is what is efectively used at the kinks. However, compared to other limiting methods, we can in principle control these undershoots using our FV limiter $H^{(c)}_{3\text{L}}$. By choosing a smaller value for $\alpha$, i.e. a smaller asymptotic region, these undershoots can be completely avoided, as can be seen in the test case with pure discontinuity, cf. Sec.~\ref{subsec:discIC}. An adaptive value for $\alpha$ would therefore eliminate the undershoots in the non-smooth region $x\in [0, 0.4]$ while still resolving the smooth smooth parts in $x\in [0.4, 1]$ with high order accuracy. However, such a local adaptivity would necessarily require to use a wider stencil for the reconstruction, because as shown in Fig.~\ref{fig:smoothVSdisc}, three points can not distinguish between the kink and a strongly curved extremum on a coarse grid. 
%
\subsection{Sod Shock Tube Problem}\label{subsec:Euler}
%
Let us consider Sod's problem, which describes a shock tube containing two different ideal gases at the left and right side of a diaphragm, placed at $x=0$. The density, velocity, and pressure of the gases in the left and right region are given by
\begin{align}
	\begin{pmatrix}
	\rho_L \\
	v_L\\
	p_L
	\end{pmatrix}
	=
	\begin{pmatrix}
	1.0\\0.0\\1.0
	\end{pmatrix},
	\quad 
	\begin{pmatrix}
	\rho_R \\
	v_R\\
	p_R
	\end{pmatrix}
	=
	\begin{pmatrix}
	0.125, \\ 0.0\\0.1
	\end{pmatrix}.
	\label{eq:SodPbIC}
\end{align}
At time $t>0$, the diaphragm is removed and the gases begin to mix. The time evolution is described by the one dimensional Euler equations, 
\begin{subequations}
	\begin{align}
		&\textbf{u}_t + \textbf{f(u)}_x = \textbf{0}\\
		\intertext{with \textbf{u}=$(\rho, \rho u, E)$, the flux function}
		&\textbf{f(u)}=\left(\rho u, \rho u^2 + p, u(E+p) \right)^T,\\
		\intertext{and the equation of state for ideal gases}
		E=\frac{p}{\gamma-1}+\frac{1}{2}\rho u^2,
	\end{align}
	\label{eq:EulerEq}
\end{subequations}
with $\gamma=1.4$. The computational domain is set to $[-2, 2]$ and the test is conducted with $N=100$ grid cells until $T_\text{end}=0.8$. We compare the new limiter function $H_{3\text{L}}^{(c)}$ with the full $3^{\text{rd}}$-order reconstruction and  WENO-JS. WENO-YC is not included in the plots because it produced negative values for pressure, when run with $\varepsilon=2.25$. 

\begin{figure}[t]
   	\centering
   	\begin{subfigure}{0.49\textwidth}
   		\centering
   		\includegraphics[width=\textwidth]{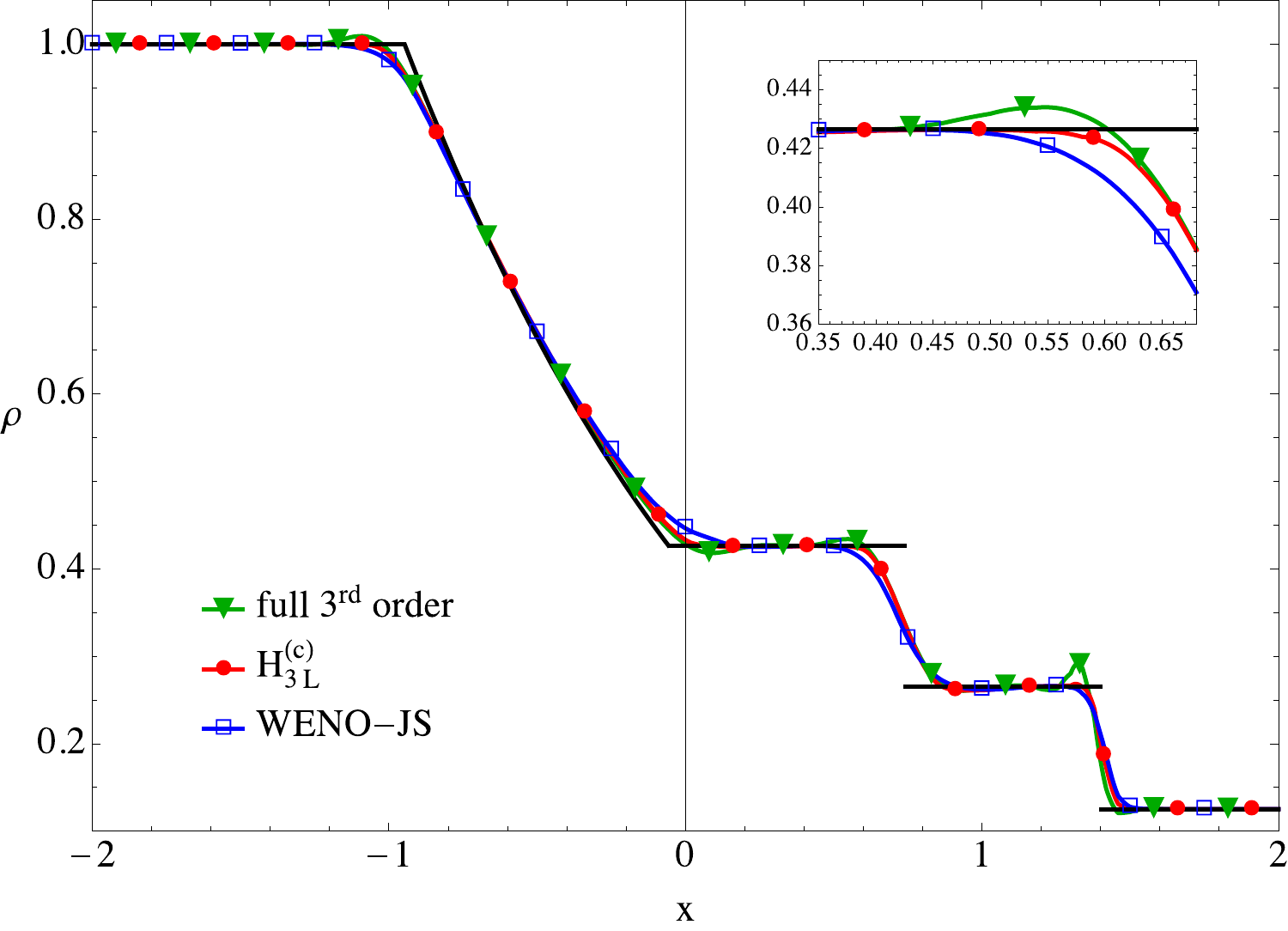}   		
   	\caption{Density profile.}
	\label{fig:SodPbRho}
   	\end{subfigure}
   	\hfill
   	\begin{subfigure}{0.49\textwidth}
   		\centering
   		\includegraphics[width=\textwidth]{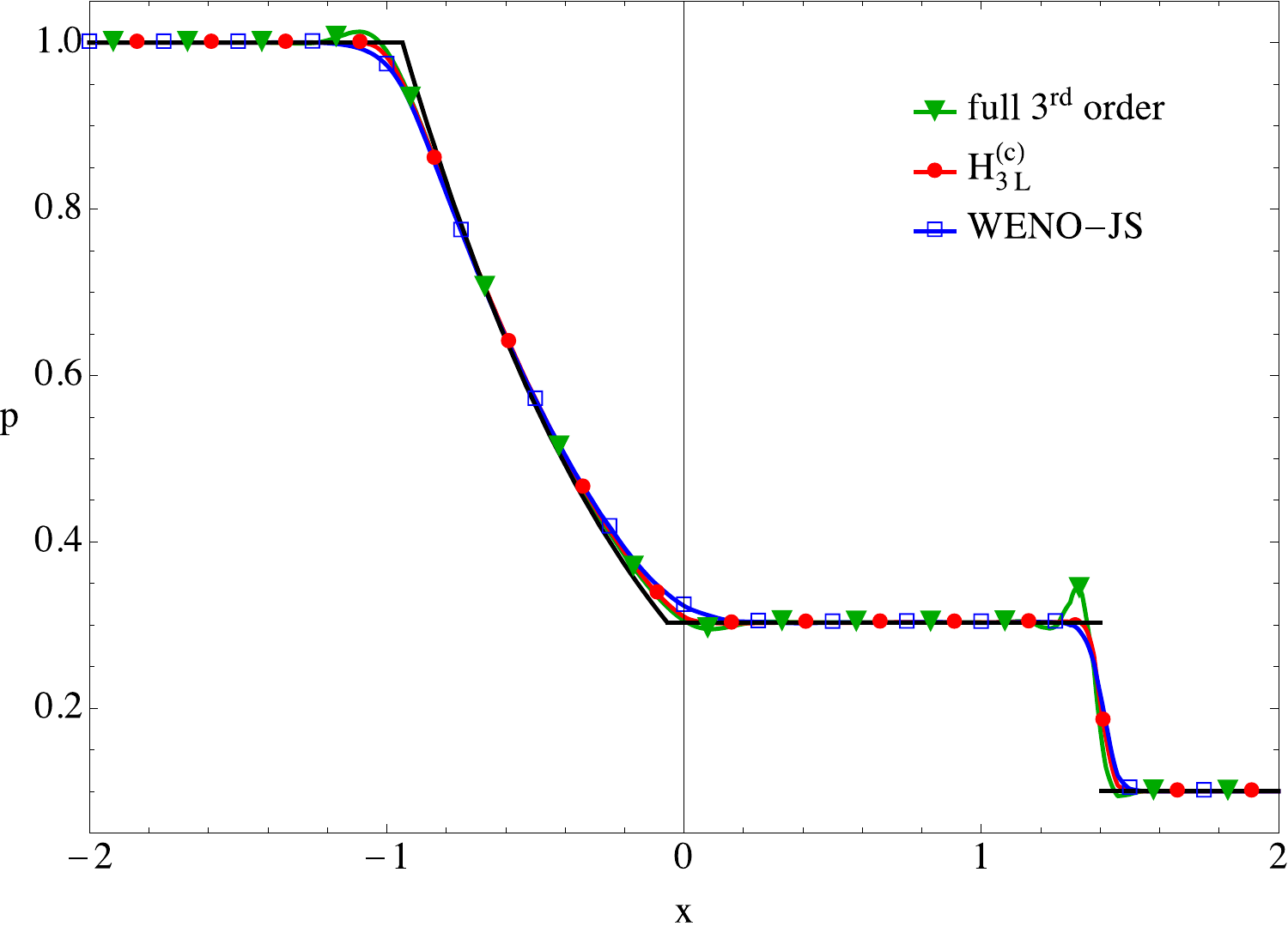}   		
   	\caption{Pressure profile.}
		\label{fig:SodPbP}
   	\end{subfigure}
   	\caption{Solution of different reconstruction techniques for Sod's Problem, Eq.~\eqref{eq:SodPbIC}, \eqref{eq:EulerEq} on the domain $[-2, 2]$ with $N=100$ grid cells, CFL $\nu=0.95$ at $T_\text{end}=0.8$.}
   	\label{fig:SodPbSol}
\end{figure}
Due to the purely discontinuous form of the initial condition, we obtain $\alpha=0$, just as in Sec.~\ref{subsec:discIC}. This means, that $H_{3\text{L}}^{(c)} = H_{3\text{L}}$, i.e. no asymptotic region exists in this test case. The reconstruction techniques have all been applied to the primitive variables $\rho, u$ and $p$.

Sod's shock tube problem leads to three characteristic waves, which can be seen in the solution, Fig.~\ref{fig:SodPbSol}. Both, the density and pressure profile show the rarefaction wave and the shock. The contact discontinuity can only be seen in the density profile, Fig.~\ref{fig:SodPbRho}, not in the pressure distribution, Fig.~\ref{fig:SodPbP}. The solution shows that applying the full-$3^\text{rd}$ order reconstruction leads to over- and undershoots close to the discontinuities. WENO-JS yields good results concerning this feature, however, does not approximate the true solution as accurate as $H_{3\text{L}}^{(c)}$, see also discussion in Sec.~\ref{subsec:discIC}.
%
\subsection{Shu-Osher Problem}\label{subsec:ShuOsher}
%
In this problem, originally introduced by Shu and Osher \cite{ShuOsher1989}, a Mach 3 shock is interacting with sine waves in the density profile. The computational domain is fixed to $[-4.5, 4.5]$ and the shock at time $t=0$ is situated at $x=-4$. The initial conditions of the primitive variables density, velocity and pressure, to the left and right of $x=-4$ are given by
\begin{figure}[H]
	\centering
	\includegraphics[width=0.6\textwidth]{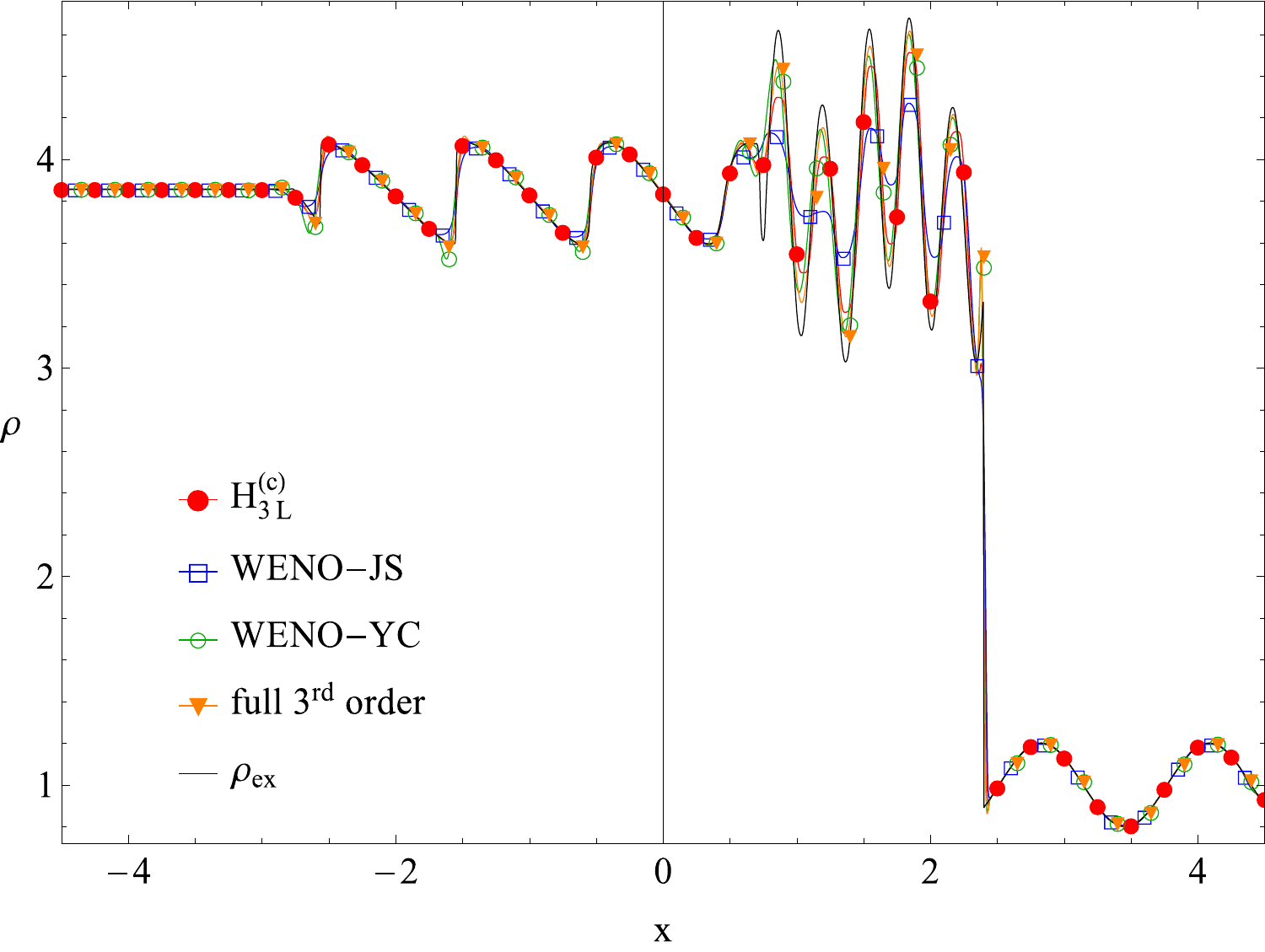}
	\caption{Solution of different reconstruction techniques for the Shu-Osher Problem, Eq.~\eqref{eq:ShuOsherPbIC}, \eqref{eq:EulerEq} with $N=640$ grid cells on the domain $[-4.5, 4.5]$, CFL $\nu=0.95$ at time $T_\text{end}=1.8$.}
	\label{fig:ShuOsherRho640}
\end{figure}
\begin{align}
	\begin{pmatrix}
	\rho_L \\
	v_L\\
	p_L
	\end{pmatrix}
	=
	\begin{pmatrix}
	3.857143\\2.629369\\10.33333
	\end{pmatrix},
	\quad 
	\begin{pmatrix}
	\rho_R \\
	v_R\\
	p_R
	\end{pmatrix}
	=
	\begin{pmatrix}
	1 + 0.2 \sin(5x), \\ 0.0\\1.0
	\end{pmatrix},
	\label{eq:ShuOsherPbIC}
\end{align}
and the time evolution is governed by the one-dimensional Euler equations, Eq.~\eqref{eq:EulerEq}. Just as in \cite{ShuOsher1989}, the solution is computed at $T_\text{end}=1.8$. We compare the solutions of the full $3^{\text{rd}}$-order reconstruction, WENO-JS, WENO-YC (where $\varepsilon=21.932$), and the new limiter function $H_{3\text{L}}^{(c)}$ with $N=640$ and $N=1280$ cells to a reference solution, which is the numerical solution of WENO-JS with $10,000$ grid cells. The CFL number is set to $\nu=0.95$ in all tests.

The combined limiter function $H_{3\text{L}}^{(c)}$ includes an asymptotic region, cf. \eqref{eq:eta}, with $\alpha=5.0$ the reconstruction techniques have all been applied to the primitive variables $\rho, u$ and $p$.\\
Fig.~\ref{fig:ShuOsherRho640} shows a comparison of the solutions obtained on $640$ grid cells of the full $3^{\text{rd}}$-order reconstruction, WENO-JS, WENO-YC, and the new limiter function $H_{3\text{L}}^{(c)}$. A zoom of the areas of interest of the solutions computed on $1280$ grid cells is shown in Fig.~\ref{fig:ShuOsherRho1280}. 
\begin{figure}[t]
   	\centering
%
   	\begin{subfigure}{0.49\textwidth}
   		\centering
   		\includegraphics[width=\textwidth]{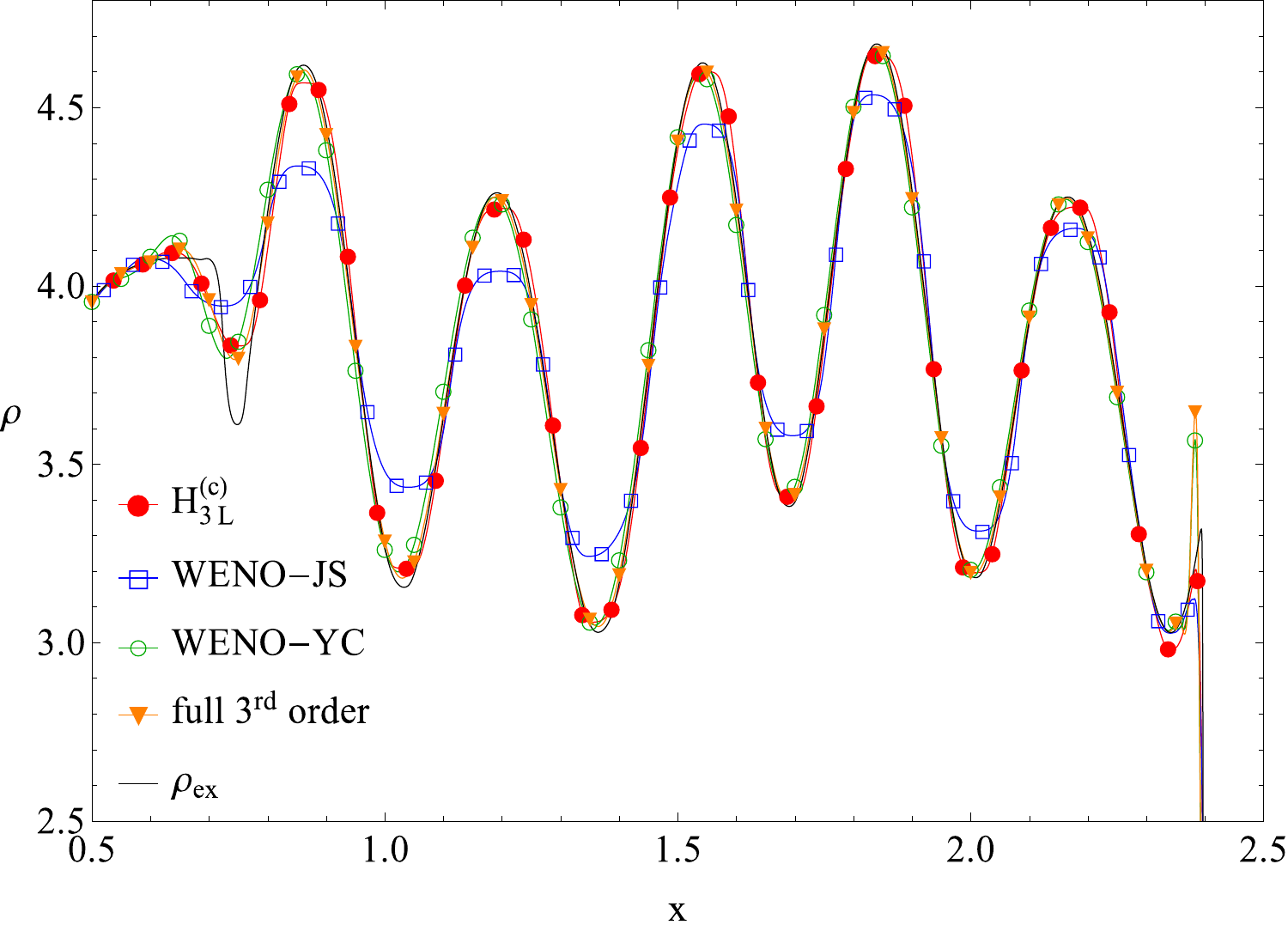}
   	\end{subfigure}
   	\hfill
   	\begin{subfigure}{0.49\textwidth}
   		\centering
   		\includegraphics[width=\textwidth]{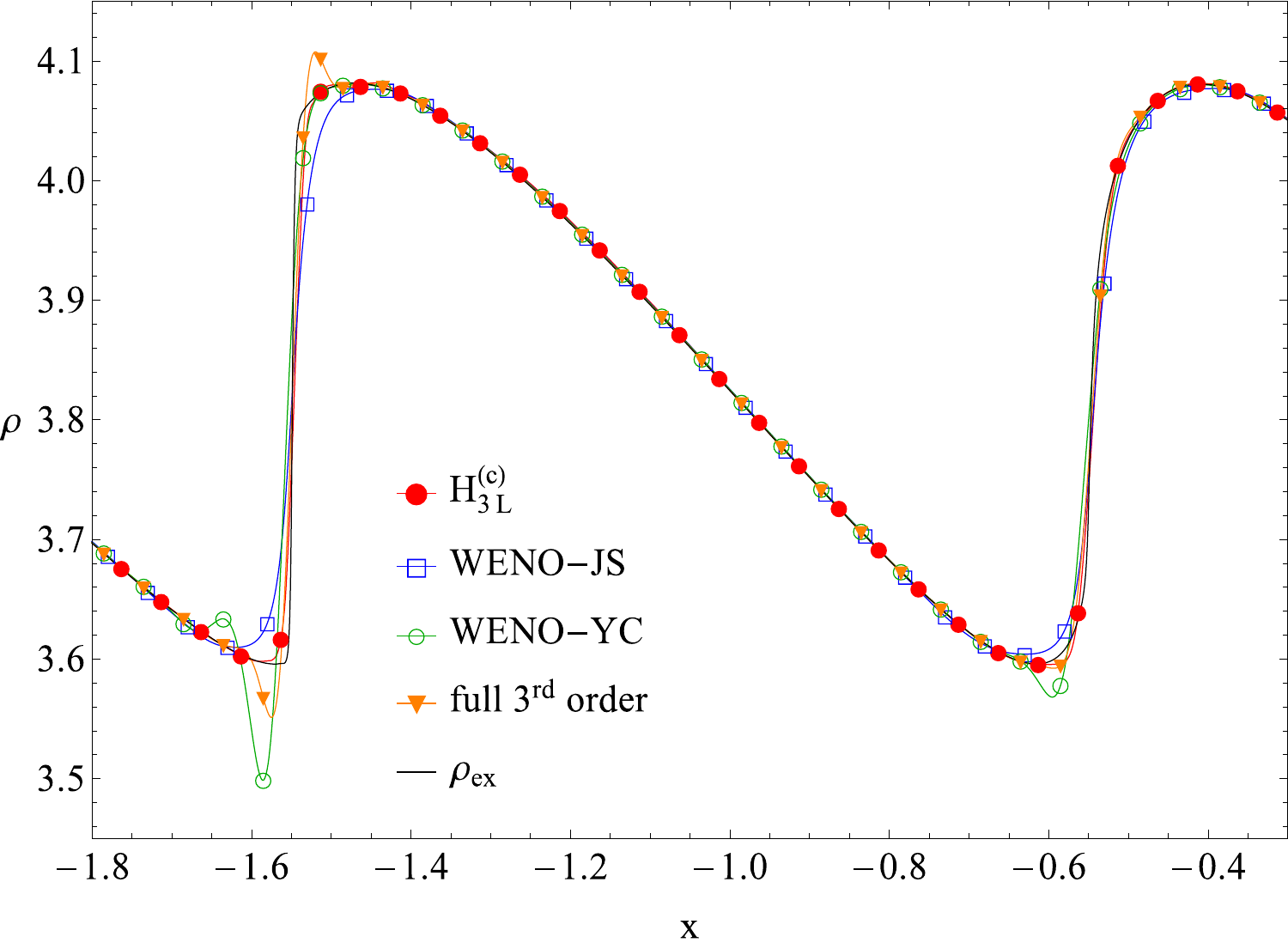}
   	\end{subfigure}
	\caption{Solution of different reconstruction techniques for the Shu-Osher Problem, Eq.~\eqref{eq:ShuOsherPbIC}, \eqref{eq:EulerEq}with $N=1280$ grid cells on the domain $[-4.5, 4.5]$, CFL $\nu=0.95$ at time $T_\text{end}=1.8$.}
   	\label{fig:ShuOsherRho1280}
\end{figure}
Overall, it can be seen that applying the full-$3^\text{rd}$ order reconstruction or WENO-YC leads to under- and overshoots close to discontinuities. WENO-JS does not produce overshoots, however, it limits too much, so that the reference solution is not approximated as well as by the limiter function $H_{3\text{L}}^{(c)}$. This is especially visible in regions with large gradients.
%
\section{Conclusions}\label{sec:conclusions}
%
In this paper we have analyzed $3^\text{rd}$-order finite volume and WENO schemes. These schemes reconstruct the solution at cell interfaces, each reconstruction based on only three mean values. We have then placed the different schemes in a unifying context. More specifically, we have analyzed and improved the FV limiter by {\v{C}}ada and Torrilhon \cite{CadaTorrilhon2009}. Our proposed FV limiter does not contain artificial parameters anymore. For a unifying view, we considered the $3^\text{rd}$-order WENO schemes proposed by Jiang and Shu \cite{JiangShu1996} and Yamaleev and Carpenter \cite{YamaleevCarpenter2009}. Plotting all schemes in the $(\delta_{i-\frac{1}{2}}, \delta_{i+\frac{1}{2}})$-plane revealed certain similarities which could also be observed in the numerical experiments. For smooth solutions, the FV limiter and WENO-YC show equally good results. However, near discontinuities, the numerical experiments showed some oscillations using the WENO-YC scheme, as predicted by Arandiga et al. \cite{ArandigaMarti2014}. This effect is avoided using our proposed FV limiter function, which reduces oscillations near discontinuities significantly.

\bibliographystyle{plain}
\bibliography{./references_complete}
\smallskip

\end{document}